\documentclass[12pt]{amsart}
\usepackage{a4wide}
\usepackage{amssymb}
\usepackage{amsmath}
\usepackage{amsfonts}
\usepackage{amsthm}
\usepackage{graphicx}
\usepackage{epsfig}
\usepackage{longtable}
\usepackage{paralist}
\usepackage[usenames,dvipsnames]{color}
\usepackage{hyperref}
\usepackage{caption}
\usepackage[labelformat=simple,labelfont={}]{subfig}

\newtheorem{thm}{Theorem}%[section]
\newtheorem{lem}[thm]{Lemma}
\newtheorem{cor}[thm]{Corollary}
\newtheorem{prop}[thm]{Proposition}

\theoremstyle{definition}

\theoremstyle{remark}
\newtheorem*{rmk}{Remark}

\newcommand{\eps}{\varepsilon}

\newcommand{\DEF}{{:=}}
\newcommand{\FED}{{=:}}

\newcommand{\PT}[1]{\mathbf{#1}}

\newcommand{\re}{\mathop{\mathrm{Re}}}

\DeclareMathOperator{\dd}{\mathrm{d}}

\DeclareMathOperator{\numint}{Q}

\DeclareMathOperator{\betafcn}{B}

\DeclareMathOperator{\gammafcn}{\Gamma}

\DeclareMathOperator{\GegenbauerC}{\mathrm{C}}

\DeclareMathOperator{\digammafcn}{\psi}

\DeclareMathOperator{\harmY}{Y}

\DeclareMathOperator{\phase}{ph}

\DeclareMathOperator{\wce}{wce}

\DeclareMathOperator{\HyperF}{F}
\DeclareMathOperator{\HyperTildeF}{\widetilde{F}}
\newcommand{\Hypergeom}[5]{{\sideset{_#1}{_#2}\HyperF\!\left(\substack{\displaystyle#3\\\displaystyle#4};#5\right)}}

\newcommand{\HypergeomReg}[5]{{\sideset{_#1}{_#2}\HyperTildeF\!\left(\substack{\displaystyle#3\\\displaystyle#4};#5\right)}}

\newcommand{\AppellFtwo}[3]{{\HyperF_{2}\!\left(\substack{\displaystyle#1\\\displaystyle#2};#3\right)}}

\newcommand{\AppellFfour}[3]{{\HyperF_{4}\!\left(\substack{\displaystyle#1\\\displaystyle#2};#3\right)}}

\newcommand{\KampedeFerietA}[6]{{\HyperF_{{#1}}^{{#2}}\!\left[\left. \substack{#3} \right| \left. \substack{#4} \right| \left. \substack{#5} \right|#6\right]}}

\newcommand{\Pochhsymb}[2]{{\left(#1\right)_{#2}}}

\allowdisplaybreaks[1]

\title[Sobolev spaces on the sphere and Stolarsky's invariance principle]{A characterization of Sobolev spaces on the sphere and an extension of Stolarsky's invariance principle to arbitrary smoothness}
\author[J. S. Brauchart and J. Dick]{Johann S. Brauchart\textasteriskcentered{} and Josef Dick\textdagger{}} %\textasteriskcentered
\thanks{\noindent \textasteriskcentered The research of this author was supported by an Australian Research Council Discovery Project. \\ \textdagger The research of this author was supported by an Australian Research Council Queen Elizabeth 2 Fellowship.}

\date{\today}

\DeclareCaptionLabelFormat{andtable}{#1˜\ #2 and \tablename˜\ \thetable}

\begin{document}

\address{J. S. Brauchart and J. Dick:
School of Mathematics and Statistics, 
University of New South Wales,
Sydney, NSW, 2052,
Australia }
\email{j.brauchart@unsw.edu.au, josef.dick@unsw.edu.au}

\begin{abstract}
In this paper we study reproducing kernel Hilbert spaces of arbitrary smoothness on the sphere $\mathbb{S}^d \subset \mathbb{R}^{d+1}$. The reproducing kernel is given by an integral representation using the truncated power function $(\PT{x} \cdot \PT{z} - t)_+^{\beta-1}$ defined on spherical caps centered at $\PT{z}$ of height $t$, which reduce to an integral over indicator functions of spherical caps as studied in [J. Brauchart, J. Dick, arXiv:1101.4448v1 [math.NA], to appear in Proc. Amer. Math. Soc.] for $\beta = 1$. This is in analogy to the generalization of the reproducing kernel to arbitrary smoothness on the unit cube. 

We show that the reproducing kernel is a sum of a Kamp{\'e} de F{\'e}riet function and the Euclidean distance $\|\PT{x}-\PT{y}\|$ of the arguments of the kernel raised to the power of $2\beta -1$ if $2\beta - 1$ is not an even integer; otherwise the logarithm of the distance $\|\PT{x}-\PT{y}\|$ appears. For $\beta \in \mathbb{N}$ the Kamp{\'e} de F{\'e}riet function reduces to a polynomial, giving a simple closed form expression for the reproducing kernel.

Using this space we can generalize Stolarsky's invariance principle to arbitrary smoothness. Previously, Warnock's formula, which is the analogue to Stolarsky's invariance principle for the unit cube $[0,1]^s$, has been generalized using similar techniques [J. Dick, Ann. Mat. Pura. Appl., (4) 187 (2008), no. 3, 385--403].
\end{abstract}


\maketitle

% \nocite{AbSt1992}

\section{Introduction and statement of results}

Let $\mathbb{S}^d = \{\PT{x} \in \mathbb{R}^{d+1}: \PT{x} \cdot \PT{x} = 1 \}$ be the unit sphere in $\mathbb{R}^{d+1}$ provided with the normalized Lebesgue surface area measure $\sigma_d$; that is $\int_{\mathbb{S}^d} \dd \sigma_d = 1$. The surface area of $\mathbb{S}^d$ is denoted by $\omega_d$. For future references we define the constant
\begin{equation} \label{eq:C.d}
C_d \DEF \frac{1}{d} \frac{\omega_{d-1}}{\omega_d} = \frac{1}{d} \frac{\gammafcn((d+1)/2)}{\sqrt{\pi} \, \gammafcn(d/2)}.
\end{equation}
A {\em spherical cap} centred at $\PT{z} \in \mathbb{S}^d$ with ``height'' $t \in [-1,1]$ is the set 
\begin{equation*}
C( \PT{z}; t ) \DEF \left\{ \PT{x} \in \mathbb{S}^d: \PT{x} \cdot \PT{z} \geq t \right\}.
\end{equation*}

Recently, we \cite{BrDi2011arXiv} used the kernel function
\begin{equation}\label{eq_kernel_sphere}
K_{\mathcal{C}}( \PT{x}, \PT{y} ) = \int_{-1}^1 \int_{\mathbb{S}^d} 1_{C( \PT{x}, t)}( \PT{z} ) 1_{C( \PT{y}, t)}( \PT{z} ) \dd \sigma_d( \PT{z} ) \dd t, \qquad \PT{x}, \PT{y} \in \mathbb{S}^d,
\end{equation}
where $1_A$ is the {\em indicator function} of the set $A \subseteq \mathbb{R}^{d+1}$, to reprove {\em Stolarsky's invariance principle} (cf. K. Stolarsky~\cite{St1973}) 
\begin{equation}
\begin{split} \label{eq:Stolarsky.s.invariance.principle}
&\frac{1}{N^2} \sum_{j=1}^N \sum_{k=1}^N \left\| \PT{x}_j - \PT{x}_k \right\| + \frac{1}{C_d} \int_{-1}^1 \int_{\mathbb{S}^d} \left| \frac{1}{N} \sum_{j=1}^N 1_{C(\PT{z};t)}(\PT{x}_k) - \sigma_d( C( \PT{z}; t ) ) \right|^2 \dd \sigma_d( \PT{z} ) \, \dd t \\
&\phantom{equals}= \int_{\mathbb{S}^d} \int_{\mathbb{S}^d} \left\| \PT{x} - \PT{y} \right\| \dd \sigma_d( \PT{x} ) \dd \sigma_d( \PT{y} ) 
\end{split}
\end{equation}
using reproducing kernel Hilbert space techniques. In this way we provided the theoretical underpinning for the particular role this principle plays in the theory of worst-case error of numerical integration rules for functions in a Sobolev space of index $(d+1)/2$ over the sphere $\mathbb{S}^d$ provided with the reproducing kernel $K_{\mathcal{C}}( \PT{x}, \PT{y} )$ as observed and utilized in \cite{BrWo2010d_pre}. This kernel has the very simple closed form (see \cite{BrDi2011arXiv})
\begin{equation} \label{eq:cal.C.kernel}
K_{\mathcal{C}}( \PT{x}, \PT{y} ) = 1 - C_d \left\| \PT{x} - \PT{y} \right\|, \qquad \PT{x}, \PT{y} \in \mathbb{S}^d.
\end{equation}
A generalization of the corresponding reproducing kernel Hilbert space to different smoothness classes has been studied in \cite[Section~6]{BrWo2010d_pre}. This is done via the decay of the Fourier coefficients in the series expansion using spherical harmonics of the reproducing kernel. However, no extension of Stolarsky's invariance principle for this generalization is known.

On the other hand, Stolarsky's invariance principle has an analogue for discrepancy defined on the high-dimensional unit cube $[0,1]^s$ with respect to rectangles anchored at the origin $\PT{0}$. This analogue is called Warnock's formula~\cite{warnock}, see also \cite[Chapter~2]{DP10}. The reproducing kernel for the unit cube is defined analogously to \eqref{eq_kernel_sphere}, where the spherical caps are replaced by rectangles anchored at $\PT{0}$. The reproducing kernel Hilbert space which corresponds to this kernel consists of functions with square integrable partial mixed derivatives of order $1$, see \cite[Chapter~2]{DP10} for details. In \cite{D08} this reproducing kernel was generalized to include functions of fractional smoothness. This generalization was achieved by replacing the indicator function $1_{[0, z)}(x)$ by $(x- z)_+^{\beta-1}$ for $\beta > 1/2$. The truncated power function $(x - z)_+^{\beta-1}$ arises from the Taylor expansion of a $\beta$-times continuously differentiable function $f:[0,1]\to\mathbb{R}$ given by
\begin{equation}\label{eq_Taylor}
f(x) = \sum_{r=0}^{\beta-1} \frac{f^{(r)}(0)}{r!} x^r + \int_0^1 \frac{f^{(\beta)}(t)}{(\beta-1)!} (x-t)_+^{\beta-1} \,\mathrm{d} t.
\end{equation}
The Taylor series can be written in a concise form using a reproducing kernel Hilbert space. Namely, if we define a reproducing kernel
\begin{equation*}
K(x,y) = \sum_{r=0}^{\beta-1} x^r y^r + \int_0^1 (x-t)_+^{\beta-1} (y-t)_+^{\beta-1} \,\mathrm{d} t
\end{equation*}
and the corresponding inner product
\begin{equation*}
\langle f, g \rangle = \sum_{r=0}^{\beta-1} \frac{f^{(r)}(0)}{r!} \frac{g^{(r)}(0)}{r!} +  \int_0^1 \frac{f^{(\beta)}(t)}{(\beta-1)!} \frac{g^{(\beta)}(t)}{(\beta-1)!} \,\mathrm{d} t,
\end{equation*}
then \eqref{eq_Taylor} can be written as
\begin{equation*}
f(x) = \langle f, K(\cdot, x) \rangle.
\end{equation*}
In \cite{D08} it was shown that this approach yields a generalization of the geometric discrepancy to a fractional geometric discrepancy.

The aim of this paper is to obtain analogue results for the sphere $\mathbb{S}^d$ as have been shown in the cube case. In particular we prove a generalization of Stolarsky's invariance principle for function classes of fractional smoothness considered in \cite[Section~6]{BrWo2010d_pre}. This case is not covered by the results in \cite{St1973}.

As in the cube case, the identities
\begin{equation} \label{eq:generalizing.indicator.fcn}
1_{C( \PT{z}; t )}( \PT{x} ) = 1_{[t, 1]}( \PT{x} \cdot \PT{z} ) = \left( \PT{x} \cdot \PT{z} - t \right)_+^0, \qquad \PT{x}, \PT{z} \in \mathbb{S}^d, -1 \leq t \leq 1,
\end{equation}
where $x_+^\alpha$ denotes the {\em truncated power function} motivate the following generalization. Let $\beta > 1/2$. For $\PT{x}$, $\PT{y}$ in $\mathbb{S}^d$ we define the function $\mathcal{K}_\beta : \mathbb{S}^d \times \mathbb{S}^d \to \mathbb{R}$ by means of
\begin{equation} \label{eq:kernel}
\mathcal{K}_\beta( \PT{x}, \PT{y} ) \DEF \int_{-1}^1 \int_{\mathbb{S}^d} \left( \PT{x} \cdot \PT{z} - t \right)_+^{\beta - 1} \left( \PT{y} \cdot \PT{z} - t \right)_+^{\beta - 1} \dd \sigma_d( \PT{z} ) \dd t.
\end{equation}
It can be checked that $\mathcal{K}_\beta(\PT{x}, \PT{y} )  < \infty$ for all $\PT{x}, \PT{y} \in \mathbb{S}^d$ and all $\beta > 1/2$. Note that for $\beta \le 1/2$ this is not true, since $\mathcal{K}_\beta(\PT{x}, \PT{y}) > \mathcal{K}_{\beta'}(\PT{x}, \PT{y})$ for $\beta < \beta'$ and $\mathcal{K}_{1/2}(\PT{x}, \PT{x}) = \infty$.

The function $\mathcal{K}_\beta$ is obviously symmetric, that is, $\mathcal{K}_\beta( \PT{x}, \PT{y} ) = \mathcal{K}_\beta( \PT{y}, \PT{x} )$. Further, let arbitrary $a_1, \dots, a_N \in \mathbb{C}$ and $\PT{x}_1, \dots, \PT{x}_N \in \mathbb{S}^d$ be given. Then we have
\begin{align*}
\sum_{j=1}^N \sum_{k=1}^N a_j \overline{a_k} \, \mathcal{K}_\beta( \PT{x}_j, \PT{x}_k ) 
&= \int_{-1}^1 \int_{\mathbb{S}^d} \sum_{j=1}^N \sum_{k=1}^N a_j \overline{a_k} \left( t - \PT{x}_j \cdot \PT{z} \right)_+^{\beta - 1} \left( t - \PT{x}_k \cdot \PT{z} \right)_+^{\beta - 1} \dd \sigma_d( \PT{z} ) \dd t \\
&= \int_{-1}^1 \int_{\mathbb{S}^d} \left| \sum_{k=1}^N a_j \left( t - \PT{x}_j \cdot \PT{z} \right)_+^{\beta - 1} \right|^2  \dd \sigma_d( \PT{z} ) \dd t \geq 0.
\end{align*}

Thus, the function $\mathcal{K}_\beta( \PT{x}, \PT{y} )$, is also positive definite. From \cite{Ar1950} it follows that $\mathcal{K}_\beta( \PT{x}, \PT{y} )$ is a reproducing kernel for each $\beta > 1/2$ that uniquely defines a reproducing kernel Hilbert space $\mathcal{H}_\beta( \mathcal{K}_\beta, \mathbb{S}^d)$ of functions endowed with a certain inner product which we denote by $(\cdot, \cdot)_{\mathcal{K}_\beta}$.

Let $\mathcal{H}_\beta \DEF \mathcal{H}_\beta( \mathcal{K}_\beta, \mathbb{S}^d)$. We note that the reproducing kernel Hilbert space $\mathcal{H}_\beta$ consists of all functions $f_1, f_2: \mathbb{S}^d \to \mathbb{R}$ which permit the integral representation
\begin{equation} \label{eq:integral.representation.f.i}
f_i( \PT{x} ) = \int_{-1}^1 \int_{\mathbb{S}^d} g_i( \PT{z}; t) \left( \PT{x} \cdot \PT{z} - t \right)_+^{\beta - 1} \, \dd \sigma_d( \PT{z} ) \dd t, \qquad i = 1, 2,
\end{equation}
for some functions $g_1, g_2: \mathbb{S}^d \times [-1,1] \to \mathbb{R}$ with $g_1, g_2 \in \mathbb{L}_2( \mathbb{S}^d \times [-1,1], \sigma_d )$ of square-integrable functions on $\mathbb{S}^d \times [-1,1]$. See Appendix~\ref{sec_appendix_a} for a proof of this statement. The function $g_i$ is called the potential function of $f_i$. Thus the inner product for the functions $f_1, f_2 \in \mathcal{H}_\beta$ is given by
\begin{equation} \label{eq:inner.product.1}
( f_1, f_2 )_{\mathcal{K}_{\beta}} = \int_{-1}^1 \int_{\mathbb{S}^d} g_1( \PT{z}; t) g_2( \PT{z}; t) \, \dd \sigma_d( \PT{z} ) \dd t.
\end{equation}
For a function $f \in \mathcal{H}_\beta$ one has
\begin{equation*}
\| f \|_{\mathcal{K}_\beta} = \left\{ \int_{-1}^1 \int_{\mathbb{S}^d} \left| g( \PT{z}; t) \right|^2 \, \dd \sigma_d( \PT{z} ) \dd t \right\}^{1/2}.
\end{equation*}
This follows from the uniqueness of the inner product $( \cdot, \cdot )_{\mathcal{K}_\beta}$ induced by the kernel $\mathcal{K}_\beta$.

In the following subsections we describe the results in more details. The proofs can be found in Section~\ref{sec:proofs} and the Appendix.

\subsection{The smoothness of the Hilbert space $\mathcal{H}_\beta( \mathcal{K}_\beta, \mathbb{S}^d)$ and the Sobolev space $\mathbb{H}^s(\mathbb{S}^d)$}

The Hilbert space $\mathbb{L}_2(\mathbb{S}^{d}, \sigma_d)$ of square-integrable real-valued functions on $\mathbb{S}^d$ is endowed with the inner product and induced norm
\begin{equation*}
( f, g )_{\mathbb{L}_2(\mathbb{S}^{d})} \DEF \int_{\mathbb{S}^d} f( \PT{x} ) g( \PT{x} ) \dd \sigma_d( \PT{x} ), \qquad \| f \|_{\mathbb{L}_2(\mathbb{S}^{d})} \DEF \sqrt{( f, f )_{\mathbb{L}_2(\mathbb{S}^{d})}}.
\end{equation*}

In the following we define the Hilbert space $\mathbb{H}^s(\mathbb{S}^d)$. 
Normalized Gegenbauer polynomials $P_n^{(d)}(t) = \GegenbauerC_n^{(d-1)/2}(t) / \GegenbauerC_k^{(d-1)/2}(1)$ are orthogonal on the interval $[-1,1]$ with respect to the weight function $(1 - t^2)^{d/2-1}$ and, thus, related to the unit sphere $\mathbb{S}^d$ as can be seen from the {\em Addition theorem for spherical harmonics}
\begin{equation*}
\sum_{k=1}^{Z(d,n)} \harmY_{n,k}( \PT{x} ) \harmY_{n,k}( \PT{y} ) = Z(d,n) \, P_n^{(d)}(\PT{x} \cdot \PT{y}), \qquad \PT{x}, \PT{y} \in \mathbb{S}^d.
\end{equation*}
The real spherical harmonics $\harmY_{n,k}$ form a complete orthonormal basis 
\begin{equation*}
\left\{ \harmY_{n,k}^{(d)} : k = 1, \dots, Z(d,n); n = 0, 1, \dots \right\}
\end{equation*}
of the Hilbert space $\mathbb{L}_2(\mathbb{S}^{d}, \sigma_d)$ of square-integrable functions on $\mathbb{S}^d$. Here
\begin{equation*} %\label{eq:Z.d.n}
Z(d,n) = \left( 2n + d - 1 \right) \frac{\gammafcn(n+d-1)}{\gammafcn(d)\,\gammafcn(n+1)} \sim \frac{2}{\gammafcn(d)} \, n^{d-1} \quad \text{as $n \to \infty$.}
\end{equation*}

Let $s > 0$ and let $\lambda'_0, \lambda'_1, \dots$ be a sequence of positive real numbers such that
\begin{equation} \label{eq:assumption}
\lambda'_{n} \asymp n^{-2s} \qquad \text{as $n \to \infty$.}
\end{equation}
The Sobolev space $\mathbb{H}^s(\mathbb{S}^d)$ over $\mathbb{S}^d$ with index $s$ can be obtained as the completion of the family $C^\infty(\mathbb{S}^d)$ of infinitely differentiable functions on $\mathbb{S}^d$ with respect to the norm
\begin{equation*}
\| f \|_{d,s} \DEF \sqrt{( f, f )_{d,s}}
\end{equation*}
induced by the inner product
\begin{equation}\label{hs_inner_product}
( f, g )_{d,s} \DEF \sum_{n=0}^\infty \sum_{k=1}^{Z(d,n)} (\lambda'_n)^{-1} \hat{f}_{n,k}^{(d)} \, \hat{g}_{n,k}^{(d)},
\end{equation}
where the Laplace-Fourier coefficients are given by
\begin{equation} \label{eq:L-F.coeff}
\hat{f}_{n,k}^{(d)} = ( f, \harmY_{n,k}^{(d)} )_{\mathbb{L}_2(\mathbb{S}^{d})} = \int_{\mathbb{S}^d} f( \PT{x} ) \harmY_{n,k}^{(d)}( \PT{x} ) \dd \sigma_d( \PT{x} ).
\end{equation}
Obviously any norm equivalent to $\|\cdot \|_{d,s}$ will produce the same space of functions. 

For $s>d/2$ the Sobolev space $\mathbb{H}^s(\mathbb{S}^d)$ is a reproducing kernel Hilbert space. If we endow $\mathbb{H}^s(\mathbb{S}^d)$ with the inner product $(\cdot, \cdot)_{d,s}$ given by \eqref{hs_inner_product}, then the corresponding reproducing kernel is given by
\begin{equation*}
K_{d,s}( \PT{x}, \PT{y} ) = \sum_{n=0}^\infty \sum_{k=1}^{Z(d,n)} \lambda'_n \harmY_{n,k}( \PT{x} ) \harmY_{n,k}( \PT{y} ) = \sum_{n=0}^\infty \lambda'_n \, Z(d,n) \, P_n^{(d)}( \PT{x} \cdot \PT{y} ), \qquad \PT{x}, \PT{y} \in \mathbb{S}^d.
\end{equation*}
The reproducing kernel $K_{d,s}$ satisfies the reproducing kernel properties
\begin{equation*}
K_{d,s}({\PT \cdot},{\PT x}) \in \mathbb{H}^s(\mathbb{S}^d) \quad \mbox{ and } \quad ( f, K_{d,s}({\PT \cdot},{\PT x}) )_{d,s} = f({\PT x}).
\end{equation*}
This is justified by the fact that point-evaluation is bounded in $\mathbb{H}^s(\mathbb{S}^d)$ for $s>d/2$ by  \eqref{eq:assumption}. Imbedding theorems relate Sobolev spaces $\mathbb{H}^s(\mathbb{S}^d)$ and classical function spaces $C^{k}(\mathbb{S}^{d})$: $\mathbb{H}^s(\mathbb{S}^d)$ is a subspace of $C^{k}(\mathbb{S}^{d})$ if $s>k+d/2$.

For further details we refer the interested reader to \cite{BrWo2010d_pre}. Therein, in Section~6, a particular instance of a reproducing kernel Hilbert space with inner product and reproducing kernel as above has been studied. The reproducing kernel studied there has the property that it can easily be evaluated at given function values, since it can be written as a polynomial of $\PT{x} \cdot \PT{y}$. However, no extension of Stolarsky's invariance principle seems possible for this kernel function.

The following theorem shows that the reproducing kernel Hilbert space $\mathcal{H}_\beta$ corresponding to the reproducing kernel $\mathcal{K}_\beta$ coincides with $\mathbb{H}^s(\mathbb{S}^d)$.

\begin{thm} \label{thm:identification.thm}
Let $d \geq 2$ and suppose that $\beta > 1/2$. Then the reproducing kernel $\mathcal{K}_\beta$ has the representation
\begin{equation} \label{eq:identification.kernel.expansion}
\mathcal{K}_\beta( \PT{x}, \PT{y} ) = \sum_{n=0}^\infty \lambda_n \, Z(d, n) \, P_n^{(d)}(\PT{x} \cdot \PT{y}),
\end{equation}
where
% \begin{equation*}
% \lambda_n \asymp n^{-2s} \quad \mbox{ as } n \to \infty
% \end{equation*}
% and
% \begin{equation*}
% s = \beta - \frac{1}{2} + \frac{d}{2}.
% \end{equation*}
\begin{equation*}
\lambda_n \asymp n^{-2s} \quad \text{as $n \to \infty$} \qquad \text{and} \quad s = \beta - \frac{1}{2} + \frac{d}{2}.
\end{equation*}
Thus, the reproducing kernel Hilbert space $\mathcal{H}_\beta( \mathcal{K}_\beta, \mathbb{S}^d)$ uniquely defined by the positive definite kernel \eqref{eq:kernel} coincides with the Sobolev space $\mathbb{H}^s(\mathbb{S}^d)$ with smoothness $s = \beta-1/2 + d/2$ and the corresponding norms $\|\cdot\|_{d,s}$ and $\|\cdot\|_{\mathcal{K}_\beta}$ are equivalent.
\end{thm}

Following \cite{BrDi2011arXiv}, we derive a generalization of Stolarsky's invariance principle in the following.

\subsection{The Worst-Case Error of Numerical Integration}

Let $d \geq 2$ and $\beta > 1/2$. In the following we assume that the weights $\lambda_0, \lambda_1, \dots$ are defined by the expansion \eqref{eq:identification.kernel.expansion}.

The {\em worst-case error} of the equal-weight numerical integration rule
\begin{equation*}
\numint_N[f] \DEF \frac{1}{N} \sum_{j=1}^{N} f({\PT x}_{j}), \qquad \PT{x}_1, \dots, \PT{x}_N \in \mathbb{S}^d, 
\end{equation*}
for functions $f$ in the unit ball in the Sobolev space $\mathcal{H}_\beta$ is defined as
\begin{equation*}
\wce( \mathcal{H}_\beta, \numint_N ) \DEF \sup \left\{ \left| \frac{1}{N} \sum_{j=1}^{N} f({\PT x}_{j}) - \int_{\mathbb{S}^d} f( \PT{x} ) \, \dd \sigma_d( \PT{x} ) \right| : f \in \mathcal{H}_\beta, \| f \|_{\mathcal{K}_\beta} \leq 1 \right\}.
\end{equation*}

Invoking the integral representation \eqref{eq:kernel} of the reproducing kernel $\mathcal{K}_\beta$ we have
\begin{equation*}
\frac{1}{N} \sum_{j=1}^N \mathcal{K}_\beta( \PT{x}, \PT{x}_j ) = \int_{-1}^1 \int_{\mathbb{S}^d} \left( \PT{x} \cdot \PT{z} - t \right)_+^{\beta - 1} \left[ \frac{1}{N} \sum_{j=1}^N \left( \PT{x}_j \cdot \PT{z} - t \right)_+^{\beta - 1} \right] \dd \sigma_d( \PT{z} ) \dd t
\end{equation*}
and 
\begin{equation*}
\int_{\mathbb{S}^d} \mathcal{K}_\beta( \PT{x}, \PT{y} ) \, \dd \sigma_d( \PT{y} ) = \int_{-1}^1 \int_{\mathbb{S}^d} \left( \PT{x} \cdot \PT{z} - t \right)_+^{\beta - 1} \left[ \int_{\mathbb{S}^d}  \left( \PT{y} \cdot \PT{z} - t \right)_+^{\beta - 1} \, \dd \sigma_d( \PT{y} ) \right] \dd \sigma_d( \PT{z} ) \dd t.
\end{equation*}
Therefore, the ``representer'' of the error of numerical integration, %\eqref{eq:representer},
\begin{align*}
& \mathcal{R}( \mathcal{H}_\beta, \numint_N; \PT{x} )_{\mathcal{K}_\beta} \\ & = \int_{-1}^1 \int_{\mathbb{S}^d} \left( \PT{x} \cdot \PT{z} - t \right)_+^{\beta - 1} \left[ \frac{1}{N} \sum_{j=1}^N \left( \PT{x}_j \cdot \PT{z} - t \right)_+^{\beta - 1} - \int_{\mathbb{S}^d}  \left( \PT{y} \cdot \PT{z} - t \right)_+^{\beta - 1} \, \dd \sigma_d( \PT{y} ) \right] \dd \sigma_d( \PT{z} ) \dd t,
\end{align*}
is of the type of function given in \eqref{eq:integral.representation.f.i}. The function is called representer since
\begin{equation}\label{eq_representer}
\frac{1}{N} \sum_{j=1}^N f(\PT{x}_j) - \int_{\mathbb{S}^d} f(\PT{x}) \,\mathrm{d} \sigma_d(\PT{x}) = \langle f, \mathcal{R}( \mathcal{H}_\beta, \numint_N; \PT{x} )_{\mathcal{K}_\beta}  \rangle_{\mathcal{K}_\beta}.
\end{equation}

It is well known that for $f \in \mathcal{H}_\beta$ the following Koksma-Hlawka type inequality holds
\begin{equation*}
\left| \frac{1}{N} \sum_{j=1}^{N} f({\PT x}_{j}) - \int_{\mathbb{S}^d} f( \PT{x} ) \, \dd \sigma_d( \PT{x} ) \right| = \left| ( f, \mathcal{R}( \mathcal{H}_\beta, \numint_N; \cdot )_{\mathcal{K}_\beta} )_{\mathcal{K}_\beta} \right| \leq \left\| f \right\|_{\mathcal{K}_\beta} \left\| \mathcal{R}( \mathcal{H}_\beta, \numint_N; \cdot )_{\mathcal{K}_\beta} \right\|_{\mathcal{K}_\beta},
\end{equation*}
where equality is assumed for the ``representer'' of the error of numerical integration of the rule $\numint_N$ for functions in $\mathcal{H}_\beta$.
See \cite{Hi98} or \cite[Chapter~2]{DP10} for numerical integration in reproducing kernel Hilbert spaces.

It follows that
\begin{align}
\left[ \wce( \mathcal{H}_\beta, \numint_N ) \right]^2 
&= \left\| \mathcal{R}( \mathcal{H}_\beta, \numint_N; \cdot )_{\mathcal{K}_\beta} \right\|_{\mathcal{K}_\beta}^2 \label{eq:wce.id.1} \\
&= \int_{\mathbb{S}^d} \int_{\mathbb{S}^d} \mathcal{K}_\beta( \PT{x}, \PT{y} ) \, \dd \sigma_d( \PT{x} ) \dd \sigma_d( \PT{y} ) - \frac{2}{N} \sum_{j=1}^N \int_{\mathbb{S}^d} \mathcal{K}_\beta( \PT{x}, \PT{x}_j ) \, \dd \sigma_d( \PT{x} ) \notag \\
&\phantom{=}+ \frac{1}{N^2} \sum_{j=1}^N \sum_{k=1}^N \mathcal{K}_\beta( \PT{x}_j, \PT{x}_k ). \notag
\end{align}
Since the reproducing kernel $\mathcal{K}_\beta( \PT{x}, \PT{y} )$ is a zonal function (that is, only depends on $\PT{x} \cdot \PT{y}$), see Theorems~\ref{thm:integral.general.case} and \ref{thm:integral.exceptional.cases}, the integral $\int_{\mathbb{S}^d} \mathcal{K}_\beta( \PT{x}, \PT{y} ) \, \dd \sigma_d( \PT{x} )$ does not depend on $\PT{y} \in \mathbb{S}^d$. Hence
\begin{equation} \label{eq:wce-A}
\left[ \wce( \mathcal{H}_\beta, \numint_N ) \right]^2 = \frac{1}{N^2} \sum_{j=1}^N \sum_{k=1}^N \mathcal{K}_\beta( \PT{x}_j, \PT{x}_k ) - \int_{\mathbb{S}^d} \int_{\mathbb{S}^d} \mathcal{K}_\beta( \PT{x}, \PT{y} ) \, \dd \sigma_d( \PT{x} ) \dd \sigma_d( \PT{y} ).
\end{equation}

\subsection{Generalizing Stolarsky's Invariance Principle}

We define local discrepancy function with coefficient $\beta$ of the point set $\mathcal{P}_N = \{\PT{x}_1,\ldots, \PT{x}_N\}$ by
\begin{equation*}
\Delta_{\mathcal{P}_N, \beta}(\PT{z}, t) = \frac{1}{N} \sum_{j=1}^N (\PT{x}_j \cdot \PT{z} - t)_+^{\beta-1} - \int_{\mathbb{S}^d} (\PT{y} \cdot \PT{z} - t)_+^{\beta-1} \,\mathrm{d} \sigma_d(\PT{y}).
\end{equation*}
For $\beta = 1$ we have
\begin{equation*}
(\PT{x} \cdot \PT{z} - t)_+^0 = 1_{C(\PT{z};t)}(\PT{x}) = \left\{\begin{array}{ll} 1 & \mbox{if } \PT{x} \in C(\PT{z};t), \\ 0 & \mbox{otherwise}. \end{array} \right.
\end{equation*}
Thus, the local discrepancy with coefficient $\beta = 1$ is just the local spherical cap discrepancy.

Note that the integral appearing in the definition of the local discrepancy above does not depend on $\PT{z}$ and $t$ and can be expressed in terms of a Gauss hypergeometric function. (For $\beta = 1$ it is simply $\sigma_d( C( \PT{z}, t) )$ as follows from \eqref{eq:generalizing.indicator.fcn}.)
\begin{prop} \label{prop:inner-most.int}
Let $d \geq 2$ and $\beta > 0$. Then
\begin{equation*}
\int_{\mathbb{S}^d}  \left( \PT{y} \cdot \PT{z} - t \right)_+^{\beta - 1} \, \dd \sigma_d( \PT{y} ) = 2^{d/2-1} \frac{\gammafcn((d+1)/2) \gammafcn(\beta)}{\sqrt{\pi} \, \gammafcn(\beta + d/2)} \left( 1 - t \right)^{\beta + d/2 - 1} \Hypergeom{2}{1}{1-d/2,d/2}{\beta+d/2}{\frac{1-t}{2}}.
\end{equation*}
\end{prop}

Using the local discrepancy function, the representer of numerical integration can be written as
\begin{equation*}
\mathcal{R}( \mathcal{H}_\beta, \numint_N; \PT{x} )_{\mathcal{K}_\beta}  = \int_{-1}^1 \int_{\mathbb{S}^d} \left( \PT{x} \cdot \PT{z} - t \right)_+^{\beta - 1} \Delta_{\mathcal{P}_N, \beta}(\PT{z},t) \dd \sigma_d( \PT{z} ) \dd t.
\end{equation*}
Hence, using the inner product representation of the worst-case error given in \eqref{eq:wce.id.1} and the inner product formula in \eqref{eq:inner.product.1}, we arrive at
\begin{equation} \label{eq:wce-B}
\wce( \mathcal{H}_\beta, \numint_N )  = \left\|  \mathcal{R}( \mathcal{H}_\beta, \numint_N; \PT{x} )_{\mathcal{K}_\beta} \right\|_{\mathcal{K}_\beta} = \left\{ \int_{-1}^1 \int_{\mathbb{S}^d} \left| \Delta_{\mathcal{P}_N, \beta}(\PT{z},t) \right|^2 \dd \sigma_d( \PT{z} ) \dd t \right\}^{1/2}.
\end{equation}

The right-hand side of \eqref{eq:wce-B} gives rise to the definition of a generalized discrepancy which for $\beta = 1$ reduces to the {\em spherical cap $\mathbb{L}_2$-discrepancy} 
\begin{equation*}
D_{\mathbb{L}_{2}}^{C}(\mathcal{P}_{N}) \DEF \left\{\int_{-1}^{1} \int_{\mathbb{S}^{d}} \left| \frac{| \mathcal{P}_{N} \cap C({\PT x};t)|}{N} - \sigma_d( C({\PT x};t) ) \right|^{2} \dd{\sigma_d}({\PT x}) \,\dd{t} \right\}^{1/2}.
\end{equation*}
As for the spherical cap $\mathbb{L}_2$-discrepancy, one could also study the $L_p$-norm analogue
\begin{equation}\label{wce_pnorm}
\left\{ \int_{-1}^1 \int_{\mathbb{S}^d} |\Delta_{\mathcal{P}_N,\beta}(\PT{z},t)|^p \,\mathrm{d} \sigma_d(\PT{z}) \,\mathrm{d} t \right\}^{1/p},
\end{equation}
with the obvious modifications for $p = \infty$. Note that this expression is defined for all $1 \le p \le \infty$ if $\beta \ge 1$ and all $1 \le p < (1-\beta)^{-1}$ for $0 < \beta < 1$. Let $1 \le q \le \infty$ with $1/p + 1/q =1$. Equation~\eqref{wce_pnorm} is the worst-case integration error for functions with finite norm
\begin{equation*}
\|f\|_{\mathcal{K}_\beta, q} \DEF \left\{ \int_{-1}^1 \int_{\mathbb{S}^d} \left| g(\PT{z},t) \right|^q \,\mathrm{d} \sigma_d(\PT{z}) \,\mathrm{d} t \right\}^{1/q},
\end{equation*}
with the obvious modifications for $q = \infty$. Here, $g$ is the potential function of $f$. This result can be obtained by applying H\"older's inequality to \eqref{eq_representer}.

The combination of the relations \eqref{eq:wce-A} and \eqref{eq:wce-B} gives the generalized invariance principle.
\begin{thm} \label{thm:gen.Stolarsky}
Let $d \geq 2$ and $\beta > 1/2$. Then for any $N$-points $\PT{x}_1, \dots, \PT{x}_N$ on $\mathbb{S}^d$ one has
\begin{equation}
\begin{split} \label{eq:gen.Stolarsky}
&\int_{-1}^1 \int_{\mathbb{S}^d} \left| \frac{1}{N} \sum_{j=1}^N \left( \PT{x}_j \cdot \PT{z} - t \right)_+^{\beta - 1} - \int_{\mathbb{S}^d}  \left( \PT{y} \cdot \PT{z} - t \right)_+^{\beta - 1} \, \dd \sigma_d( \PT{y} ) \right|^2 \dd \sigma_d( \PT{z} ) \dd t \\ %= \left[ \wce( \mathcal{H}_\beta, \numint_N ) \right]^2 \\
&\phantom{equals}= \frac{1}{N^2} \sum_{j=1}^N \sum_{k=1}^N \mathcal{K}_\beta( \PT{x}_j, \PT{x}_k ) - \int_{\mathbb{S}^d} \int_{\mathbb{S}^d} \mathcal{K}_\beta( \PT{x}, \PT{y} ) \, \dd \sigma_d( \PT{x} ) \dd \sigma_d( \PT{y} ).
\end{split}
\end{equation}
\end{thm}

The double integral over the kernel $\mathcal{K}_\beta$ in \eqref{eq:gen.Stolarsky} can be reduced as follows. In the following let ${_p}\HyperF_q$ denote the generalized hypergeometric function and $(a)_n$ the Pochhammer symbol.

\begin{prop} \label{prop:double.int.kernel}
Let $d \geq 2$ and $\beta > 1/2$. Then 
\begin{equation*}
\begin{split}
&\int_{\mathbb{S}^d} \int_{\mathbb{S}^d} \mathcal{K}_\beta( \PT{x}, \PT{y} ) \, \dd \sigma_d( \PT{x} ) \dd \sigma_d( \PT{y} ) \\
&\phantom{equals}= 2^{d-2} \left[ 2^{\beta-1/2 + d/2} \frac{\gammafcn((d+1)/2) \gammafcn(\beta)}{\sqrt{\pi} \, \gammafcn(\beta + d/2)} \right]^2 \int_{0}^1 x^{2\beta + d - 2} \left[ \Hypergeom{2}{1}{1-d/2,d/2}{\beta+d/2}{x} \right]^2 \dd x.
\end{split}
\end{equation*}

For $d=2$ one has
\begin{equation*}
\int_{\mathbb{S}^2} \int_{\mathbb{S}^2} \mathcal{K}_\beta( \PT{x}, \PT{y} ) \, \dd \sigma_2( \PT{x} ) \dd \sigma_2( \PT{y} ) = \frac{1}{\beta} \frac{2^{2\beta}}{2\beta \left( 2\beta+1 \right)}.
\end{equation*}

In general, for $d \geq 2$ one has
\begin{equation*}
\begin{split}
\int_{\mathbb{S}^d} \int_{\mathbb{S}^d} \mathcal{K}_\beta( \PT{x}, \PT{y} ) \, \dd \sigma_d( \PT{x} ) \dd \sigma_d( \PT{y} ) 
&= \frac{2^{d-2}}{2\beta+d-1}  \left[ 2^{\beta-1/2 + d/2} \frac{\gammafcn((d+1)/2) \gammafcn(\beta)}{\sqrt{\pi} \, \gammafcn(\beta + d/2)} \right]^2 \\
&\phantom{=\pm}\times \KampedeFerietA{1:1;1}{1:2;2}{\displaystyle 2\beta + d - 1 \\ \displaystyle 2\beta + d}{\displaystyle 1 - d/2, d/2 \\ \displaystyle \beta + d/2}{\displaystyle 1 - d/2, d/2 \\ \displaystyle \beta + d/2}{1,1}.
\end{split} 
\end{equation*}
For even dimension $d$ the Kamp{\'e} de F{\'e}riet function terminates after finitely many terms.
\end{prop}

If $\beta = M$, where $M$ is a positive integer, the reproducing kernel $\mathcal{K}_\beta$ takes on a particularly simple form suitable for numerical computations. We recall that this kernel is the reproducing kernel for the Sobolev space $\mathbb{H}^s( \mathbb{S}^d )$ with smoothness $s = M + ( d - 1 ) / 2$.

\begin{prop} \label{prop:case.beta.EQ.integer}
Let $d \geq 2$ and suppose that $\beta$ is a positive integer $M$. Then
\begin{equation*}
\mathcal{K}_M( \PT{x}, \PT{y} ) = \mathcal{Q}_{M-1}( \PT{x}, \PT{y} ) + (-1)^M c_d(M) \left\| \PT{x} - \PT{y} \right\|^{2M-1}, \qquad \PT{x}, \PT{y} \in \mathbb{S}^d,
\end{equation*}
where $\mathcal{Q}_{M-1}( \PT{x}, \PT{y} )$ is the following polynomial of degree $M-1$ in terms of $\PT{x} \cdot \PT{y}$, 
\begin{equation*}
\mathcal{Q}_{M-1}( \PT{x}, \PT{y} ) %=  2^{2M-1} \frac{\Pochhsymb{d/2}{2M-1}}{\Pochhsymb{d}{2M-1}} \sum_{m=0}^{M-1} \frac{\Pochhsymb{1-M}{m}\Pochhsymb{1-M}{m}}{\Pochhsymb{2-d/2-2M}{m} \left( 2 \left( M - m \right) - 1 \right) m!} \left( \frac{1 - \PT{x} \cdot \PT{y}}{2} \right)^m 
= \frac{2^{2M-1}}{2M-1} \frac{\Pochhsymb{d/2}{2M-1}}{\Pochhsymb{d}{2M-1}} \Hypergeom{3}{2}{1-M,1/2-M,1-M}{3/2-M,2-d/2-2M}{\frac{1 - \PT{x} \cdot \PT{y}}{2}}
\end{equation*} 
and the positive constant $c_d(M)$ is given by
\begin{equation*}
c_d(M) = 2^{-2M} \frac{\gammafcn( ( d + 1 ) / 2 )}{\sqrt{\pi} \, \gammafcn( d/2 )} \frac{\gammafcn( M ) \gammafcn( M )}{\Pochhsymb{1/2}{M}\Pochhsymb{d/2}{M}}.
\end{equation*}

Moreover,
\begin{align*}
& \int_{\mathbb{S}^d} \int_{\mathbb{S}^d} \mathcal{Q}_{M-1}( \PT{x}, \PT{y} ) \, \dd \sigma_d( \PT{x} ) \dd \sigma_d( \PT{y} ) \\ & = (-1)^{M-1} \frac{(M-1)!}{\Pochhsymb{3/2}{M-1}} \Hypergeom{4}{3}{1-M,1/2-M,M+d-1,d/2}{1-M,(d+1)/2,d}{1}.
\end{align*}
(The $_{4}\HyperF_{3}$-function is a polynomial of degree $M-1$ evaluated at $1$.) For $d = 2$ this relation becomes
\begin{equation*}
\int_{\mathbb{S}^2} \int_{\mathbb{S}^2} \mathcal{Q}_{M-1}( \PT{x}, \PT{y} ) \, \dd \sigma_2( \PT{x} ) \dd \sigma_2( \PT{y} ) = \frac{2^{2M}}{M^2 \left( 2 M + 1 \right)^2} + \frac{(-1)^{M-1}}{M \left( 2 M + 1 \right)^2} \frac{(M-1)!}{\Pochhsymb{1/2}{M}}.
\end{equation*}
\end{prop}

The polynomial part $\mathcal{Q}_{M-1}$ of the kernel $\mathcal{K}_M$, $M$ a positive integer, gets annihilated in
\begin{equation*}
\frac{1}{N^2} \sum_{j=1}^N \sum_{k=1}^N \mathcal{K}_M( \PT{x}_j, \PT{x}_k ) - \int_{\mathbb{S}^d} \int_{\mathbb{S}^d} \mathcal{K}_M( \PT{x}, \PT{y} ) \, \dd \sigma_d( \PT{x} ) \dd \sigma_d( \PT{y} )
\end{equation*}
if the nodes $\PT{x}_1, \dots, \PT{x}_N$ form a spherical $t$-design with $t \geq M-1$. 

A {\em spherical $t$-design} $\{ \PT{x}_1, \dots, \PT{x}_{N(t)} \} \subseteq \mathbb{S}^d$, introduced in the ground breaking paper \cite{DeGoSei1977} by Delsarte, Goethals and Seidel, is characterized by
\begin{equation*}
\int_{\mathbb{S}^d} P( \PT{x} ) \, \dd \sigma_d( \PT{x} ) = \frac{1}{N(t)} \sum_{j = 1}^{N(t)} P( \PT{x}_j ) \qquad \text{for all polynomials $P$ with degree $\leq t$.}
\end{equation*}
Seymour and Zaslavsky~\cite{SeZa1984} proved that a spherical $t$-design always exists if $N(t)$ is sufficiently large. (From a bound in \cite{DeGoSei1977} it follows that $N(t) \geq c t^d$ for some constant $c > 0$.)

\begin{thm}[{\cite{BrHe2007}, \cite{He2006, HeSl2005b, HeSl2005, HeSl2006}}] \label{thm:bounds.spherical.designs}
Let $s > d/2$. There exists a constant $c(s,d) >0$ depending only on $s$ and $d$ such that for any $N$-point quadrature rule $Q_N$ we have
\begin{equation} \label{eq:asymptotic_in_t.bound1}
\frac{c(s,d)}{N^{s/d}} \leq \mathrm{wce}(\mathbb{H}^s,Q_N) \qquad \mbox{for all } N \ge 1. \end{equation}
Further, there exists a constant $C(s,d)$ depending only on $s$ and $d$ such that
\begin{equation} \label{eq:asymptotic_in_t.bound2}
\mathrm{wce}(\mathbb{H}^s,\widehat{Q}_N) \leq \frac{C(s,d)}{t^s}  \qquad \mbox{for all }  t \ge 1,
\end{equation}
where the quadrature rule $\widehat{Q}_N$ uses an $N(t)$-point spherical $t$-design as quadrature points.
\end{thm}

\begin{rmk}
Korevaar and Meyers~\cite{KoMe1993} conjectured that to every $t$ there exists a spherical $t$-design with $N(t) \leq C \, t^d$ ($C > 0$ some constant) points. (A. Bondarenko, D. Radchenko, and M. Viazovska propose in \cite{BoRaVi2011arXiv} a proof of this conjecture.) In light of this the bounds in \eqref{eq:asymptotic_in_t.bound1} and \eqref{eq:asymptotic_in_t.bound2} are best possible.
\end{rmk}

As a corollary to \eqref{eq:wce-B}, Theorem~\ref{thm:gen.Stolarsky}, Proposition~\ref{prop:case.beta.EQ.integer}, and Theorem~\ref{thm:bounds.spherical.designs} we have the following result.

\begin{cor}
Let $d \geq 2$ and $\beta > 1/2$. There exists a constant $c'(\beta,d) > 0$ independent of $\beta$ and $d$ such that for any $N$-point configuration $\mathcal{P}_N$ we have
\begin{equation*}
\frac{c'(\beta,d)}{N^{1/2 + (\beta - 1/2)/d}} \leq \left(\int_{-1}^1 \int_{\mathbb{S}^d} \left| \Delta_{\mathcal{P}_N,\beta}(\PT{z},t) \right|^2 \dd \sigma_d( \PT{z} ) \dd t \right)^{1/2} \qquad \mbox{for all } N \ge 1.
\end{equation*}
Further, there exists a constant $C'(\beta,d) > 0$ depending only on $\beta$ and $d$ such that
\begin{equation*} %\label{eq:asymptotic_in_t.bound}
\left(\int_{-1}^1 \int_{\mathbb{S}^d} \left| \Delta_{X_N,\beta}(\PT{z},t) \right|^2 \dd \sigma_d( \PT{z} ) \dd t \right)^{1/2}  \leq \frac{C'(\beta,d)}{t^{d/2 + \beta-1/2}}  \qquad \mbox{for all }  t \ge 1,
\end{equation*}
where the point set $X_N$ is an $N(t)$-point spherical $t$-design.

Moreover, if $\beta = M$, $M$ a positive integer, then for any spherical $t$-design $X_N = \{\PT{x}_1,\ldots, \PT{x}_N\}$ on $\mathbb{S}^d$ with $t \geq M - 1$ and $N = N(t)$ points, one has
\begin{equation*}
\begin{split}
&\int_{-1}^1 \int_{\mathbb{S}^d} \left| \Delta_{X_N,\beta}(\PT{z},t) \right|^2 \dd \sigma_d( \PT{z} ) \dd t \\ 
&\phantom{equals}= (-1)^{M-1} c_d(M) \left[ \int_{\mathbb{S}^d} \int_{\mathbb{S}^d} \left\| \PT{x} - \PT{y} \right\|^{2M-1} \, \dd \sigma_d( \PT{x} ) \dd \sigma_d( \PT{y} ) - \frac{1}{N^2} \sum_{j=1}^N \sum_{k=1}^N \left\| \PT{x}_j - \PT{x}_k \right\|^{2M-1} \right].
\end{split}
\end{equation*}
\end{cor}

\subsection{Representations of the reproducing kernel}

Our first result in this subsection concerns a closed form representation of the reproducing kernel $\mathcal{K}_\beta$ for arbitrary $\beta > 1/2$. Essentially, it can be expressed as a sum of a constant multiple of a (conditionally positive definite) (signed) power of the Euclidean distance and a certain integral of a hypergeometric function over the sphere $\mathbb{S}^d$ which turns out to be a special case of a  Kamp{\'e} de F{\'e}riet function. The latter simplifies to a polynomial of the point distance (or their inner product) if $\beta$ is a positive integer, as stated in Proposition~\ref{prop:case.beta.EQ.integer}. The exceptional cases when $\beta$ equals $L+1/2$ for some positive integer $L$, can be obtained by a limit process from the general result. In this way a logarithmic term in the Euclidean distance is introduced.

A {\em Kamp{\'e} de F{\'e}riet function in two variables} is defined by the double series expansion
\begin{equation}
\begin{split} \label{eq:Kampe.de.Feriet.function}
% &\KampedeFeriet{q:s;\nu}{p:r;\mu}{a_1, \dots,a_p : c_1, \dots, c_r; f_1, \dots, f_\mu}{b_1, \dots, b_q : d_1, \dots, d_s; g_1, \dots, g_\nu}{x,y} \\
&\KampedeFerietA{q:s;\nu}{p:r;\mu}{\displaystyle a_1, \dots,a_p \\ \displaystyle b_1, \dots, b_q}{\displaystyle c_1, \dots, c_r \\ \displaystyle d_1, \dots, d_s}{ \displaystyle f_1, \dots, f_\mu \\ \displaystyle g_1, \dots, g_\nu}{x,y} \\
&\phantom{equals}= \sum_{m=0}^\infty \sum_{n=0}^\infty \frac{\Pochhsymb{a_1}{m+n}\cdots \Pochhsymb{a_p}{m+n}}{\Pochhsymb{b_1}{m+n} \cdots \Pochhsymb{b_q}{m+n}}  \frac{\Pochhsymb{c_1}{m} \cdots \Pochhsymb{c_r}{m}}{\Pochhsymb{d_1}{m} \cdots \Pochhsymb{d_s}{m}} \frac{\Pochhsymb{f_1}{n} \cdots \Pochhsymb{f_\mu}{n}}{\Pochhsymb{g_1}{n} \cdots \Pochhsymb{g_\nu}{n}} \, \frac{x^m y^n}{m! n!},
\end{split}
\end{equation}
where $(a)_m$ denotes the Pochhammer symbol. It is a generalization of the {\em Gauss hypergeometric function}
\begin{equation*}
\Hypergeom{2}{1}{a,b}{c}{z} = \sum_{n=0}^\infty \frac{\Pochhsymb{a}{n} \Pochhsymb{b}{n}}{\Pochhsymb{c}{n} n!} z^n.
\end{equation*}

The second part of the theorem(s) concerns the ultraspherical expansion in terms of normalized Gegenbauer polynomials $P_n^{(d)}(t) = \GegenbauerC_n^{(d-1)/2}(t) / \GegenbauerC_k^{(d-1)/2}(1)$.

\begin{thm}[general case] \label{thm:integral.general.case}
Let $d\geq 2$. Let $\beta > 1/2$ with $2\beta-1$ not an even integer $> 0$; that is, $\beta - 1/2 = L + \eps$ for some integer $L\geq0$ and some real $0 < \eps < 1$. 

The kernel \eqref{eq:kernel} has the following representation
\begin{equation}  \label{eq:cal.K.beta.form}
\mathcal{K}_\beta( \PT{x}, \PT{y} ) = \mathcal{Q}_{\beta-1}( \PT{x}, \PT{y} ) + (-1)^{L+1} c_d(\beta) \left\| \PT{x} - \PT{y} \right\|^{2\beta-1}, \qquad \PT{x}, \PT{y} \in \mathbb{S}^d,
\end{equation}
where 
\begin{align*}
\mathcal{Q}_{\beta-1}( \PT{x}, \PT{y} ) 
&= \int_{\mathbb{S}^d} \frac{\frac{1}{2\beta-1} \Hypergeom{2}{1}{1/2-\beta,1-\beta}{3/2-\beta}{\left( \frac{\PT{x} \cdot \PT{z}-\PT{y} \cdot \PT{z}}{2 - \PT{x} \cdot \PT{z} - \PT{y} \cdot \PT{z}} \right)^2}}{\left( 1 - \frac{1}{2} \PT{x} \cdot \PT{z} - \frac{1}{2} \PT{y} \cdot \PT{z} \right)^{1-2\beta}} \dd \sigma_{d}(\PT{z}) \\
% &= \KampedeFeriet{1 : 0; 1}{2 : 0; 1}{1/2-\beta, 1 - \beta : \text{---}; 1/2}{(d+1)/2 : \text{---}; 3/2 - \beta}{\frac{1+\PT{x} \cdot \PT{y}}{2}, \frac{1-\PT{x} \cdot \PT{y}}{2}}.
&= \KampedeFerietA{1 : 0; 1}{2 : 0; 1}{\displaystyle 1/2-\beta, 1 - \beta \\ \displaystyle (d+1)/2}{ \displaystyle \vphantom{\displaystyle (d+1)/2}- \\ \displaystyle \vphantom{\displaystyle (d+1)/2}- }{ \displaystyle 1/2 \\ \displaystyle 3/2 - \beta}{\frac{1+\PT{x} \cdot \PT{y}}{2}, \frac{1-\PT{x} \cdot \PT{y}}{2}}, \qquad \PT{x}, \PT{y} \in \mathbb{S}^d,
\end{align*}
and $c_d(\beta)$ is the positive constant 
\begin{equation*}
c_d(\beta) \DEF \frac{2^{-2\beta}}{\sin( \pi (\beta-1/2-L) )} \frac{\gammafcn((d+1)/2) \gammafcn(\beta)\gammafcn(\beta)}{\gammafcn(\beta+d/2) \gammafcn(\beta+1/2)}.
\end{equation*}
For $\beta = 1$ the kernel \eqref{eq:cal.K.beta.form} reduces to \eqref{eq:cal.C.kernel}. 

Let $s$ be defined by means of $2\beta - 1 = 2 s - d$. The ultraspherical expansions %of $\mathcal{Q}_{\beta-1}( \PT{x}, \PT{y} )$,
\begin{equation} \label{eq:thm.K.beta.2.series}
\mathcal{Q}_{\beta-1}( \PT{x}, \PT{y} ) = \sum_{k=0}^{\infty} A_{k}^{(2)}(s,d) \, Z(d,k) P_k^{(d)}(\PT{x} \cdot \PT{y}),
\end{equation}
reduces to a polynomial of degree $M-1$ if $\beta = M$ ($M$ a positive integer) and the coefficients $A_{k}^{(2)}(s,d)$ have the following asymptotic behavior otherwise: 
\begin{equation*}
A_{k}^{(2)}(s,d) \sim - 2^{d-2} \left[ \frac{\gammafcn( (d+1)/2 ) \gammafcn(s+1/2) \gammafcn(2s)}{\gammafcn( (d + 1)/2 - s) \gammafcn(s+1)} \right]^2 \frac{\gammafcn(d/2)}{\gammafcn(d/2+2s)} \, k^{-4s} \qquad \text{as $k \to \infty$.}
\end{equation*}
Thus, the series expansion \eqref{eq:thm.K.beta.2.series} converges for all $\PT{x}, \PT{y} \in \mathbb{S}^d$ if $s > d/2$.
\end{thm}

Various forms of the coefficients $A_{k}^{(2)}(s,d)$ can be found in the Proof of Theorem~\ref{thm:integral.general.case}.

The following potential-theoretical constant (also appearing at the right-hand side of Stolarsky's invariance principle \eqref{eq:Stolarsky.s.invariance.principle} if $\lambda = 1$) plays a role:
\begin{equation} \label{eq:V.lambda.S.d}
V_{\lambda}(\mathbb{S}^d) \DEF \int_{\mathbb{S}^d} \int_{\mathbb{S}^d} \left\| \PT{x} - \PT{y} \right\|^\lambda \dd \sigma_d( \PT{x} ) \dd \sigma_d( \PT{y} ) = 2^{d-1+\lambda} \frac{\gammafcn((d+1)/2) \gammafcn(( d + \lambda ) / 2)}{\sqrt{\pi} \, \gammafcn(d + \lambda / 2)}.
\end{equation}
(Its computation can be found, for example, in S. Hubbert and B. J. C. Baxter~\cite{HuBa2001}; also cf. Lemma~\ref{lem:aux.res.2}. For the connection to potential-theory we refer the reader to G. Bj{\"o}rck~\cite{Bj1956}.)

The expansion of the (signed) power of the Euclidean distance is well-known and follows, for example, from S. Hubbert and B. J. C. Baxter~\cite{HuBa2001}. (For details see, for example, \cite{BrWo2010d_pre}.) 
Using the notation of Theorem~\ref{thm:integral.general.case} one has
% For convenience we let $2\beta-1 = 2s - d$ and let the integer $L$ be defined as before (that is, $s-d/2=L+\eps$ with $0 < \eps < 1$). Then 
\begin{equation} \label{eq:signed.generalized.distance.expansion}
(-1)^{L+1} \left\| \PT{x} - \PT{y} \right\|^{2s-d} = \sum_{k=0}^\infty \frac{a_k(s,d)}{Z(d,k)} \, Z(d,k) P_k^{(d)}( \PT{x} \cdot \PT{y} ), \qquad \PT{x}, \PT{y} \in \mathbb{S}^d,
\end{equation}
where (cf. \eqref{eq:V.lambda.S.d})
\begin{align}
\frac{a_0(s,d)}{Z(d,0)} &= (-1)^{L+1} V_{2s-d}( \mathbb{S}^d ) = (-1)^{L+1} 2^{2s-1}\,\frac{\gammafcn\left((d+1)/2\right)\gammafcn(s)}{\sqrt{\pi}\,\gammafcn(d/2+s)}, \notag \\
\frac{a_k(s,d)}{Z(d,k)} &= (-1)^{L+1} V_{2s-d}( \mathbb{S}^d ) \frac{\Pochhsymb{d/2-s}{k}}{\Pochhsymb{d/2+s}{k}}, \quad k \geq 1. \label{eq:a.k.s.d}
\end{align}
The coefficient $a_k(s,d) / Z(d,k)$ is alternating in sign for $k \leq L$ and positive for $k \geq L+1$. The coefficients have the following asymptotic behavior
\begin{equation*}
\frac{a_k(s,d)}{Z(d,k)} \sim (-1)^{L+1} V_{2s-d}( \mathbb{S}^d ) \frac{\gammafcn(d/2+s)}{\gammafcn(d/2-s)} \, k^{-2s} \qquad \text{as $k \to \infty$.} \notag
\end{equation*}

For further references we record the two special cases ($L = \lfloor \beta - 1/2 \rfloor$)
\begin{align}
\mathcal{K}_\beta( \PT{x}, \PT{x} ) 
&= \frac{1}{2\beta-1} \frac{\gammafcn((d+1)/2) \gammafcn(2\beta + d/2-1)}{\gammafcn(\beta + (d -1)/2) \gammafcn(\beta + d/2)}, \label{eq:K.beta.x.x} \\
\mathcal{K}_\beta( \PT{x}, -\PT{x} ) &= (-1)^{L+1} c_d(\beta) 2^{2\beta-1} + \frac{1}{2\beta-1} \Hypergeom{3}{2}{1/2-\beta, 1 - \beta, 1/2}{3/2 - \beta, (d+1)/2}{1}. \label{eq:K.beta.x.neg.x}
\end{align}
The relation \eqref{eq:K.beta.x.x} is valid for each $\beta > 0$. The relation \eqref{eq:K.beta.x.neg.x} is valid for $2\beta-1$ not an integer $\geq 0$.

\begin{thm}[exceptional cases] \label{thm:integral.exceptional.cases}
Let $d\geq 2$. Let $2\beta-1$ be the positive even integer $2L$. 

The kernel \eqref{eq:kernel} has the following representation: for $\PT{x}, \PT{y} \in \mathbb{S}^d$
\begin{equation}  \label{eq:cal.K.beta.form.exceptional}
\mathcal{K}_{L+1/2}( \PT{x}, \PT{y} ) = \widetilde{\mathcal{Q}}_{L-1/2}( \PT{x}, \PT{y} ) + (-1)^{L+1} \frac{\Pochhsymb{1/2}{L} \Pochhsymb{1/2}{L}}{2^{2L} \Pochhsymb{(d+1)/2}{L} \, L!} \left\| \PT{x} - \PT{y} \right\|^{2L} \ln \| \PT{x} - \PT{y} \|,  %(-1)^{L+1} c_d(\beta) \left\| \PT{x} - \PT{y} \right\|^{2\beta-1}, \qquad \PT{x}, \PT{y} \in \mathbb{S}^d,
\end{equation}
where the integral representation of $\widetilde{\mathcal{Q}}_{L-1/2}$ is given by
\begin{equation*}
\begin{split}
\widetilde{\mathcal{Q}}_{L-1/2}( \PT{x}, \PT{y} ) 
&= \int_{\mathbb{S}^d} \frac{\frac{1}{2} \sum_{n=1}^{L} \frac{1}{n} \frac{\Pochhsymb{1/2-L}{L-n}}{(L-n)!} \left( \frac{\PT{x} \cdot \PT{z}-\PT{y} \cdot \PT{z}}{2 - \PT{x} \cdot \PT{z} - \PT{y} \cdot \PT{z}} \right)^{2L-2n}}{\left( 1 - \frac{1}{2} \PT{x} \cdot \PT{z} - \frac{1}{2} \PT{y} \cdot \PT{z} \right)^{-2L}} \dd \sigma_{d}(\PT{z}) \\
&\phantom{=}- \int_{\mathbb{S}^d} \frac{\frac{1}{2} \frac{\Pochhsymb{1/2-L}{L}}{L!} \sum_{n=1}^\infty \frac{1}{n} \frac{\Pochhsymb{1/2}{n}}{\Pochhsymb{L+1}{n}} \left( \frac{\PT{x} \cdot \PT{z}-\PT{y} \cdot \PT{z}}{2 - \PT{x} \cdot \PT{z} - \PT{y} \cdot \PT{z}} \right)^{2n+2L}}{\left( 1 - \frac{1}{2} \PT{x} \cdot \PT{z} - \frac{1}{2} \PT{y} \cdot \PT{z} \right)^{-2L}} \dd \sigma_{d}(\PT{z}) \\
&\phantom{=}+ \frac{(-1)^{L+1}}{2} \frac{\Pochhsymb{1/2}{L}}{2^{2L} \, L!} \int_{\mathbb{S}^d} \left( \PT{x} \cdot \PT{z}-\PT{y} \cdot \PT{z} \right)^{2L} \\
&\phantom{=\pm}\times \Bigg[ \ln \Big( \frac{\frac{\PT{x}-\PT{y}}{\| \PT{x} - \PT{y} \|} \cdot \PT{z}}{2 - \PT{x} \cdot \PT{z}-\PT{y} \cdot \PT{z}} \Big)^2 + \digammafcn( 1 / 2 ) - \digammafcn( L + 1 ) \Bigg] \dd \sigma_d( \PT{z} ), \qquad \PT{x}, \PT{y} \in \mathbb{S}^d,
\end{split}
\end{equation*}
and the series representation of $\widetilde{\mathcal{Q}}_{L-1/2}$ is given by
\begin{equation*}
\begin{split}
\widetilde{\mathcal{Q}}_{\beta-1}( \PT{x}, \PT{y} ) 
&= \frac{1}{2} \sum_{n=0}^{L-1} \frac{1}{L-n} \frac{\Pochhsymb{1/2-L}{n}\Pochhsymb{1/2}{n}}{\Pochhsymb{(d+1)/2}{n} \, n!} \Hypergeom{2}{1}{-n, 1/2 - L }{1/2-n}{\frac{1+\PT{x} \cdot \PT{y}}{2}} \\
&\phantom{=}- \frac{1}{2} \sum_{n=1}^\infty \frac{1}{n} \frac{\Pochhsymb{1/2-L}{n+L}\Pochhsymb{1/2}{n+L}}{\Pochhsymb{(d+1)/2}{n+L} \, (n+L)!} \Hypergeom{2}{1}{-L-n, 1/2 - L}{1/2-L-n}{\frac{1+\PT{x} \cdot \PT{y}}{2}} \\
&\phantom{=}+ \frac{1}{2} \frac{\Pochhsymb{1/2-L}{L}\Pochhsymb{1/2}{L}}{\Pochhsymb{(d+1)/2}{L} \, L!} \sum_{k=1}^L \frac{(-1)^k}{k} \frac{\Pochhsymb{-L}{k}}{\Pochhsymb{1/2-L}{k}} \left( \frac{1-\PT{x} \cdot \PT{y}}{2} \right)^{L-k} \left( \frac{1+\PT{x} \cdot \PT{y}}{2} \right)^k \\
&\phantom{=}- \frac{(-1)^{L+1}}{2} \frac{\Pochhsymb{1/2}{L} \Pochhsymb{1/2}{L}}{\Pochhsymb{(d+1)/2}{L} L!} \, \left( \frac{1-\PT{x} \cdot \PT{y}}{2} \right)^{L} \Big( 2 \ln 2 + \digammafcn( L + 1 ) + \digammafcn( L + ( d + 1 ) / 2 ) \\
&\phantom{=\pm}- \digammafcn( 1 / 2 ) - \digammafcn( L + 1 / 2 ) \Big), \qquad \PT{x}, \PT{y} \in \mathbb{S}^d.
\end{split}
\end{equation*}

Let $s$ be defined by means of $2\beta - 1 = 2 s - d$; that is $s = L + d/2$. The coefficients in the ultraspherical expansions %of $\mathcal{Q}_{\beta-1}( \PT{x}, \PT{y} )$,
\begin{equation} \label{eq:thm.K.beta.2.series.exceptional}
\widetilde{\mathcal{Q}}_{L-1/2}( \PT{x}, \PT{y} ) = \sum_{k=0}^{\infty} \widetilde{A}_{k}^{(2)}(L+d/2,d) \, Z(d,k) P_k^{(d)}(\PT{x} \cdot \PT{y}),
\end{equation}
have the following asymptotic behavior: 
\begin{equation*} 
\widetilde{A}_{k}^{(2)}(L+d/2,d) \sim - 2^{d-2} \left[ \frac{\gammafcn( (d+1)/2 ) \gammafcn(L+(d+1)/2) \gammafcn(2L+d)}{\gammafcn( 1/2 - L) \gammafcn(L+1+d/2)} \right]^2 \frac{\gammafcn(d/2)}{\gammafcn(d/2+2L+d)} \, k^{-4s}.
\end{equation*}
Thus, the series expansion \eqref{eq:thm.K.beta.2.series.exceptional} converges for all $\PT{x}, \PT{y} \in \mathbb{S}^d$ if $s > d/2$.
\end{thm}

The ultraspherical expansion of the function
\begin{equation*}
(-1)^{L+1} \left\| \PT{x} - \PT{y} \right\|^{2L} \ln \| \PT{x} - \PT{y} \|
\end{equation*}
is also  well-known; cf., for example, S. Hubbert and B. J. C. Baxter~\cite{HuBa2001}. (For details see, for example, \cite{BrWo2010d_pre}.) Indeed, for a positive integer $L$ one has
\begin{equation*}
(-1)^{L+1} \left\| \PT{x} - \PT{y} \right\|^{2L} \ln \| \PT{x} - \PT{y} \| = \sum_{k = 0}^\infty \frac{b_k(2L,d)}{Z(d, k)} \, Z(d, k) \, P_k^{(d)}( \PT{x} \cdot \PT{y}  ), \qquad \PT{x}, \PT{y} \in \mathbb{S}^d,
\end{equation*}
with the following coefficients: 
\begin{align*}
\frac{b_0(2L,d)}{Z(d, 0)} &= (-1)^{L+1} V_{2L}^{\mathrm{log}}(\mathbb{S}^d) = (-1)^{L+1} \frac{V_{2L}( \mathbb{S}^d )}{2} \left[ \digammafcn( L + d / 2 ) + 2 \log 2 - \digammafcn( L + d ) \right]
\intertext{and for $k \leq L$}
\frac{b_k(2L,d)}{Z(d, k)} &= \frac{1}{2} \frac{a_k(2L,d)}{Z(d, k)} \left[ \digammafcn( L + 1 ) + \digammafcn( L + d/2 ) + 2 \log 2 - \digammafcn( L + 1 - k ) - \digammafcn( k + L + d ) \right],
\intertext{where $a_k(2L,d) / Z(d, k)$ is given in \eqref{eq:a.k.s.d}, and for $k > L$}
\frac{b_k(2L,d)}{Z(d, k)} &= 2^{2L + d - 2} \frac{\gammafcn( ( d + 1 ) / 2 ) \gammafcn( L + d / 2 ) \gammafcn(L + 1)}{\sqrt{\pi}} \frac{\gammafcn( k - L )}{\gammafcn( k + L + d )}.
\end{align*}
The coefficients have the following asymptotic behavior 
\begin{equation*}
\frac{b_k(2L,d)}{Z(d, k)} \sim 2^{2L + d - 2} \frac{\gammafcn( ( d + 1 ) / 2 ) \gammafcn( L + d / 2 ) \gammafcn(L + 1)}{\sqrt{\pi}} k^{-2s} \qquad \text{as $k \to \infty$.}
\end{equation*}

For further references we record that
\begin{equation}
\begin{split} \label{eq:K.L.plus.half.at.antipodal}
\mathcal{K}_{L+1/2}( \PT{x}, -\PT{x} )  
&= - \frac{1}{2} \sum_{\substack{n=0 \\ n\neq L}}^\infty \frac{1}{n-L} \frac{\Pochhsymb{1/2-L}{n} \Pochhsymb{1/2}{n}}{\Pochhsymb{(d+1)/2}{n} n!} - \frac{(-1)^{L+1}}{2} \frac{\Pochhsymb{1/2}{L} \Pochhsymb{1/2}{L}}{\Pochhsymb{(d+1)/2}{L} L!} \\
&\phantom{=\pm}\times \Big( \digammafcn( L + 1 ) + \digammafcn( L + ( d + 1 ) / 2 ) - \digammafcn( 1 / 2 ) - \digammafcn( L + 1 / 2 ) \Big). 
\end{split}
\end{equation}

\section{Proofs}
\label{sec:proofs}

First, we prove the following auxiliary results.

\begin{lem} \label{lem:aux.res.1}
Let $\beta > 0$. Then for $-1 \leq a \leq 1$ and $-1 \leq b \leq 1$ with $a \neq b$
\begin{align*}
H_\beta( a, b ) 
&\DEF \int_{-1}^1 \left( a - t \right)_+^{\beta-1} \left( b - t \right)_+^{\beta-1} \dd t \\
&= \frac{1}{\beta} \left( 1 + \min\{a, b\} \right) \left( 1 + a \right)^{\beta - 1} \left( 1 + b \right)^{\beta - 1} \Hypergeom{2}{1}{1-\beta,1}{1+\beta}{\frac{1+\min\{a,b\}}{1+\max\{a,b\}}}.
\end{align*}

If $b = a$ and $-1 < a \leq 1$, then the integral exists if and only if $\beta > 1/2$. In this case
\begin{equation*}
H_\beta( a, a ) = \frac{1}{2 \beta - 1} \left( 1 + a \right)^{2\beta - 1}.
\end{equation*}

If $a = b = - 1$, then $H_\beta( -1, -1 ) = 0$ for all $\beta > 0$.
\end{lem}

\begin{proof}
Without loss of generality we may assume that $-1 \leq a < b \leq 1$. The change of variable $( 1 + a ) \, u = 1 + t$ gives
\begin{align*}
H_\beta( a, b ) 
&= \int_{-1}^1 \left( a - t \right)_+^{\beta-1} \left( b - t \right)_+^{\beta-1} \dd t \\
&= \int_{-1}^{a} \left( a - t \right)^{\beta-1} \left( b - t \right)^{\beta-1} \dd t \\
&= \left( 1 + a \right) \int_0^1 \left[ 1 + a - \left( 1 + a \right) u \right]^{\beta - 1} \left[ 1 + b - \left( 1 + a \right) u \right]^{\beta - 1} \dd u \\
&= \left( 1 + a \right)^{\beta} \left( 1 + b \right)^{\beta - 1} \int_0^1 \left( 1 - u \right)^{\beta-1} \left( 1 - \frac{1 + a}{1+b} \, u \right)^{\beta-1} \dd u.
\end{align*}
The integral represents a Gauss hypergeometric function (cf. \cite[Eq.~15.6.1]{DLMF2011.08.29}); that is
\begin{equation*}
H_\beta( a, b ) = \left( 1 + \min\{a, b\} \right) \left( 1 + a \right)^{\beta - 1} \left( 1 + b \right)^{\beta - 1} \frac{1}{\beta} \Hypergeom{2}{1}{1-\beta,1}{1+\beta}{\frac{1+a}{1+b}}.
\end{equation*}
If $\min\{a,b\} = b$, then $a$ and $b$ need to be interchanged in the result above. This yields the general formula for the integral. 

In the case $b = a$ and $-1 < a \leq 1$ one has (if and only if $\beta > 1/2$)
\begin{equation*}
H_\beta( a, a ) = \int_{-1}^1 \left( a - t \right)_+^{2\beta-2} \dd t = \left( 1 + a \right)^{2\beta-1} \int_0^1 \left(1 - u \right)^{2\beta-2} \dd u = \frac{\left( 1 + a \right)^{2\beta-1}}{2\beta-1}.
\end{equation*}
Clearly, $H_\beta(-1,-1) = 0$ for $\beta > 0$. This gives the special case.
\end{proof}

\begin{lem} \label{lem:aux.res.2}
Let $\gamma \geq 1$. Then
\begin{equation*}
\int_{\mathbb{S}^d} \left( 1 + \PT{x} \cdot \PT{z} \right)^{\gamma} \dd \sigma_d( \PT{z} ) = 2^{-\gamma} \int_{\mathbb{S}^d} \left\| \PT{x} - \PT{z} \right\|^{2\gamma} \dd \sigma_d( \PT{z} ) = 2^{d - 1 + \gamma} \frac{\gammafcn((d+1)/2) \gammafcn(\gamma + d/2)}{\sqrt{\pi} \, \gammafcn(\gamma + d)}.
\end{equation*}
\end{lem}

\begin{proof}
The first part follows from the spherical symmetry of the measure $\sigma_d$ and the identity
\begin{equation*}
\left\| \PT{x} - \PT{y} \right\|^2 = 2 \left( 1 - \PT{x} \cdot \PT{y} \right), \qquad \PT{x}, \PT{y} \in \mathbb{S}^d.
\end{equation*}

Using the Funk-Hecke formula (cf. \cite{Mu1966}), we obtain for the second part
\begin{equation*}
\int_{\mathbb{S}^d} \left( 1 + \PT{x} \cdot \PT{z} \right)^{\gamma} \dd \sigma_d( \PT{z} ) = \frac{\omega_{d-1}}{\omega_d} \int_{-1}^1 \left( 1 + t \right)^\gamma \left( 1 - t^2 \right)^{d/2-1} \dd t.
\end{equation*}
The change of variable $2u = 1 + t$ yields
\begin{equation*}
\int_{\mathbb{S}^d} \left( 1 + \PT{x} \cdot \PT{z} \right)^{\gamma} \dd \sigma_d( \PT{z} ) = 2^{\gamma + d - 2 + 1} \frac{\omega_{d-1}}{\omega_d} \int_0^1 u^{\gamma+d/2-1} \left( 1 - u \right)^{d/2-1} \dd u.
\end{equation*}
The last integral is the Beta function $\betafcn( \gamma + d/2, d/2 )$. By \eqref{eq:C.d} and \cite[Eq.~5.12.1]{DLMF2011.08.29}
\begin{equation*}
\int_{\mathbb{S}^d} \left( 1 + \PT{x} \cdot \PT{z} \right)^{\gamma} \dd \sigma_d( \PT{z} ) = 2^{d - 1 + \gamma} \frac{\gammafcn((d+1)/2)}{\sqrt{\pi} \, \gammafcn(d/2)} \frac{\gammafcn(\gamma + d/2) \gammafcn(d/2)}{\gammafcn(\gamma + d)}.
\end{equation*}
Simplification gives the result.
\end{proof}

\begin{lem} \label{lem:aux.res.3}
Let $\beta > 0$. Let $1 \geq a > b > -1$. If $2 \beta - 1$ is not an integer $\geq 0$, then
\begin{equation*}
\begin{split}
&\frac{1}{\beta} \left( 1 + a \right)^{\beta - 1} \left( 1 + b \right)^{\beta} \Hypergeom{2}{1}{1-\beta,1}{1+\beta}{\frac{1+b}{1+a}} \\
&\phantom{equals}= \frac{\left( 1 + a \right)^{2\beta - 1}}{2\beta-1} \Hypergeom{2}{1}{1-\beta,1-2\beta}{2-2\beta}{\frac{a-b}{1+a}} + \frac{\gammafcn( 1 - 2 \beta ) \gammafcn(\beta)}{\gammafcn(1 - \beta)} \left( a - b \right)^{2\beta-1}.
\end{split}
\end{equation*}
If $2 \beta - 1 = m$ and $m$ a positive even integer, then
\begin{equation*}
\begin{split}
&\frac{1}{\beta} \left( 1 + a \right)^{\beta - 1} \left( 1 + b \right)^{\beta} \Hypergeom{2}{1}{1-\beta,1}{1+\beta}{\frac{1+b}{1+a}} \\
&\phantom{equals}= \frac{\left( 1 + a \right)^{\beta - 1} \left( 1 + b \right)^{\beta} }{\gammafcn( 2 \beta )} \sum_{k=0}^{m-1} (-1)^k \Pochhsymb{1-\beta}{k} (m - 1 - k)! \left( \frac{a-b}{1+a} \right)^k \\
&\phantom{equals=}- \frac{(-1)^m \gammafcn(\beta) }{\gammafcn(1-\beta) \gammafcn(2\beta)} \left( a - b \right)^{2\beta-1} \left( \frac{1+b}{1+a} \right)^\beta  \sum_{k=0}^\infty \frac{\Pochhsymb{\beta}{k}}{k!} \left( \frac{a-b}{1+a} \right)^k \Big( \digammafcn( k + \beta ) - \digammafcn( k + 1 ) \Big) \\
&\phantom{equals=}- \frac{(-1)^m \gammafcn(\beta) }{\gammafcn(1-\beta) \gammafcn(2\beta)} \left( \ln \frac{a-b}{1+a} \right) \left( a - b \right)^{2\beta-1}.
\end{split}
\end{equation*}
If $\beta$ is a positive integer $n$ (that is, $2\beta-1$ is a positive odd integer), then \begin{equation*}
\frac{1}{\beta} \left( 1 + a \right)^{\beta - 1} \left( 1 + b \right)^{\beta} \Hypergeom{2}{1}{1-\beta,1}{1+\beta}{\frac{1+b}{1+a}} = - \left( 1 + a \right)^{n - 1} \left( 1 + b \right)^{n} \sum_{k=0}^{n-1} \frac{\Pochhsymb{1-n}{k}}{\Pochhsymb{1-2n}{k+1}} \left( \frac{a-b}{1+a} \right)^k.
\end{equation*}
\end{lem}

\begin{proof}
First, we observe that by the assumptions on $a$ and $b$ there holds
\begin{equation*}
0 < \frac{1+b}{1+a} < 1 \qquad \text{and} \qquad 0 < \frac{a-b}{1+a} < 1. 
\end{equation*}
Thus, the conditions on the argument of the hypergeometric function imposed by the variable transformation rules used in the following are satisfied.

Let $2 \beta - 1$ be not an integer $\geq 0$. Then, by \cite[Eq.~15.8.4]{DLMF2011.08.29} (and after changing to regularized Gauss hypergeometric functions)
\begin{align*}
f_\beta( a, b) 
&\DEF \frac{1}{\beta} \left( 1 + a \right)^{\beta - 1} \left( 1 + b \right)^{\beta} \Hypergeom{2}{1}{1-\beta,1}{1+\beta}{\frac{1+b}{1+a}} \\
&= \frac{\gammafcn(\beta+1)}{\beta} \left( 1 + a \right)^{\beta - 1} \left( 1 + b \right)^{\beta} \HypergeomReg{2}{1}{1-\beta,1}{1+\beta}{\frac{1+b}{1+a}} \\
&= \frac{\gammafcn(\beta+1)}{\beta} \left( 1 + a \right)^{\beta - 1} \left( 1 + b \right)^{\beta} \frac{\pi}{\sin\big( \pi ( 2 \beta - 1 ) \big)} \Bigg[ \frac{1}{\gammafcn(2\beta) \gammafcn(\beta)} \HypergeomReg{2}{1}{1-\beta,1}{2-2\beta}{\frac{a-b}{1+a}} \\
&\phantom{=}- \frac{1}{\gammafcn(1 - \beta) \gammafcn(1)} \left( \frac{a-b}{1+a} \right)^{2\beta-1} \HypergeomReg{2}{1}{2\beta,\beta}{2\beta}{\frac{a-b}{1+a}} \Bigg].
\end{align*}

In the second hypergeometric function the lower parameter coincides with one of the upper ones. Thus, it reduces to (cf. \cite[Eq.~15.4.6]{DLMF2011.08.29})
\begin{equation*}
\HypergeomReg{2}{1}{2\beta,\beta}{2\beta}{\frac{a-b}{1+a}} = \frac{1}{\gammafcn(2\beta)} \Hypergeom{2}{1}{2\beta,\beta}{2\beta}{\frac{a-b}{1+a}} = \frac{1}{\gammafcn(2\beta)} \left( 1 - \frac{a-b}{1+a} \right)^{-\beta} = \frac{1}{\gammafcn(2\beta)} \left( \frac{1+b}{1+a} \right)^{-\beta}.
\end{equation*}
To the first hypergeometric function we apply the last of the  transformations \cite[Eq.~15.8.1]{DLMF2011.08.29}
\begin{equation*}
\HypergeomReg{2}{1}{1-\beta,1}{2-2\beta}{\frac{a-b}{1+a}} = \left( \frac{1+b}{1+a} \right)^{-\beta} \HypergeomReg{2}{1}{1-\beta,1-2\beta}{2-2\beta}{\frac{a-b}{1+a}}
\end{equation*}
From the reflection formula for the Gamma function (\cite[Eq.~5.5.1]{DLMF2011.08.29}) it follows that
\begin{equation*}
\frac{\pi}{\sin\big( \pi ( 2 \beta - 1 ) \big)} = - \frac{\pi}{\sin\big( \pi ( 1 - 2 \beta ) \big)} = - \gammafcn( 1 - 2 \beta ) \gammafcn( 2 \beta ).
\end{equation*}
Taking all these observations into account, we arrive at
\begin{align*}
f_\beta( a, b) 
&= - \gammafcn(\beta) \left( 1 + a \right)^{\beta - 1} \left( 1 + b \right)^{\beta} \gammafcn( 1 - 2 \beta ) \gammafcn( 2 \beta ) \Bigg[ \frac{1}{\gammafcn(2\beta) \gammafcn(\beta)} \left( \frac{1+b}{1+a} \right)^{-\beta} \\
&\phantom{=\pm}\times \HypergeomReg{2}{1}{1-\beta,1-2\beta}{2-2\beta}{\frac{a-b}{1+a}} - \frac{1}{\gammafcn(1 - \beta)} \left( \frac{a-b}{1+a} \right)^{2\beta-1} \frac{1}{\gammafcn(2\beta)} \left( \frac{1+b}{1+a} \right)^{-\beta} \Bigg] \\
&= - \gammafcn( 1 - 2 \beta ) \left( 1 + a \right)^{2\beta - 1} \Bigg[ \HypergeomReg{2}{1}{1-\beta,1-2\beta}{2-2\beta}{\frac{a-b}{1+a}} - \frac{\gammafcn(\beta) }{\gammafcn(1 - \beta)} \left( \frac{a-b}{1+a} \right)^{2\beta-1} \Bigg].
\end{align*}
The first part of the lemma follows after changing back to Gauss hypergeometric function.

Let $2 \beta - 1 = m$ and $m$ an even integer $\geq 0$. Then (cf. \cite[Eq.~15.8.10]{DLMF2011.08.29})
\begin{align*}
f_\beta( a, b) 
&= \frac{\gammafcn(\beta+1)}{\beta} \left( 1 + a \right)^{\beta - 1} \left( 1 + b \right)^{\beta} \Bigg\{ \frac{1}{\gammafcn( \beta ) \gammafcn( 2 \beta )} \sum_{k=0}^{m-1} \frac{\Pochhsymb{1-\beta}{k} \Pochhsymb{1}{k} (m - 1 - k)!}{k!} \left( - \frac{a-b}{1+a} \right)^k \\
&\phantom{=}- \frac{\left( - \frac{a-b}{1+a} \right)^m}{\gammafcn(1-\beta) \gammafcn(1)} \sum_{k=0}^\infty \frac{\Pochhsymb{\beta}{k} \Pochhsymb{2\beta}{k}}{k! (k + m)!} \left( \frac{a-b}{1+a} \right)^k \Big[ \ln \frac{a-b}{1+a} - \digammafcn( k + 1 ) + \digammafcn( k + \beta ) \Big] \Bigg\} \\
&= \frac{\left( 1 + a \right)^{\beta - 1} \left( 1 + b \right)^{\beta} }{\gammafcn( 2 \beta )} \sum_{k=0}^{m-1} \Pochhsymb{1-\beta}{k} (m - 1 - k)! \left( - \frac{a-b}{1+a} \right)^k \\
&\phantom{=}- \frac{(-1)^m \gammafcn(\beta) }{\gammafcn(1-\beta)} \left( a - b \right)^{2\beta-1} \left( \frac{1+b}{1+a} \right)^\beta  \sum_{k=0}^\infty \frac{\Pochhsymb{\beta}{k} \Pochhsymb{2\beta}{k}}{k! (k + m)!} \left( \frac{a-b}{1+a} \right)^k \\
&\phantom{=\pm}\times \Big[ \ln \frac{a-b}{1+a} - \digammafcn( k + 1 ) + \digammafcn( k + \beta ) \Big].
\end{align*}
Using identities for the Pochhammer symbol, one sees that
\begin{equation*}
\frac{\Pochhsymb{2\beta}{k}}{(k + m)!} = \frac{\Pochhsymb{m+1}{k}}{(k + m)!} = \frac{\Pochhsymb{m+1}{k}}{\Pochhsymb{1}{m+k}} = \frac{\Pochhsymb{m+1}{k}}{\Pochhsymb{1}{m} \Pochhsymb{m+1}{k}} = \frac{1}{m!} = \frac{1}{\gammafcn(2\beta)}.
\end{equation*}
This yields
\begin{equation*}
\begin{split}
f_\beta( a, b) 
&= \frac{\left( 1 + a \right)^{\beta - 1} \left( 1 + b \right)^{\beta} }{\gammafcn( 2 \beta )} \sum_{k=0}^{m-1} \Pochhsymb{1-\beta}{k} (m - 1 - k)! \left( - \frac{a-b}{1+a} \right)^k \\
&\phantom{=}- \frac{(-1)^m \gammafcn(\beta) }{\gammafcn(1-\beta) \gammafcn(2\beta)} \left( a - b \right)^{2\beta-1} \left( \frac{1+b}{1+a} \right)^\beta  \sum_{k=0}^\infty \frac{\Pochhsymb{\beta}{k}}{k!} \left( \frac{a-b}{1+a} \right)^k \Big( \digammafcn( k + \beta ) - \digammafcn( k + 1 ) \Big) \\
&\phantom{=}- \frac{(-1)^m \gammafcn(\beta) }{\gammafcn(1-\beta) \gammafcn(2\beta)} \left( a - b \right)^{2\beta-1} \left( \ln \frac{a-b}{1+a} \right) \left( \frac{1+b}{1+a} \right)^\beta \Hypergeom{2}{1}{\beta,1}{1}{\frac{a-b}{1+a}}.
\end{split}
\end{equation*}
The right-most hypergeometric function reduces to $( ( 1 + b ) / ( 1 + a ) )^{-\beta}$. The second part of the lemma follows.

Finally, let $\beta = n$ and $n$ a positive integer. Then $f_n(a,b)$ reduces to a polynomial of degree $n-1$ and \cite[Eq.~15.8.7]{DLMF2011.08.29} can be applied. (Note that after the transformation the lower parameter satisfies $2-2n = 1 - n - M$ with $M \geq 0$; cf. \cite[Eq.~15.2.5]{DLMF2011.08.29}.)
\begin{align*}
f_\beta( a, b) 
&= \frac{1}{n} \left( 1 + a \right)^{n - 1} \left( 1 + b \right)^{n} \Hypergeom{2}{1}{1-n,1}{1+n}{\frac{1+b}{1+a}} \\
&= \left( 1 + a \right)^{n - 1} \left( 1 + b \right)^{n} \frac{1}{n} \frac{\Pochhsymb{n}{n-1}}{\Pochhsymb{1+n}{n-1}} \Hypergeom{2}{1}{1-n,1}{2-2n}{\frac{a-b}{1+a}} \\
&= \left( 1 + a \right)^{n - 1} \left( 1 + b \right)^{n} \frac{1}{2n - 1} \Hypergeom{2}{1}{1-n,1}{2-2n}{\frac{a-b}{1+a}} \\
&= - \left( 1 + a \right)^{n - 1} \left( 1 + b \right)^{n} \sum_{k=0}^{n-1} \frac{\Pochhsymb{1-n}{k}}{\Pochhsymb{1-2n}{k+1}} \left( \frac{a-b}{1+a} \right)^k,
\end{align*}
where it is understood that $\Pochhsymb{0}{0} = 1$. This completes the proof.
\end{proof}

With this preparations we are ready now to proof Theorem~\ref{thm:integral.general.case}.

\begin{proof}[Proof of Theorem~\ref{thm:integral.general.case}]
Let $\beta > 1/2$. The kernel function \eqref{eq:kernel} can be recast as
\begin{equation*}
\mathcal{K}_\beta( \PT{x}, \PT{y} ) = \int_{\mathbb{S}^d} H_\beta( \PT{x} \cdot \PT{z}, \PT{y} \cdot \PT{z} ) \dd \sigma_d( \PT{z} ),
\end{equation*}
where 
\begin{equation*}
H_\beta( a, b ) \DEF \int_{-1}^1 \left( a - t \right)_+^{\beta-1} \left( b - t \right)_+^{\beta-1} \dd t.
\end{equation*}

{\bf Special values $\mathcal{K}_\beta( \PT{x}, \pm \PT{x} )$.} Let $\PT{x}, \PT{y} \in \mathbb{S}^d$ with $\PT{x} = \PT{y}$. Then by Lemma~\ref{lem:aux.res.1} 
\begin{equation*}
H_\beta( \PT{x} \cdot \PT{z}, \PT{x} \cdot \PT{z} ) = \frac{1}{2 \beta - 1} \left( 1 + \PT{x} \cdot \PT{z} \right)^{2\beta - 1}
\end{equation*}
and therefore
\begin{align*}
\mathcal{K}_\beta( \PT{x}, \PT{x} )
&= \int_{\mathbb{S}^d} H_\beta( \PT{x} \cdot \PT{z}, \PT{x} \cdot \PT{z} ) \dd \sigma_d( \PT{z} ) = \frac{1}{2 \beta - 1} \int_{\mathbb{S}^d} \left( 1 + \PT{x} \cdot \PT{z} \right)^{2\beta - 1} \dd \sigma_d( \PT{z} ).
\end{align*}
By Lemma~\ref{lem:aux.res.2} and the duplication formula for the gamma function
\begin{align*}
\mathcal{K}_\beta( \PT{x}, \PT{x} ) 
&= \frac{2^{2\beta + d - 2}}{2\beta-1} \frac{\gammafcn((d+1)/2) \gammafcn(2\beta + d/2-1)}{\sqrt{\pi} \, \gammafcn(2\beta + d -1)} = \frac{1}{2\beta-1} \frac{\gammafcn((d+1)/2) \gammafcn(2\beta + d/2-1)}{\gammafcn(\beta + (d -1)/2) \gammafcn(\beta + d/2)}.
\end{align*}

Let $\PT{x}, \PT{y} \in \mathbb{S}^d$ with $\PT{x} = -\PT{y}$. Then by Lemma~\ref{lem:aux.res.1} 
\begin{equation*}
H_\beta( \PT{x} \cdot \PT{z}, - \PT{x} \cdot \PT{z} ) = \frac{1}{\beta} \left( 1 - | \PT{x} \cdot \PT{z} | \right)^\beta \left( 1 + | \PT{x} \cdot \PT{z} | \right)^{\beta-1} \Hypergeom{2}{1}{1-\beta,1}{1+\beta}{\frac{1 - | \PT{x} \cdot \PT{z} |}{1 + | \PT{x} \cdot \PT{z} |}}
\end{equation*}
is a zonal function that does not depend on the sign of $\PT{x} \cdot \PT{z}$. By the Funk-Hecke formula
\begin{equation*}
\mathcal{K}_\beta( \PT{x}, -\PT{x} ) = 2 \frac{\omega_{d-1}}{\omega_d} \frac{1}{\beta} \int_0^1 \left( 1 - u \right)^\beta \left( 1 + u \right)^{\beta-1} \Hypergeom{2}{1}{1-\beta,1}{1+\beta}{\frac{1 - u}{1 + u}} \left( 1 - u^2 \right)^{d/2-1} \dd u.
\end{equation*}
The change of variables
\begin{equation*}
x = \frac{1-u}{1+u}, \qquad \text{i.e.} \qquad u = \frac{1-x}{1+x}, \quad \dd u = - \frac{2}{(1+x)^2} \dd x,
\end{equation*}
yields 
\begin{align*}
\mathcal{K}_\beta( \PT{x}, -\PT{x} ) 
&= 2 \frac{\omega_{d-1}}{\omega_d} \frac{1}{\beta} \int_0^1 \left( \frac{2}{1+x} \right)^{\beta+d/2-1} \left( \frac{2x}{1+x} \right)^{\beta-1+d/2-1} \Hypergeom{2}{1}{1-\beta,1}{1+\beta}{x} \frac{2}{(1+x)^2} \dd x \\
&= 2^{2\beta+d-1} \frac{\omega_{d-1}}{\omega_d} \frac{1}{\beta} \int_0^1 \frac{x^{\beta-1+d/2-1} \left( 1 - x \right)^{1-1}}{\left( 1 + x \right)^{2\beta+d-1}} \Hypergeom{2}{1}{1-\beta,1}{1+\beta}{x} \dd x.
\end{align*}
By \cite[Eq.~2.21.22]{PrBrMa1990III} and \eqref{eq:C.d}
\begin{align*}
\mathcal{K}_\beta( \PT{x}, -\PT{x} ) 
&= 2^{2\beta+d-1} \frac{\omega_{d-1}}{\omega_d} \frac{1}{\beta} \frac{1}{2^{2\beta+d-1}} \frac{\gammafcn(\beta+1) \gammafcn(1) \gammafcn(2\beta)}{\gammafcn(2\beta-1) \gammafcn( \beta + 1 )} \Hypergeom{3}{2}{1,2\beta+d-1,2\beta}{2\beta+1,\beta+1}{\frac{1}{2}} \\
&= \frac{\gammafcn((d+1)/2)}{\sqrt{\pi} \, \gammafcn(d/2)} \frac{2\beta-1}{\beta} \Hypergeom{3}{2}{1,2\beta+d-1,2\beta}{2\beta+1,\beta+1}{\frac{1}{2}}.
\end{align*}

For $d=2$ the $_3\HyperF_2$-hypergeometric function reduces to a $_2\HyperF_1$ function that simplifies further as follows (see \cite[Eq.~15.4.28]{DLMF2011.08.29})
\begin{equation*}
\begin{split}
\mathcal{K}_\beta( \PT{x}, -\PT{x} ) 
&= \frac{\gammafcn((d+1)/2)}{\sqrt{\pi} \, \gammafcn(d/2)} \frac{2\beta-1}{\beta} \Hypergeom{2}{1}{1,2\beta}{\beta+1}{\frac{1}{2}} = \frac{\gammafcn(3/2)}{\sqrt{\pi} \, \gammafcn(1)} \frac{2\beta-1}{\beta} \sqrt{\pi} \frac{\gammafcn(\beta+1)}{\gammafcn(1) \gammafcn( \beta+1/2)} \\
&= \sqrt{\pi} \frac{\gammafcn(\beta)}{\gammafcn( \beta-1/2)}.
\end{split}
\end{equation*}

{\bf Integral representation.} Let $\PT{x}, \PT{y} \in \mathbb{S}^d$ with $\PT{x} \neq \PT{y}$.

We introduce cylinder coordinates with respect to the Pole $\PT{p} = ( \PT{x} - \PT{y} ) / \| \PT{x} - \PT{y} \|$. Then
\begin{align*}
\PT{x} \cdot \PT{p} &= \frac{1 - \PT{x} \cdot \PT{y}}{\| \PT{x} - \PT{y} \|} = \frac{1}{2} \frac{2 \left( 1 - \PT{x} \cdot \PT{y} \right)}{\| \PT{x} - \PT{y} \|} =  \frac{1}{2} \frac{\| \PT{x} - \PT{y} \|^2}{\| \PT{x} - \PT{y} \|} = \frac{1}{2} \left\| \PT{x} - \PT{y} \right\| \FED v, \\
\PT{y} \cdot \PT{p} &= \frac{\PT{y} \cdot \PT{x} - 1}{\| \PT{x} - \PT{y} \|} = - v
\end{align*}
and one can write
\begin{align*}
\PT{x} &= ( \sqrt{1 - v^2} \, \PT{x}^*, v), \qquad \PT{x}^* \in \mathbb{S}^{d-1}, \\
\PT{y} &= ( \sqrt{1 - v^2} \, \PT{y}^*, -v), \qquad \PT{y}^* \in \mathbb{S}^{d-1}, \\
\PT{z} &= ( \sqrt{1 - u^2} \, \PT{z}^*, u), \qquad \PT{z}^* \in \mathbb{S}^{d-1}, -1 \leq u \leq 1.
\end{align*}
By construction of these cylinder coordinates there holds that $\PT{x}^* = \PT{y}^*$. It follows that
\begin{align*}
\PT{x} \cdot \PT{z} &= u v + \sqrt{1 - v^2} \, \sqrt{1 - u^2} \, \PT{x}^* \cdot \PT{z}^*, \\
\PT{y} \cdot \PT{z} &= - u v + \sqrt{1 - v^2} \, \sqrt{1 - u^2} \, \PT{x}^* \cdot \PT{z}^*.
\end{align*}
From the definitions of the these inner products there follows that %$\PT{x} \cdot \PT{z}_{(-u,\PT{z}^*)} = \PT{y} \cdot \PT{z}_{(u,\PT{z}^*)}$. 
\begin{equation*}
\PT{x} \cdot \PT{z}_{(-u,\PT{z}^*)} = \PT{y} \cdot \PT{z}_{(u,\PT{z}^*)}, \qquad \text{for all $-1 \leq u \leq 1$ and $\PT{z}^* \in \mathbb{S}^{d-1}$.}
\end{equation*}

Consequently
\begin{equation*}
\PT{x} \cdot \PT{z} \geq \PT{y} \cdot \PT{z} \qquad \text{iff} \qquad u v \geq - u v \qquad \text{iff} \qquad 2 u v \geq 0 \qquad \text{iff} \qquad u \geq 0.
\end{equation*}
In other words, one has $\PT{x} \cdot \PT{z} \geq \PT{y} \cdot \PT{z}$ for every $\PT{z}$ in the upper half sphere ($u \geq 0$) and one has $\PT{x} \cdot \PT{z} \leq \PT{y} \cdot \PT{z}$ for every $\PT{z}$ in the lower half sphere ($u \leq 0$). Hence, by Lemma~\ref{lem:aux.res.1}
\begin{align*}
H_\beta( \PT{x} \cdot \PT{z}, \PT{y} \cdot \PT{z} ) 
&= \frac{1}{\beta} \left( 1 + \min\{\PT{x} \cdot \PT{z}, \PT{y} \cdot \PT{z}\} \right) \left( 1 + \PT{x} \cdot \PT{z} \right)^{\beta - 1} \left( 1 + \PT{y} \cdot \PT{z} \right)^{\beta - 1} \\
&\phantom{=\pm}\times \Hypergeom{2}{1}{1-\beta,1}{1+\beta}{\frac{1+\min\{\PT{x} \cdot \PT{z},\PT{y} \cdot \PT{z}\}}{1+\max\{\PT{x} \cdot \PT{z},\PT{y} \cdot \PT{z}\}}} \\
&= \frac{1}{\beta} \left( 1 + \PT{x} \cdot \PT{z} \right)^{\beta - 1} \left( 1 + \PT{y} \cdot \PT{z} \right)^{\beta} \Hypergeom{2}{1}{1-\beta,1}{1+\beta}{\frac{1+\PT{y} \cdot \PT{z}}{1+\PT{x} \cdot \PT{z}}}, \qquad u \geq 0, \\
H_\beta( \PT{x} \cdot \PT{z}, \PT{y} \cdot \PT{z} ) 
&= \frac{1}{\beta} \left( 1 + \PT{x} \cdot \PT{z} \right)^{\beta} \left( 1 + \PT{y} \cdot \PT{z} \right)^{\beta - 1} \Hypergeom{2}{1}{1-\beta,1}{1+\beta}{\frac{1+\PT{x} \cdot \PT{z}}{1+\PT{y} \cdot \PT{z}}}, \qquad u \leq 0.
\end{align*}

By the Funk-Hecke formula (note that $\PT{z}$ depends on $u$ and $\PT{z}^*$)
\begin{equation*}
\mathcal{K}_\beta( \PT{x}, \PT{y} ) = \frac{\omega_{d-1}}{\omega_d} \int_{-1}^1 \int_{\mathbb{S}^{d-1}} H_\beta( \PT{x} \cdot \PT{z}, \PT{y} \cdot \PT{z} ) \left( 1 - u^2 \right)^{d/2-1} \dd \sigma_{d-1}(\PT{z}^*) \dd u.
\end{equation*}
From the definitions of the inner products $\PT{x} \cdot \PT{z}$ and $\PT{y} \cdot \PT{z}$ it follows that $\PT{x} \cdot \PT{z}_{(-u,\PT{z}^*)} = \PT{y} \cdot \PT{z}_{(u,\PT{z}^*)}$. Hence
\begin{equation} \label{eq:cal.K.beta}
\mathcal{K}_\beta( \PT{x}, \PT{y} ) = 2 \frac{\omega_{d-1}}{\omega_d} \int_{0}^1 \int_{\mathbb{S}^{d-1}} H_\beta( \PT{x} \cdot \PT{z}, \PT{y} \cdot \PT{z} ) \left( 1 - u^2 \right)^{d/2-1} \dd \sigma_{d-1}(\PT{z}^*) \dd u.
\end{equation}

We compute the following integral
\begin{equation*}
h_\beta( u ) \DEF \int_{\mathbb{S}^{d-1}} H_\beta( \PT{x} \cdot \PT{z}, \PT{y} \cdot \PT{z} ) \dd \sigma_{d-1}(\PT{z}^*)
\end{equation*}
By our previous considerations (recall, that $\PT{z} = \PT{z}(u,\PT{z}^*)$)
\begin{equation*}
h_\beta( u ) = \int_{\mathbb{S}^{d-1}} \frac{1}{\beta} \left( 1 + \PT{x} \cdot \PT{z} \right)^{\beta - 1} \left( 1 + \PT{y} \cdot \PT{z} \right)^{\beta} \Hypergeom{2}{1}{1-\beta,1}{1+\beta}{\frac{1+\PT{y} \cdot \PT{z}}{1+\PT{x} \cdot \PT{z}}} \dd \sigma_{d-1}(\PT{z}^*), \quad u \geq 0.
\end{equation*}
We observe, that in the case $\PT{y} = - \PT{x}$, $\PT{x} \in \mathbb{S}^d$, the integrand above does not depend on $\PT{z}^*$; that is
\begin{equation} \label{eq:h.beta.special}
h_\beta( u ) \Big|_{\PT{y} = - \PT{x}} = \frac{1}{\beta} \left( 1 + u \right)^{\beta - 1} \left( 1 - u \right)^{\beta} \Hypergeom{2}{1}{1-\beta,1}{1+\beta}{\frac{1-u}{1+u}}, \qquad u \geq 0.
\end{equation}

For the transformation of variable in the hypergeometric functions above we have to distinguish between three cases: {\bf (a)} $2\beta-1$ is not a positive integer ($\beta > 1/2$ by assumption), {\bf (b)} $2\beta-1$ is an even positive integer, and {\bf (c)} $\beta$ is a positive integer. Further observe that 
\begin{equation*}
0 \leq \frac{1+\PT{y} \cdot \PT{z}}{1+\PT{x} \cdot \PT{z}} < 1 \qquad \text{for $0 < u \leq 1$.}
\end{equation*}
(The lower bound is assumed when $\PT{y} \cdot \PT{z} = - 1$; that is, if and only if $\PT{z} = ( - \sqrt{1-v^2} \PT{x}^*, v)$.) 

{\bf Case (a).} Let $2\beta-1$ be not a positive integer. First, let $\PT{y} = - \PT{x}$; that is, $v = 1$. Application of Lemma~\ref{lem:aux.res.3} to \eqref{eq:h.beta.special} and the quadratic transformation \cite[Eq.~15.8.1]{DLMF2011.08.29} gives
\begin{align}
h_\beta( u ) \Big|_{\PT{y} = - \PT{x}} 
&= \frac{\gammafcn( 1 - 2 \beta ) \gammafcn(\beta)}{\gammafcn(1 - \beta)} \left( 2 u \right)^{2\beta-1} + \frac{\left( 1 + u \right)^{2\beta - 1}}{2\beta-1} \Hypergeom{2}{1}{1-2\beta,1-\beta}{2-2\beta}{\frac{2u}{1+u}} \notag \\
&= \frac{\gammafcn( 1 - 2 \beta ) \gammafcn(\beta)}{\gammafcn(1 - \beta)} \left( 2 u \right)^{2\beta-1} + \frac{1}{2\beta-1} \Hypergeom{2}{1}{1/2-\beta,1-\beta}{3/2-\beta}{u^2}, \qquad u > 0. \label{eq:h.beta.anti.diagonal}
\end{align}

Hence, using the series expansion of hypergeometric functions
\begin{align*}
\mathcal{K}_\beta( \PT{x}, - \PT{x} ) 
&= 2 \frac{\omega_{d-1}}{\omega_d} \int_0^1 h_\beta( u ) \Big|_{\PT{y} = - \PT{x}} \left( 1 - u^2 \right)^{d/2-1} \dd u \\
&= 2^{2\beta-1} \frac{\gammafcn( 1 - 2 \beta ) \gammafcn(\beta)}{\gammafcn(1 - \beta)} 2 \frac{\omega_{d-1}}{\omega_d} \int_0^1 u^{2\beta-1} \left( 1 - u^2 \right)^{d/2-1} \dd u \\
&\phantom{=}+ \frac{1}{2\beta-1} 2 \frac{\omega_{d-1}}{\omega_d} \sum_{n=0}^\infty \frac{\Pochhsymb{1/2-\beta}{n} \Pochhsymb{1-\beta}{n}}{\Pochhsymb{3/2-\beta}{n} n!} \int_0^1 u^{2n} \left( 1 - u^2 \right)^{d/2-1} \dd u.
\end{align*}
The substitution $u = x^{1/2}$ in the integrals yields
\begin{equation} \label{eq:integral.aux.1}
\int_0^1 u^{2\alpha} \left( 1 - u^2 \right)^{d/2-1} \dd u = \frac{1}{2} \int_0^1 x^{\alpha-1/2} \left( 1 - x \right)^{d/2-1} \dd x = \frac{1}{2} \frac{\gammafcn( \alpha + 1/2 ) \gammafcn( d/2 )}{\gammafcn( \alpha + (d+1)/2 )}.
\end{equation}
Therefore (using \eqref{eq:C.d})
\begin{align*}
\mathcal{K}_\beta( \PT{x}, - \PT{x} ) 
&= 2^{2\beta-1} \frac{\gammafcn( 1 - 2 \beta ) \gammafcn(\beta)}{\gammafcn(1 - \beta)} \frac{\gammafcn((d+1)/2)}{\sqrt{\pi} \, \gammafcn(d/2)} \frac{\gammafcn( \beta ) \gammafcn( d/2 )}{\gammafcn( \beta + d/2 )} \\
&\phantom{=}+ \frac{1}{2\beta-1} \frac{\gammafcn((d+1)/2)}{\sqrt{\pi} \, \gammafcn(d/2)} \sum_{n=0}^\infty \frac{\Pochhsymb{1/2-\beta}{n} \Pochhsymb{1-\beta}{n}}{\Pochhsymb{3/2-\beta}{n} n!} \frac{\gammafcn( n + 1/2 ) \gammafcn( d/2 )}{\gammafcn( n + (d+1)/2 )} \\
&= 2^{2\beta-1} \frac{\gammafcn((d+1)/2)}{\sqrt{\pi}} \frac{\gammafcn( 1 - 2 \beta ) \gammafcn(\beta) \gammafcn( \beta )}{\gammafcn(1 - \beta) \gammafcn( \beta + d/2 )} \\
&\phantom{=}+ \frac{1}{2\beta-1} \sum_{n=0}^\infty \frac{\Pochhsymb{1/2-\beta}{n} \Pochhsymb{1-\beta}{n}}{\Pochhsymb{3/2-\beta}{n} n!} \frac{\Pochhsymb{1/2}{n}}{\Pochhsymb{(d+1)/2}{n}}.
\end{align*}
Comparison with \eqref{eq:cal.K.beta.1} below and changing to hypergeometric function yields
\begin{equation*}
\mathcal{K}_\beta( \PT{x}, - \PT{x} ) = 2^{2\beta-1} \tilde{c}_d(\beta) + \frac{1}{2\beta-1} \Hypergeom{3}{1}{1/2-\beta, 1 - \beta, 1/2}{3/2 - \beta, (d+1)/2}{1}.
\end{equation*}
Taking into account \eqref{eq:tilde.c.d} below, equation \eqref{eq:K.beta.x.neg.x} follows.

Let $\PT{y} \neq - \PT{x}$ and $\PT{y} \neq \PT{x}$; that is $0 < v < 1$. By Lemma~\ref{lem:aux.res.3}
\begin{equation}
\begin{split} \label{eq:inner.integral}
h_\beta( u ) 
&= \frac{\gammafcn( 1 - 2 \beta ) \gammafcn(\beta)}{\gammafcn(1 - \beta)} \int_{\mathbb{S}^{d-1}} \left( \PT{x} \cdot \PT{z} - \PT{y} \cdot \PT{z} \right)^{2\beta-1} \dd \sigma_{d-1}(\PT{z}^*) \\
&\phantom{=}+ \int_{\mathbb{S}^{d-1}} \frac{\left( 1 + \PT{x} \cdot \PT{z} \right)^{2\beta - 1}}{2\beta-1} \Hypergeom{2}{1}{1-\beta,1-2\beta}{2-2\beta}{\frac{\PT{x} \cdot \PT{z}-\PT{y} \cdot \PT{z}}{1+\PT{x} \cdot \PT{z}}} \dd \sigma_{d-1}(\PT{z}^*), \quad u > 0.
\end{split}
\end{equation}

Because of our parameterization, the first integral can be evaluated as follows:
\begin{align*}
h_{\beta,1}( u ) \Bigg/ \frac{\gammafcn( 1 - 2 \beta ) \gammafcn(\beta)}{\gammafcn(1 - \beta)} 
&\DEF \int_{\mathbb{S}^{d-1}} \left( \PT{x} \cdot \PT{z} - \PT{y} \cdot \PT{z} \right)^{2\beta-1} \dd \sigma_{d-1}(\PT{z}^*) \\
&= \left\| \PT{x} - \PT{y} \right\|^{2\beta-1} \int_{\mathbb{S}^{d-1}} \left( \frac{\PT{x} - \PT{y}}{\| \PT{x} - \PT{y} \|} \cdot \PT{z} \right)^{2\beta-1} \dd \sigma_{d-1}(\PT{z}^*) \\
&= \left\| \PT{x} - \PT{y} \right\|^{2\beta-1} \int_{\mathbb{S}^{d-1}} \left( \PT{p} \cdot \PT{z} \right)^{2\beta-1} \dd \sigma_{d-1}(\PT{z}^*) \\
&= \left\| \PT{x} - \PT{y} \right\|^{2\beta-1} \int_{\mathbb{S}^{d-1}} u^{2\beta-1} \dd \sigma_{d-1}(\PT{z}^*) \\
&= \left\| \PT{x} - \PT{y} \right\|^{2\beta-1} u^{2\beta-1}.
\end{align*}

The second integral is, in fact, well-defined for $\re \beta > 1/2$ and $2-2\beta$ not an integer $\leq 0$. Moreover, the hypergeometric function in the integrand has an integral representation for $1/2 < \re \beta < 1$ (cf. \cite[Eq.~15.6.1]{DLMF2011.08.29}). Let $1/2 < \re \beta < 1$. Thus, by the Funk-Hecke formula and the substitution $2 x = 1 - \tau$ 
\begin{align*}
h_{\beta,2}( u ) 
&\DEF \int_{\mathbb{S}^{d-1}} \frac{\left( 1 + \PT{x} \cdot \PT{z} \right)^{2\beta - 1}}{2\beta-1} \Hypergeom{2}{1}{1-\beta,1-2\beta}{2-2\beta}{\frac{\PT{x} \cdot \PT{z}-\PT{y} \cdot \PT{z}}{1+\PT{x} \cdot \PT{z}}} \dd \sigma_{d-1}(\PT{z}^*) \\
&= - \frac{\gammafcn( 1 - 2 \beta )}{\gammafcn( 1 - \beta ) \gammafcn( 1 - \beta )} \int_{\mathbb{S}^{d-1}} \int_0^1 t^{-\beta} \left( 1 - t \right)^{-\beta} \left( 1 + \PT{x} \cdot \PT{z} - 2 u v \, t \right)^{2\beta - 1} \dd t \dd \sigma_{d-1}(\PT{z}^*) \\
&= - \frac{\gammafcn( 1 - 2 \beta )}{\gammafcn( 1 - \beta ) \gammafcn( 1 - \beta )} \frac{\omega_{d-2}}{\omega_{d-1}} \int_{-1}^1 \int_0^1 \frac{t^{-\beta} \left( 1 - t \right)^{-\beta} \left( 1 - \tau^2 \right)^{(d-1)/2-1}}{\left( 1 + u v + \sqrt{1 - v^2} \, \sqrt{1 - u^2} \, \tau - 2 u v \, t \right)^{1 - 2\beta} } \dd t \dd \tau \\
&= \frac{1}{2\beta-1} \, \frac{\gammafcn( 2 - 2 \beta )}{\gammafcn( 1 - \beta ) \gammafcn( 1 - \beta )} 2^{d-2} \frac{\omega_{d-2}}{\omega_{d-1}} \left( 1 + u v + \sqrt{1 - v^2} \, \sqrt{1 - u^2} \right)^{2\beta-1} \\
&\phantom{=\pm}\times \int_0^1 \int_0^1 \frac{x^{(d-1)/2-1} t^{1-\beta-1} \left( 1 - x \right)^{(d-1)/2-1} \left( 1 - t \right)^{1-\beta-1}}{\left( 1 - w \, x - z \, t \right)^{1-2\beta}} \dd x \dd t,
\end{align*}
where we set 
\begin{equation*}
w = \frac{2 \sqrt{1 - v^2} \, \sqrt{1 - u^2}}{1 + u v + \sqrt{1 - v^2} \, \sqrt{1 - u^2}}, \qquad z = \frac{2uv}{1 + u v + \sqrt{1 - v^2} \, \sqrt{1 - u^2}}.
\end{equation*}
We note that 
\begin{equation*}
| w | + | z | = w + z = \frac{2 u v + 2 \sqrt{1 - v^2} \, \sqrt{1 - u^2}}{1 + u v + \sqrt{1 - v^2} \, \sqrt{1 - u^2}} < 1 \qquad \text{iff} \qquad \PT{x} \cdot \PT{z} < 1.
\end{equation*}
For $d\geq 2$ and $1/2 < \re \beta < 1$ the right-most integral above represents an Appell $\HyperF_2$-function (cf. \cite[Eq.~16.15.2]{DLMF2011.08.29}). Moreover, invoking \eqref{eq:C.d}, we see that 
\begin{align*}
h_{\beta,2}( u ) 
&= \frac{1}{2\beta-1} \frac{\gammafcn( 2 - 2 \beta )}{\gammafcn( 1 - \beta ) \gammafcn( 1 - \beta )} 2^{d-2} \frac{\gammafcn(d/2)}{\sqrt{\pi} \, \gammafcn((d-1)/2)} \left( 1 + u v + \sqrt{1 - v^2} \, \sqrt{1 - u^2} \right)^{2\beta-1} \\
&\phantom{=\pm}\times \frac{\gammafcn((d-1)/2) \gammafcn(1-\beta) \gammafcn((d-1)/2) \gammafcn(1-\beta)}{\gammafcn(d-1) \gammafcn(2-2\beta)} \, \AppellFtwo{1-2\beta, (d-1)/2, 1 - \beta}{d-1, 2 - 2 \beta}{w,z}.
\end{align*}
Simplification gives
\begin{equation*}
h_{\beta,2}( u ) = \frac{1}{2\beta-1} \left( 1 + u v + \sqrt{1 - v^2} \, \sqrt{1 - u^2} \right)^{2\beta-1} \AppellFtwo{1-2\beta, (d-1)/2, 1 - \beta}{d-1, 2 - 2 \beta}{w,z}.
\end{equation*}
By analytic continuation this identity holds for all $\beta > 1/2$ except $\beta$ is a positive integer. (It is assumed that $\PT{x} \cdot \PT{z} < 1$.)

Letting $a = 2 \sqrt{1 - v^2} \, \sqrt{1 - u^2}$, $b = 2 uv$, $A = 1 + u v + \sqrt{1 - v^2} \, \sqrt{1 - u^2}$ and observing that $2 A - a - b = 2$, Proposition~\ref{prop:Appell.F.identity} yields the following representation in terms of an Appell $\HyperF_4$ function:
\begin{equation} \label{eq:h.beta.2}
h_{\beta,2}( u ) = \frac{1}{2\beta-1} \AppellFfour{1/2-\beta,1-\beta}{d/2,3/2-\beta}{\left( 1 - v^2 \right) \left( 1 - u^2 \right), v^2 u^2}.
\end{equation}

Putting everything together, we arrive at
\begin{equation*}
\begin{split}
h_\beta( u ) 
&= \frac{\gammafcn( 1 - 2 \beta ) \gammafcn(\beta)}{\gammafcn(1 - \beta)} \left\| \PT{x} - \PT{y} \right\|^{2\beta-1} u^{2\beta-1} \\
&\phantom{=}+ \frac{1}{2\beta-1} \AppellFfour{1/2-\beta,1-\beta}{d/2,3/2-\beta}{\left( 1 - v^2 \right) \left( 1 - u^2 \right), v^2 u^2}.
\end{split}
\end{equation*}

Substitution into \eqref{eq:cal.K.beta} yields
\begin{equation*}
\mathcal{K}_\beta( \PT{x}, \PT{y} ) = \mathcal{K}_{\beta,1}( \PT{x}, \PT{y} ) + \mathcal{K}_{\beta,2}( \PT{x}, \PT{y} ),
\end{equation*}
where (using \eqref{eq:C.d} and functional relations for the gamma function)
\begin{align}
\mathcal{K}_{\beta,1}( \PT{x}, \PT{y} ) 
&= 2 \frac{\omega_{d-1}}{\omega_d} \frac{\gammafcn( 1 - 2 \beta ) \gammafcn(\beta)}{\gammafcn(1 - \beta)} \left\| \PT{x} - \PT{y} \right\|^{2\beta-1} \int_{0}^1 u^{2\beta-1} \left( 1 - u^2 \right)^{d/2-1} \dd u \notag \\
&= 2 \frac{\omega_{d-1}}{\omega_d} \frac{\gammafcn( 1 - 2 \beta ) \gammafcn(\beta)}{\gammafcn(1 - \beta)} \left\| \PT{x} - \PT{y} \right\|^{2\beta-1} \frac{1}{2} \int_0^1 x^{\beta-1} \left( 1 - x \right)^{d/2-1} \dd x \notag \\
&= \frac{\gammafcn((d+1)/2)}{\sqrt{\pi} \, \gammafcn(d/2)} \frac{\gammafcn( 1 - 2 \beta ) \gammafcn(\beta)}{\gammafcn(1 - \beta)} \frac{\gammafcn(\beta) \gammafcn(d/2)}{\gammafcn(\beta+d/2)} \left\| \PT{x} - \PT{y} \right\|^{2\beta-1} \notag \\
&= \frac{\gammafcn((d+1)/2)}{\sqrt{\pi}} \frac{\gammafcn( 1 - 2 \beta )}{\gammafcn(1 - \beta)} \frac{\gammafcn(\beta)\gammafcn(\beta)}{\gammafcn(\beta+d/2)} \left\| \PT{x} - \PT{y} \right\|^{2\beta-1} \notag \\
&= \tilde{c}_d(\beta) \left\| \PT{x} - \PT{y} \right\|^{2\beta-1}, \qquad \tilde{c}_d(\beta) \DEF 2^{-2\beta} \frac{\gammafcn((d+1)/2)}{\pi} \frac{\gammafcn(1/2-\beta) \gammafcn(\beta)\gammafcn(\beta)}{\gammafcn(\beta+d/2)}, \label{eq:cal.K.beta.1}
\end{align}
and (recall \eqref{eq:inner.integral})
\begin{equation*}
\begin{split}
\mathcal{K}_{\beta,2}( \PT{x}, \PT{y} ) 
&= 2 \frac{\omega_{d-1}}{\omega_d} \int_0^1 \int_{\mathbb{S}^{d-1}} \frac{\left( 1 + \PT{x} \cdot \PT{z} \right)^{2\beta - 1}}{2\beta-1} \Hypergeom{2}{1}{1-\beta,1-2\beta}{2-2\beta}{\frac{\PT{x} \cdot \PT{z}-\PT{y} \cdot \PT{z}}{1+\PT{x} \cdot \PT{z}}} \\
&\phantom{=\pm}\times \left( 1 - u^2 \right)^{d/2-1} \dd \sigma_{d-1}(\PT{z}^*) \dd t.
\end{split}
\end{equation*}
A linear transformation of variable for hypergeometric functions \cite[Eq.~15.8.1]{DLMF2011.08.29} yields
\begin{equation*}
\begin{split}
&\frac{\left( 1 + \PT{x} \cdot \PT{z} \right)^{2\beta - 1}}{2\beta-1} \Hypergeom{2}{1}{1-\beta,1-2\beta}{2-2\beta}{\frac{\PT{x} \cdot \PT{z}-\PT{y} \cdot \PT{z}}{1+\PT{x} \cdot \PT{z}}} \\
&\phantom{equals}= \frac{\left( 1 + \PT{y} \cdot \PT{z} \right)^{2\beta - 1}}{2\beta-1} \Hypergeom{2}{1}{1-\beta,1-2\beta}{2-2\beta}{\frac{\PT{y} \cdot \PT{z}-\PT{x} \cdot \PT{z}}{1+\PT{y} \cdot \PT{z}}}.
\end{split}
\end{equation*}
Invoking the fact that $\PT{x} \cdot \PT{z}_{(-u,\PT{z}^*)} = \PT{y} \cdot \PT{z}_{(u,\PT{z}^*)}$ for all $-1 \leq u \leq 1$, we arrive at
\begin{equation*}
\begin{split}
\mathcal{K}_{\beta,2}( \PT{x}, \PT{y} ) 
&= \frac{\omega_{d-1}}{\omega_d} \int_{-1}^1 \int_{\mathbb{S}^{d-1}} \frac{\left( 1 + \PT{x} \cdot \PT{z} \right)^{2\beta - 1}}{2\beta-1} \Hypergeom{2}{1}{1-\beta,1-2\beta}{2-2\beta}{\frac{\PT{x} \cdot \PT{z}-\PT{y} \cdot \PT{z}}{1+\PT{x} \cdot \PT{z}}} \\
&\phantom{=\pm}\times \left( 1 - u^2 \right)^{d/2-1} \dd \sigma_{d-1}(\PT{z}^*) \dd t.
\end{split}
\end{equation*}
Reversing application of the Funk-Hecke formula yields
\begin{equation}
\mathcal{K}_{\beta,2}( \PT{x}, \PT{y} ) = \int_{\mathbb{S}^d} \frac{\left( 1 + \PT{x} \cdot \PT{z} \right)^{2\beta - 1}}{2\beta-1} \Hypergeom{2}{1}{1-\beta,1-2\beta}{2-2\beta}{\frac{\PT{x} \cdot \PT{z}-\PT{y} \cdot \PT{z}}{1+\PT{x} \cdot \PT{z}}} \dd \sigma_{d}(\PT{z}).
\end{equation}
The quadratic transformation of variable \cite[Eq.~15.8.1]{DLMF2011.08.29} and a change of variable $\PT{z} \mapsto - \PT{z}$ give the symmetric form
\begin{equation} \label{eq:K.beta.2.symmetric.form}
\mathcal{K}_{\beta,2}( \PT{x}, \PT{y} ) = \frac{2^{1-2\beta}}{2\beta-1} \int_{\mathbb{S}^d} \frac{\Hypergeom{2}{1}{1/2-\beta,1-\beta}{3/2-\beta}{\left( \frac{\PT{x} \cdot \PT{z}-\PT{y} \cdot \PT{z}}{2 - \PT{x} \cdot \PT{z} - \PT{y} \cdot \PT{z}} \right)^2}}{\left( 2 - \PT{x} \cdot \PT{z} - \PT{y} \cdot \PT{z} \right)^{1-2\beta}} \dd \sigma_{d}(\PT{z}).
\end{equation}

The constant in front of the Euclidean distance in \eqref{eq:cal.K.beta.1} can be rewritten as follows. Let $L = \lfloor \beta - 1/2 \rfloor$ be the largest integer $\leq \beta - 1/2$. Then $\beta - 1/2 = L + \eps$ for some $0 \leq \eps < 1$. Application of the reflection formula for the Gamma function gives
\begin{equation*}
\gammafcn(1/2-\beta) = \frac{\gammafcn(\beta+1/2) \gammafcn(1/2-\beta)}{\gammafcn(\beta+1/2)} = \frac{1}{\gammafcn(\beta+1/2)} \frac{\pi}{\sin( \pi ( \beta+1/2 ) )} = \frac{(-1)^{L+1}}{\gammafcn(\beta+1/2)} \frac{\pi}{\sin( \pi \eps )}
\end{equation*}
and therefore
\begin{equation} \label{eq:tilde.c.d}
\tilde{c}_d(\beta) = \frac{(-1)^{L+1}}{\sin( \pi (\beta-1/2-L) )} \frac{2^{1-2\beta}}{2\beta-1} \frac{\gammafcn((d+1)/2) \gammafcn(\beta)\gammafcn(\beta)}{\gammafcn(\beta+d/2) \gammafcn(\beta-1/2)}.
\end{equation}

Using the series representation of the Appell $\HyperF_4$ function (cf. Proof of Proposition~\ref{prop:Appell.F.identity}) in \eqref{eq:h.beta.2}, we get 
\begin{align*}
\mathcal{K}_{\beta,2}( \PT{x}, \PT{y} )
&= \frac{2}{2\beta-1} \frac{\omega_{d-1}}{\omega_d} \int_0^1 h_{\beta,2}(u) \left( 1 - u^2 \right)^{d/2-1} \dd u \\
&= \frac{2}{2\beta-1} \frac{\omega_{d-1}}{\omega_d} \sum_{m=0}^\infty \sum_{n=0}^\infty \frac{\Pochhsymb{1/2 - \beta}{m+n} \Pochhsymb{1-\beta}{m+n}}{\Pochhsymb{d/2}{m} \Pochhsymb{3/2-\beta}{n} \, m! n!} \left( 1 - v^2 \right)^m v^{2n} \, b_{m,n},
\end{align*}
where
\begin{equation*}
b_{m,n} \DEF \int_0^1 u^{2n} \left( 1 - u^2 \right)^{m+d/2-1} \dd u.
\end{equation*}
The change of variable $x = u^2$ gives a beta integral. Hence
\begin{align*}
b_{m,n} 
&= \frac{1}{2} \int_0^1 x^{n+1/2-1} \left( 1 - x \right)^{m+d/2-1} \dd x = \frac{1}{2} \betafcn(n+1/2,m+d/2) \\
&= \frac{1}{2} \frac{\gammafcn(n+1/2) \gammafcn(m+d/2)}{\gammafcn(m+n+(d+1)/2)} = \frac{1}{2} \frac{\gammafcn(1/2) \gammafcn(d/2)}{\gammafcn((d+1)/2)} \frac{\Pochhsymb{1/2}{n} \Pochhsymb{d/2}{m}}{\Pochhsymb{(d+1)/2}{m+n}}.
\end{align*}
Taking into account \eqref{eq:C.d}, we arrive at
\begin{equation} \label{eq:K.beta.2.first}
\mathcal{K}_{\beta,2}( \PT{x}, \PT{y} ) = \frac{1}{2\beta-1} \sum_{m=0}^\infty \sum_{n=0}^\infty \frac{\Pochhsymb{1/2 - \beta}{m+n} \Pochhsymb{1-\beta}{m+n} \Pochhsymb{1/2}{n}}{\Pochhsymb{(d+1)/2}{m+n} \Pochhsymb{3/2-\beta}{n} \, m! n!} \left( 1 - v^2 \right)^m v^{2n}.
\end{equation}
The double series expansion has the structure given in \eqref{eq:Kampe.de.Feriet.function}. (This gives the first part of the theorem (for the case $2\beta - 1$ is not an integer $\geq 0$).)

Reordering the sum with respect to constant $m+n$ gives
\begin{equation} \label{eq:K.beta.2.2ndform}
\mathcal{K}_{\beta,2}( \PT{x}, \PT{y} ) = \frac{1}{2\beta-1} \sum_{n=0}^\infty \frac{\Pochhsymb{1/2 - \beta}{n} \Pochhsymb{1-\beta}{n}}{\Pochhsymb{(d+1)/2}{n} \, n!} \left[ \sum_{n = 0}^n \binom{n}{n} \frac{\Pochhsymb{1/2}{n}}{\Pochhsymb{3/2-\beta}{n}} v^{2n} \left( 1 - v^2 \right)^{n -n} \right].
\end{equation}
Using \cite[Eq.s~15.8.1 and 15.8.7]{DLMF2011.08.29}, the square-bracketed equals
\begin{equation*}
\left( 1 - v^2 \right)^n \Hypergeom{2}{1}{-n,1/2}{3/2-\beta}{\frac{v^2}{v^2-1}} = \Hypergeom{2}{1}{-n,1-\beta}{3/2-\beta}{v^2} = \frac{\Pochhsymb{1/2}{n}}{\Pochhsymb{3/2-\beta}{n}} \Hypergeom{2}{1}{-n, 1 - \beta}{1/2-n}{1-v^2}.
\end{equation*}

By definition of $v$, we obtain
\begin{equation} \label{eq:K.beta.2.second}
\mathcal{K}_{\beta,2}( \PT{x}, \PT{y} ) = \frac{1}{2\beta-1} \sum_{n=0}^\infty \frac{\Pochhsymb{1/2 - \beta}{n} \Pochhsymb{1-\beta}{n}}{\Pochhsymb{(d+1)/2}{n} \, n!} \Hypergeom{2}{1}{-n,1-\beta}{3/2-\beta}{\frac{1 - \PT{x} \cdot \PT{y}}{2}}.
\end{equation}

We conclude {\bf Case (a)} by deriving the expansion of the kernel $\mathcal{K}_\beta(\PT{x},\PT{y})$ in terms of spherical harmonics. Evidently, this kernel is a zonal function as it depends on the inner product $\PT{x} \cdot \PT{y}$. Thus, the expansion takes on the form of a series in terms of ultraspherical polynomials. Since the expansion of the (signed) power of the Euclidean distance is well-known, we focus here on the expansion of $\mathcal{K}_{\beta,2}( \PT{x}, \PT{y} )$.

By Proposition~\ref{prop:hypergeometric.polynomial.expansion.specialized}
\begin{align*}
\mathcal{K}_{\beta,2}( \PT{x}, \PT{y} ) 
&= \frac{1}{2\beta-1} \sum_{n=0}^\infty \sum_{k=0}^n \frac{\Pochhsymb{1/2 - \beta}{n} \Pochhsymb{1-\beta}{n}}{\Pochhsymb{(d+1)/2}{n} \, n!} \frac{a_{k,n}}{Z(d,k)} \, Z(d,k) P_k^{(d)}(\PT{x} \cdot \PT{y}) \\
&= \sum_{k=0}^{\infty} A_{k}^{(2)}(s,d) \, Z(d,k) P_k^{(d)}(\PT{x} \cdot \PT{y}),
\end{align*}
where (using $\beta$ where it is convenient)
\begin{align*}
A_{k}^{(2)}(s,d)
&\DEF \frac{1}{2\beta-1} \sum_{n=k}^\infty \frac{\Pochhsymb{1/2 - \beta}{n} \Pochhsymb{1-\beta}{n}}{\Pochhsymb{(d+1)/2}{n} \, n!} \frac{a_{k,n}}{Z(d,k)}
\end{align*}
and 
\begin{equation*}
\frac{a_{k,n}}{Z(d,k)} = \frac{n! \Pochhsymb{1-\beta}{k}\Pochhsymb{d/2}{k}}{\Pochhsymb{3/2-\beta}{k}\Pochhsymb{d}{2k} (n-k)!} \Hypergeom{3}{2}{k-n,k+1-\beta,k+d/2}{k+3/2-\beta,2k+d}{1}.
\end{equation*}

Let $k_0$ be the smallest integer with $k_0+1-\beta > 0$. Then we can use the integral representation of $a_{k,n}/Z(d,k)$ given in Proposition~\ref{prop:hypergeometric.polynomial.expansion.specialized} if $k \geq k_0$. Let $k \geq k_0$. Then
\begin{align*}
A_{k}^{(2)}(s,d)
&\DEF \frac{1}{2\beta-1} \sum_{n=k}^\infty \frac{\Pochhsymb{1/2 - \beta}{n} \Pochhsymb{1-\beta}{n}}{\Pochhsymb{(d+1)/2}{n} \, n!} \frac{a_{k,n}}{Z(d,k)} \\
&= \frac{1}{2\beta-1} \frac{\gammafcn(3/2 - \beta) \gammafcn(d)}{\gammafcn(1 - \beta) \gammafcn(1/2) \gammafcn(d/2) \gammafcn(d/2)} \sum_{n=k}^\infty \frac{\Pochhsymb{1/2 - \beta}{n} \Pochhsymb{1-\beta}{n}}{\Pochhsymb{(d+1)/2}{n} \, n!} \frac{n!}{\Pochhsymb{d/2}{k} (n-k)!} \\
&\phantom{=\pm}\times \int_0^1 \int_0^1 t^{k + 1 - \beta-1} x^{k + d/2-1} \left( 1 - t \right)^{1/2-1} \left( 1 - x \right)^{k + d/2-1} \left( 1 - x t \right)^{n-k} \dd x \dd t \\
&= \frac{1}{2\beta-1} \frac{\gammafcn(3/2 - \beta) \gammafcn(d)}{\gammafcn(1 - \beta) \gammafcn(1/2) \gammafcn(d/2) \gammafcn(d/2)} \frac{1}{\Pochhsymb{d/2}{k}} \sum_{n=k}^\infty \frac{\Pochhsymb{1/2 - \beta}{n} \Pochhsymb{1-\beta}{n}}{\Pochhsymb{(d+1)/2}{n} \, (n-k)!} \\
&\phantom{=\pm}\times \int_0^1 \int_0^1 t^{k + 1 - \beta-1} x^{k + d/2-1} \left( 1 - t \right)^{1/2-1} \left( 1 - x \right)^{k + d/2-1} \left( 1 - x t \right)^{n-k} \dd x \dd t.
\end{align*}

We compute
\begin{align*}
&\sum_{n=k}^\infty \frac{\Pochhsymb{1/2 - \beta}{n} \Pochhsymb{1-\beta}{n}}{\Pochhsymb{(d+1)/2}{n} \, (n-k)!} \left( 1 - x t \right)^{n-k} = \sum_{n=0}^\infty \frac{\Pochhsymb{1/2 - \beta}{n+k} \Pochhsymb{1-\beta}{n+k}}{\Pochhsymb{(d+1)/2}{n+k} \, n!} \left( 1 - x t \right)^{n} \\
&\phantom{equals}= \frac{\Pochhsymb{1/2 - \beta}{k} \Pochhsymb{1-\beta}{k}}{\Pochhsymb{(d+1)/2}{k}} \sum_{n=0}^\infty \frac{\Pochhsymb{k + 1/2 - \beta}{n} \Pochhsymb{k + 1-\beta}{n}}{\Pochhsymb{k + (d+1)/2}{n} \, n!} \left( 1 - x t \right)^{n} \\
&\phantom{equals}= \frac{\Pochhsymb{1/2 - \beta}{k} \Pochhsymb{1-\beta}{k}}{\Pochhsymb{(d+1)/2}{k}} \Hypergeom{2}{1}{k+1/2-\beta, k+1-\beta}{k+(d+1)/2}{1 - x t} \\
&\phantom{equals}= \frac{\Pochhsymb{1/2 - \beta}{k} \Pochhsymb{1-\beta}{k}}{\Pochhsymb{(d+1)/2}{k}} \left( x t \right)^{2\beta-1-k+d/2} \Hypergeom{2}{1}{\beta+d/2, \beta+(d-1)/2}{k+(d+1)/2}{1 - x t}.
\end{align*}
Hence, for $k + 1 - d/2 - 2 \beta > 0$ (and $k > \beta - 1$), we have
\begin{equation}
\begin{split} \label{eq:A.k.beta.first}
A_{k}^{(2)}(s,d)
&= - \frac{1}{2} \frac{\gammafcn(d/2-s) \gammafcn(d)}{\gammafcn((d+1)/2-s) \gammafcn(1/2) \gammafcn(d/2) \gammafcn(d/2)} \frac{\Pochhsymb{d/2 - s}{k} \Pochhsymb{(d+1)/2-s}{k}}{\Pochhsymb{(d+1)/2}{k}\Pochhsymb{d/2}{k}} \\
&\times \Bigg\{ \sum_{n=0}^\infty \frac{\Pochhsymb{s+1/2}{n} \Pochhsymb{s}{n}}{\Pochhsymb{k+(d+1)/2}{n} \, n!} \int_0^1 \int_0^1 \frac{t^{s+1/2-1} x^{2s-1} \left( 1 - t \right)^{1/2-1} \left( 1 - x \right)^{k + d/2-1}}{\left( 1 - x t \right)^{-n}} \dd x \dd t \Bigg\}. 
\end{split}
\end{equation}
Invoking integral representations for the hypergeometric $_2\HyperF_1$ and $_3\HyperF_2$ functions and a transformation for $_3\HyperF_2$ functions of unity argument (cf. \cite[Eq.~7.4.4.83]{PrBrMa1990III}), we arrive at
\begin{align*}
&\int_0^1 \int_0^1 \frac{t^{s+1/2-1} x^{2s-1} \left( 1 - t \right)^{1/2-1} \left( 1 - x \right)^{k + d/2-1}}{\left( 1 - x t \right)^{-n}} \dd x \dd t \\
&\phantom{equals}= \frac{\gammafcn(2s) \gammafcn(k+d/2)}{\gammafcn(k+d/2+2s)} \int_0^1 t^{s+1/2-1} \left( 1 - t \right)^{1/2-1} \Hypergeom{2}{1}{-n,2s}{k+d/2+2s}{t} \dd t \\
&\phantom{equals}= \frac{\gammafcn(s+1/2) \gammafcn(1/2)}{\gammafcn(s+1)} \frac{\gammafcn(2s) \gammafcn(k+d/2)}{\gammafcn(k+d/2+2s)} \Hypergeom{3}{2}{-n,s+1/2,2s}{s+1,k+d/2+2s}{1} \\
&\phantom{equals}= \frac{\gammafcn(s+1/2) \gammafcn(1/2)}{\gammafcn(s+1)} \frac{\gammafcn(2s) \gammafcn(k+d/2)}{\gammafcn(k+d/2+2s)} \frac{\Pochhsymb{k+(d+1)}{n}}{\Pochhsymb{k+d/2+2s}{n}} \Hypergeom{3}{2}{-n,1-s,1/2}{1+s,k+(d+1)/2}{1}.
\end{align*}

Substitution into \eqref{eq:A.k.beta.first} yields
\begin{equation}
\begin{split} \label{eq:A.k.beta.second}
A_{k}^{(2)}(s,d)
&= - \frac{1}{2} \frac{\gammafcn(d) \gammafcn(d/2-s) \gammafcn(s+1/2) \gammafcn(2s)}{\gammafcn(d/2) \gammafcn((d+1)/2-s) \gammafcn(s+1) \gammafcn(d/2+2s)} \frac{\Pochhsymb{d/2 - s}{k} \Pochhsymb{(d+1)/2-s}{k}}{\Pochhsymb{(d+1)/2}{k} \Pochhsymb{d/2+2s}{k}} \\
&\times \Bigg\{ \sum_{n=0}^\infty \frac{\Pochhsymb{s+1/2}{n} \Pochhsymb{s}{n}}{\Pochhsymb{k+d/2+2s}{n} \, n!} \Hypergeom{3}{2}{-n,1-s,1/2}{1+s,k+(d+1)/2}{1} \Bigg\}. 
\end{split}
\end{equation}

For asymptotic analysis of $A_{k}^{(2)}(s,d)$ for large $k$ we reverse to \eqref{eq:A.k.beta.first}. The infinite series is already in a hybrid-asymptotic form (as $k$ becomes large). For the curly braced expression we can, thus, write
\begin{align*}
\Bigg\{ \cdots \Bigg\} 
&= \int_0^1 \int_0^1 t^{s+1/2-1} x^{2s-1} \left( 1 - t \right)^{1/2-1} \left( 1 - x \right)^{k + d/2-1} \dd x \dd t + \frac{\left( s + 1/2 \right) s}{k + (d+1)/2} \Omega(k) \\
&= \frac{\gammafcn(s+1/2) \gammafcn(1/2)}{\gammafcn(s+1)} \frac{\gammafcn(2s) \gammafcn(k+d/2)}{\gammafcn(k+d/2+2s)} + \frac{\left( s + 1/2 \right) s}{k + (d+1)/2} \Omega(k),
\end{align*}
where
\begin{equation*}
\Omega(k) \DEF \sum_{n=0}^\infty \frac{\Pochhsymb{1+s+1/2}{n} \Pochhsymb{1+s}{n}}{\Pochhsymb{1+k+(d+1)/2}{n} \, (n+1)!} \int_0^1 \int_0^1 \frac{t^{s+1/2-1} x^{2s-1} \left( 1 - t \right)^{1/2-1} \left( 1 - x \right)^{k + d/2-1}}{\left( 1 - x t \right)^{-n-1}} \dd x \dd t.
\end{equation*}
Since $(1 - x t)^{n+1} \leq 1$ and $(n+1)! > n!$, we have the estimate
\begin{equation*}
0 < \Omega(k) \leq \frac{\gammafcn(s+1/2) \gammafcn(1/2)}{\gammafcn(s+1)} \frac{\gammafcn(2s) \gammafcn(k+d/2)}{\gammafcn(k+d/2+2s)} \Hypergeom{2}{1}{1+s+1/2,1+s}{1+k+(d+1)/2}{1}.
\end{equation*}
Observe that this hypergeometric function at unity argument decreases as $k$ gets larger. For $k > 2s + 1 - d/2$ we can apply \cite[Eq.~15.4.20]{DLMF2011.08.29} to evaluate the hypergeometric function at argument unity:
\begin{equation*}
0 < \Omega(k) \leq \frac{\gammafcn(s+1/2) \gammafcn(1/2)}{\gammafcn(s+1)} \frac{\gammafcn(2s) \gammafcn(k+d/2)}{\gammafcn(k+d/2+2s)} \frac{\gammafcn(1+k+(d+1)/2) \gammafcn(k+d/2-1-2s)}{\gammafcn(k+d/2-s)\gammafcn(k+(d+1)/2-s)}.
\end{equation*}
We conclude that
\begin{equation*}
\Bigg\{ \cdots \Bigg\} = \frac{\gammafcn(s+1/2) \gammafcn(1/2)}{\gammafcn(s+1)} \frac{\gammafcn(2s) \gammafcn(k+d/2)}{\gammafcn(k+d/2+2s)} \left( 1 + \mathcal{O}(k^{-1}) \right) \qquad \text{as $k \to \infty$.}
\end{equation*}
Hence, 
\begin{align*}
\frac{\Pochhsymb{d/2 - s}{k} \Pochhsymb{(d+1)/2-s}{k}}{\Pochhsymb{(d+1)/2}{k}\Pochhsymb{d/2}{k}} \Bigg\{ \cdots \Bigg\} 
&= \frac{\gammafcn( (d+1)/2 ) \gammafcn( d/2 )\gammafcn(s+1/2) \gammafcn(1/2) \gammafcn(2s)}{\gammafcn( d/s - s) \gammafcn( (d + 1)/2 - s) \gammafcn(s+1)} \\
&\phantom{=\pm}\times \frac{\gammafcn( k + d/s - s) \gammafcn( k + (d + 1)/2 - s)}{\gammafcn( k + (d+1)/2 ) \gammafcn(k+d/2+2s)} \left( 1 + \mathcal{O}(k^{-1}) \right).
\end{align*}
Putting everything together, we obtain
\begin{equation*}
\begin{split}
A_{k}^{(2)}(s,d)
&= - \frac{1}{2} \frac{\gammafcn( (d+1)/2 ) \gammafcn( d/2 ) \gammafcn(d) \gammafcn(s+1/2) \gammafcn(2s) \gammafcn(s+1/2) \gammafcn(1/2) \gammafcn(2s)}{\gammafcn(d/2) \gammafcn((d+1)/2-s) \gammafcn(s+1) \gammafcn(d/2+2s) \gammafcn( (d + 1)/2 - s) \gammafcn(s+1)} \\
&\phantom{=\pm}\times \frac{\gammafcn( k + d/s - s) \gammafcn( k + (d + 1)/2 - s)}{\gammafcn( k + (d+1)/2 ) \gammafcn(k+d/2+2s)} \left( 1 + \mathcal{O}(k^{-1}) \right), \qquad \text{as $k \to \infty$.}
\end{split}
\end{equation*}
Application of \cite[Eq.~5.11.12]{DLMF2011.08.29} yields that as $k \to \infty$
\begin{equation*}
A_{k}^{(2)}(s,d) \sim - \frac{1}{2} \frac{\gammafcn(1/2) \gammafcn( (d+1)/2 ) \gammafcn(d) \gammafcn(s+1/2) \gammafcn(s+1/2) \gammafcn(2s) \gammafcn(2s)}{\gammafcn((d+1)/2-s) \gammafcn( (d + 1)/2 - s) \gammafcn(d/2+2s) \gammafcn(s+1) \gammafcn(s+1)} \, k^{-4s}.
\end{equation*}
The duplication formula for the gamma function gives
\begin{equation} \label{eq:A.k.s.d}
A_{k}^{(2)}(s,d) \sim - 2^{d-2} \left[ \frac{\gammafcn( (d+1)/2 ) \gammafcn(s+1/2) \gammafcn(2s)}{\gammafcn( (d + 1)/2 - s) \gammafcn(s+1)} \right]^2 \frac{\gammafcn(d/2)}{\gammafcn(d/2+2s)} \, k^{-4s}.
\end{equation}
This completes the proof of Case {\bf (a)} of the Theorem~\ref{thm:integral.general.case}. 

{\bf Case (b)}. This is the content of Theorem~\ref{thm:integral.exceptional.cases}.

{\bf Case (c)}. Inspection of Case {\bf (a)} of this proof shows that it is save to take the limit as $\beta$ goes to a positive integer $M$ to obtain (cf., in particular, \eqref{eq:cal.K.beta.1} and \eqref{eq:tilde.c.d})
\begin{equation*}
\mathcal{K}_{M,1}( \PT{x}, \PT{y} ) = \lim_{\beta\to M} \mathcal{K}_{\beta,1}( \PT{x}, \PT{y} ) = \tilde{c}_d(M) \left\| \PT{x} - \PT{y} \right\|^{2M-1}, 
\end{equation*}
where (note that $L = M - 1$)
\begin{align*}
\tilde{c}_d(M) 
&= \frac{(-1)^{M}}{\sin( \pi (M-1/2-L) )} \frac{2^{1-2M}}{2M-1} \frac{\gammafcn((d+1)/2) \gammafcn(M)\gammafcn(M)}{\gammafcn(M+d/2) \gammafcn(M-1/2)} \\
&= (-1)^{M} \frac{\gammafcn((d+1)/2) \gammafcn(M)\gammafcn(M)}{\gammafcn(M+d/2) \gammafcn(M-1/2)}.
\end{align*}

Similarly (cf. \eqref{eq:K.beta.2.symmetric.form}, \eqref{eq:K.beta.2.first} and \eqref{eq:K.beta.2.second})
\begin{align}
\mathcal{K}_{M,2}( \PT{x}, \PT{y} )
&= \lim_{\beta\to M} \mathcal{K}_{\beta,2}( \PT{x}, \PT{y} ) \notag \\
&= \frac{2^{1-2M}}{2M-1} \int_{\mathbb{S}^d} \frac{\Hypergeom{2}{1}{1/2-M,1-M}{3/2-M}{\left( \frac{\PT{x} \cdot \PT{z}-\PT{y} \cdot \PT{z}}{2 - \PT{x} \cdot \PT{z} - \PT{y} \cdot \PT{z}} \right)^2}}{\left( 2 - \PT{x} \cdot \PT{z} - \PT{y} \cdot \PT{z} \right)^{1-2M}} \dd \sigma_{d}(\PT{z}) \notag \\
&= \frac{1}{2M-1} \sum_{m=0}^\infty \sum_{n=0}^\infty \frac{\Pochhsymb{1/2 - M}{m+n} \Pochhsymb{1-M}{m+n} \Pochhsymb{1/2}{n}}{\Pochhsymb{(d+1)/2}{m+n} \Pochhsymb{3/2-M}{n} \, m! n!} \left( 1 - v^2 \right)^m v^{2n} \notag \\
&= \frac{1}{2M-1} \sum_{n=0}^{M-1} \frac{\Pochhsymb{1/2 - M}{n} \Pochhsymb{1-M}{n}}{\Pochhsymb{(d+1)/2}{n} \, n!} \Hypergeom{2}{1}{-n,1-M}{3/2-M}{\frac{1 - \PT{x} \cdot \PT{y}}{2}}. \label{eq:K.beta.2..beta.EQ.integer}
\end{align}
(The last series terminates, since $\Pochhsymb{1-M}{n} = 0$ for $n \geq M$.)

Since $\mathcal{K}_{M,2}( \PT{x}, \PT{y} )$ is a polynomial of degree $M-1$ and Proposition~\ref{prop:hypergeometric.polynomial.expansion.specialized} applies if $\beta$ is a positive integer, we can simply write (cf. Proof of Theorem~\ref{thm:integral.general.case})
\begin{equation*}
\mathcal{K}_{M,2}( \PT{x}, \PT{y} ) = \sum_{k=0}^{M-1} A_{k}^{(2)}(s,d) \, Z(d,k) P_k^{(d)}(\PT{x} \cdot \PT{y}),
\end{equation*}
where (using $\beta$ where it is convenient)
\begin{equation} \label{eq:A.k.s.d.beta.EQ.M}
A_{k}^{(2)}(s,d) \DEF \frac{1}{2M-1} \sum_{n=k}^{M-1} \frac{\Pochhsymb{1/2 - M}{n} \Pochhsymb{1-M}{n}}{\Pochhsymb{(d+1)/2}{n} \, n!} \frac{a_{k,n}}{Z(d,k)}
\end{equation}
and 
\begin{equation*}
\frac{a_{k,n}}{Z(d,k)} = \frac{n! \Pochhsymb{1-M}{k}\Pochhsymb{d/2}{k}}{\Pochhsymb{3/2-M}{k}\Pochhsymb{d}{2k} (n-k)!} \Hypergeom{3}{1}{k-n,k+1-M,k+d/2}{k+3/2-M,2k+d}{1}.
\end{equation*}
(Note that the $_3\HyperF_2$ hypergeometric function above reduces to a polynomial of degree $M-1-k$.) 

Furthermore, the relations for $\mathcal{K}_\beta( \PT{x}, \PT{x} )$ and $\mathcal{K}_\beta( \PT{x}, -\PT{x} )$ also extend to the case when $\beta$ is a positive integer $M$. This completes the proof.
\end{proof}

For the proof of Theorem~\ref{thm:integral.exceptional.cases} we need the following auxiliary result.

\begin{lem} \label{lem:aux.res.4}
Let $d \geq 2$. Let $L$ be a positive integer, $n$ an integer $\geq0$ and $\eps$ a real number with $| \eps | < 1$ and $\eps \neq0$. Then for $\PT{x}, \PT{y} \in \mathbb{S}^d$ with $\PT{y} \neq \PT{x}$ and $\PT{y} \neq - \PT{x}$ there holds
\begin{equation*}
\begin{split}
&\int_{\mathbb{S}^d} \frac{\left( \frac{\PT{x} \cdot \PT{z}-\PT{y} \cdot \PT{z}}{2 - \PT{x} \cdot \PT{z} - \PT{y} \cdot \PT{z}} \right)^{2n}}{\left( 1 - \frac{1}{2} \PT{x} \cdot \PT{z} - \frac{1}{2} \PT{y} \cdot \PT{z} \right)^{-2L-2\eps}} \dd \sigma_{d}(\PT{z}) \\
&\phantom{equals}= \frac{\Pochhsymb{1/2}{n}}{\Pochhsymb{(d+1)/2}{n}} \, v^{2n} \Hypergeom{2}{1}{n-L-\eps,n-L+1/2-\eps}{n+(d+1)/2}{1 - v^2}.
\end{split}
\end{equation*}
\end{lem}

\begin{proof}
We use the same notation as in the Proof of Theorem~\ref{thm:integral.general.case}. 

Let $g_n( \PT{x}, \PT{y} )$ denote the integral we want to compute. Then applying the Funk-Hecke formula twice, we obtain
\begin{align*}
g_n( \PT{x}, \PT{y} ) 
&= \int_{\mathbb{S}^d} \frac{\left( \frac{1}{2} \PT{x} \cdot \PT{z}- \frac{1}{2} \PT{y} \cdot \PT{z} \right)^{2n}}{\left( 1 - \frac{1}{2} \PT{x} \cdot \PT{z} - \frac{1}{2} \PT{y} \cdot \PT{z} \right)^{2n-2L-2\eps}} \dd \sigma_{d}(\PT{z}) \\
&= \frac{\omega_{d-1}}{\omega_d} \int_{-1}^1 \int_{\mathbb{S}^{d-1}} \frac{\left( v u \right)^{2n} \left( 1 - u^2 \right)^{d/2-1}}{\left( 1 - \sqrt{1-v^2} \sqrt{1-u^2} \, \PT{x}^* \cdot \PT{z}^* \right)^{2n-2L-2\eps}} \dd \sigma_{d-1}(\PT{z}^*) \dd u \\
&= \frac{\omega_{d-1}}{\omega_d} \frac{\omega_{d-2}}{\omega_{d-1}} \int_{-1}^1  \int_{-1}^1 \frac{\left( v u \right)^{2n} \left( 1 - u^2 \right)^{d/2-1} \left( 1 - \tau^2 \right)^{(d-1)/2-1}}{\left( 1 - \sqrt{1-v^2} \sqrt{1-u^2} \, \tau \right)^{2n-2L-2\eps}} \dd \tau \dd u.
\end{align*}
The standard substitution $2x = 1 + \tau$ gives
\begin{align*}
g_n( \PT{x}, \PT{y} ) 
&= \frac{\omega_{d-1}}{\omega_d} 2^{d-2} \frac{\omega_{d-2}}{\omega_{d-1}} v^{2n} \int_{-1}^1 \frac{u^{2n} \left( 1 - u^2 \right)^{d/2-1}}{\left( 1 + \sqrt{1-v^2} \sqrt{1-u^2} \right)^{2n-2L-2\eps}} \\
&\phantom{=\pm}\times \int_0^1 \frac{x^{(d-1)/2-1} \left( 1 - x \right)^{(d-1)/2-1}}{\left( 1 - \frac{2\sqrt{1-v^2} \sqrt{1-u^2}}{1 + \sqrt{1-v^2} \sqrt{1-u^2}} \, x \right)^{2n-2L-2\eps}} \dd x \, \dd u.
\end{align*}
The inner integral represents (cf. \cite[Eq.~15.6.1]{DLMF2011.08.29})
\begin{equation*}
\frac{\gammafcn((d-1)/2) \gammafcn((d-1)/2)}{\gammafcn(d-1)} \Hypergeom{2}{1}{2n-2L-2\eps,(d-1)/2}{d-1}{\frac{2\sqrt{1-v^2} \sqrt{1-u^2}}{1 + \sqrt{1-v^2} \sqrt{1-u^2}}}
\end{equation*}
which can be written as follows when using the quadratic transformation \cite[Eq.~15.8.13]{DLMF2011.08.29}
\begin{equation*}
\begin{split}
&2^{2-d} \frac{\sqrt{\pi} \gammafcn((d-1)/2)}{\gammafcn(d/2)} \left( 1 + \sqrt{1-v^2} \sqrt{1-u^2} \right)^{2n-2L-2\eps} \\
&\phantom{equals\pm}\times \Hypergeom{2}{1}{n-L-\eps,n-L+1/2-\eps}{d/2}{\left( 1 - v^2 \right) \left( 1 - u^2 \right)}.
\end{split}
\end{equation*}
Taking into account \eqref{eq:C.d}, we arrive at
\begin{equation*}
\begin{split}
&g_n( \PT{x}, \PT{y} ) = 2 \frac{\gammafcn((d+1)/2)}{\sqrt{\pi} \, \gammafcn(d/2)} v^{2n} \int_{0}^1 u^{2n} \left( 1 - u^2 \right)^{d/2-1} \\
&\phantom{equals\pm}\times \Hypergeom{2}{1}{n-L-\eps,n-L+1/2-\eps}{d/2}{\left( 1 - v^2 \right) \left( 1 - u^2 \right)} \dd u.
\end{split}
\end{equation*}
The change of variable $x = 1 - u^2$ yields
\begin{equation*}
\begin{split}
&g_n( \PT{x}, \PT{y} ) = \frac{\gammafcn((d+1)/2)}{\sqrt{\pi} \, \gammafcn(d/2)} v^{2n} \int_{0}^1 x^{d/2-1} \left( 1 - x \right)^{n-1/2} \\
&\phantom{equals\pm}\times \Hypergeom{2}{1}{n-L-\eps,n-L+1/2-\eps}{d/2}{\left( 1 - v^2 \right) x} \dd u.
\end{split}
\end{equation*}
This is the integral representation of a $_3\HyperF_2$-hypergeometric function (cf. \cite[Eq.~16.5.2]{DLMF2011.08.29}):
\begin{equation*}
g_n( \PT{x}, \PT{y} ) = \frac{\gammafcn((d+1)/2)}{\sqrt{\pi} \, \gammafcn(d/2)} v^{2n} \frac{\gammafcn(d/2) \gammafcn(n+1/2)}{\gammafcn(n+(d+1)/2)} \Hypergeom{3}{2}{n-L-\eps,n-L+1/2-\eps,d/2}{d/2, n+(d+1)/2}{1 - v^2}
\end{equation*}
Because of an upper parameter coinciding with a lower parameter, the $_3\HyperF_2$-function reduces to a $_2\HyperF_1$-function
\begin{equation*}
g_n( \PT{x}, \PT{y} ) = \frac{\Pochhsymb{1/2}{n}}{\Pochhsymb{(d+1)/2}{n}} \, v^{2n} \Hypergeom{2}{1}{n-L-\eps,n-L+1/2-\eps}{n+(d+1)/2}{1 - v^2}.
\end{equation*}
This completes the proof.
\end{proof}

\begin{proof}[Proof of Theorem~\ref{thm:integral.exceptional.cases}] 
We use the same notation as in Theorem~\ref{thm:integral.general.case} and its proof. The result is obtained by a limit process as $2\beta-1$ tends to an even integer $2L$. That is, let $\beta - 1/2 = L + \eps$ with $\eps$ sufficiently small and $\eps \neq 0$. We consider limits as $\eps \to 0$.

Let $\PT{y} = \PT{x}$. Then \eqref{eq:K.beta.x.x} implies that
\begin{equation*}
\mathcal{K}_{L+1/2}( \PT{x}, \PT{x} ) = \lim_{\eps \to 0} \mathcal{K}_{L+1/2+\eps}( \PT{x}, \PT{x} ) = \frac{1}{2L} \frac{\gammafcn((d+1)/2) \gammafcn(2L + d/2)}{\gammafcn(L + d/2) \gammafcn(L + (d+1)/2)}.
\end{equation*}

Let $\PT{y} = - \PT{x}$, $\PT{x}$. Then \eqref{eq:K.beta.x.neg.x} implies that 
\begin{align*}
\mathcal{K}_\beta( \PT{x}, -\PT{x} ) 
&= (-1)^{L+1} c_d(\beta) 2^{2\beta-1} + \frac{1}{2\beta-1} \Hypergeom{3}{2}{1/2-\beta, 1 - \beta, 1/2}{3/2 - \beta, (d+1)/2}{1} \\
&= (-1)^{L+1} c_d(\beta) 2^{2\beta-1} - \frac{1}{2} \sum_{n=0}^\infty \frac{\Pochhsymb{1/2-\beta}{n} \Pochhsymb{1-\beta}{n} \Pochhsymb{1/2}{n}}{\Pochhsymb{1/2-\beta}{n+1} \Pochhsymb{(d+1)/2}{n} n!} \\
&= (-1)^{L+1} c_d(L+1/2+\eps) 2^{2L+2\eps} - \frac{1}{2} \sum_{n=0}^\infty \frac{1}{n-L-\eps} \frac{\Pochhsymb{1/2-L-\eps}{n} \Pochhsymb{1/2}{n}}{\Pochhsymb{(d+1)/2}{n} n!}.
\end{align*}
Hence
\begin{align*}
\mathcal{K}_{L+1/2}( \PT{x}, -\PT{x} )  
&= \lim_{\eps\to0} \mathcal{K}_\beta( \PT{x}, -\PT{x} ) \\
&= \lim_{\eps\to0} \left\{ \frac{1}{2} \frac{1}{\eps} \frac{\Pochhsymb{1/2-L-\eps}{L} \Pochhsymb{1/2}{L}}{\Pochhsymb{(d+1)/2}{L} L!} + (-1)^{L+1} c_d(L+1/2+\eps) 2^{2L+2\eps} \right\} \\
&\phantom{=}- \frac{1}{2} \sum_{\substack{n=0 \\ n\neq L}}^\infty \frac{1}{n-L} \frac{\Pochhsymb{1/2-L}{n} \Pochhsymb{1/2}{n}}{\Pochhsymb{(d+1)/2}{n} n!}
\end{align*}
The limit above is the same as in \eqref{eq:K.L.b} below but for $v=1$ (that is, $\| \PT{x} - \PT{y} \| = 2$). We can use \eqref{eq:curly.braces.limit} below with $v=1$.
\begin{equation*}
\begin{split}
\mathcal{K}_{L+1/2}( \PT{x}, -\PT{x} )  
&= - \frac{1}{2} \sum_{\substack{n=0 \\ n\neq L}}^\infty \frac{1}{n-L} \frac{\Pochhsymb{1/2-L}{n} \Pochhsymb{1/2}{n}}{\Pochhsymb{(d+1)/2}{n} n!} + \frac{(-1)^{L+1}}{2} \frac{\Pochhsymb{1/2}{L} \Pochhsymb{1/2}{L}}{\Pochhsymb{(d+1)/2}{L} L!} \\
&\phantom{equals=\pm}\times \Big( \digammafcn( 1 / 2 ) + \digammafcn( L + 1 / 2 ) - \digammafcn( L + 1 ) - \digammafcn( L + ( d + 1 ) / 2 ) \Big).
\end{split}
\end{equation*}
This shows \eqref{eq:K.L.plus.half.at.antipodal}.

For the remaining proof we assume the following: let $\PT{y} \neq - \PT{x}$ and $\PT{y} \neq \PT{x}$; that is $0 < v < 1$. 

The integrand of the integral representation of $\mathcal{Q}_{\beta-1}$ in Theorem~\ref{thm:integral.general.case} splits into two non-critical and one critical part as follows.
\begin{align}
&\frac{1}{2\beta-1} \Hypergeom{2}{1}{1/2-\beta,1-\beta}{3/2-\beta}{x^2} = \frac{1}{2\beta-1} \sum_{n=0}^\infty \frac{\Pochhsymb{1/2-\beta}{n} \Pochhsymb{1-\beta}{n}}{\Pochhsymb{3/2-\beta}{n} n!} x^{2n} \notag \\
&\phantom{equals}= - \frac{1}{2} \sum_{n=0}^\infty \frac{\Pochhsymb{1-\beta}{n}}{\left( n + 1/2 - \beta \right) n!} x^{2n} = - \frac{1}{2} \sum_{n=0}^\infty \frac{\Pochhsymb{1/2-\eps-L}{n}}{\left( n - L - \eps \right) n!} x^{2n} \notag \\
&\phantom{equals}= \frac{1}{2} \sum_{n=0}^{L-1} \frac{\Pochhsymb{1/2-\eps-L}{n}}{\left( L - n + \eps \right) n!} x^{2n} - \frac{1}{2} \frac{\Pochhsymb{1/2-\eps-L}{L}}{\left( - \eps \right) L!} x^{2L} - \frac{1}{2} \sum_{n=L+1}^\infty \frac{\Pochhsymb{1/2-\eps-L}{n}}{\left( n - L - \eps \right) n!} x^{2n}. \label{eq:repr.hypergeometric.1}
\end{align}
Consequently,
\begin{align*}
&\lim_{\eps \to 0} \left\{ \frac{1}{2\beta-1} \Hypergeom{2}{1}{1/2-\beta,1-\beta}{3/2-\beta}{x^2} + - \frac{1}{2} \frac{\Pochhsymb{1/2-\eps-L}{L}}{\left( - \eps \right) L!} x^{2L} \right\} \\
&\phantom{equals}= \frac{1}{2} \sum_{n=0}^{L-1} \frac{\Pochhsymb{1/2-L}{n}}{\left( L - n \right) n!} x^{2n} - \frac{1}{2} \sum_{n=L+1}^\infty \frac{\Pochhsymb{1/2-L}{n}}{\left( n - L \right) n!} x^{2n} \\
&\phantom{equals}= \frac{1}{2} \sum_{n=1}^{L} \frac{1}{n} \frac{\Pochhsymb{1/2-L}{L-n}}{(L-n)!} \left( x^{2} \right)^{L-n} - \frac{1}{2} x^{2L} \sum_{n=1}^\infty \frac{1}{n} \frac{\Pochhsymb{1/2-L}{n+L}}{ (n+L)!} x^{2n} \\
&\phantom{equals}= \frac{1}{2} \sum_{n=1}^{L} \frac{1}{n} \frac{\Pochhsymb{1/2-L}{L-n}}{(L-n)!} \left( x^{2} \right)^{L-n} - \frac{1}{2} \frac{\Pochhsymb{1/2-L}{L}}{L!}  x^{2L} \sum_{n=1}^\infty \frac{1}{n} \frac{\Pochhsymb{1/2}{n}}{\Pochhsymb{L+1}{n}} x^{2n}.
\end{align*}

The critical contribution to the desired limit (as $\eps \to 0$) for $\mathcal{K}_\beta$ are due to the singled out term above and the second part in \eqref{eq:cal.K.beta.form}. 

{\bf Integral representation.} Hence, by Theorem~\ref{thm:integral.general.case} and \eqref{eq:cal.K.beta.1} %(recall, $\beta = L + 1/2 + \eps$, $\eps \neq 0$ small)
\begin{equation}
\begin{split} \label{eq:int.representation}
\mathcal{K}_\beta( \PT{x}, \PT{y} ) 
&= \int_{\mathbb{S}^d} \frac{\frac{1}{2\beta-1} \Hypergeom{2}{1}{1/2-\beta,1-\beta}{3/2-\beta}{\left( \frac{\PT{x} \cdot \PT{z}-\PT{y} \cdot \PT{z}}{2 - \PT{x} \cdot \PT{z} - \PT{y} \cdot \PT{z}} \right)^2} + \frac{1}{2} \frac{\Pochhsymb{1-\beta}{L}}{\left( - \eps \right) \gammafcn(L+1)} \left( \frac{\PT{x} \cdot \PT{z}-\PT{y} \cdot \PT{z}}{2 - \PT{x} \cdot \PT{z} - \PT{y} \cdot \PT{z}} \right)^{2L} }{\left( 1 - \frac{1}{2} \PT{x} \cdot \PT{z} - \frac{1}{2} \PT{y} \cdot \PT{z} \right)^{1-2\beta}} \dd \sigma_{d}(\PT{z}) \\
&\phantom{=}+ \int_{\mathbb{S}^d} \left\{ \frac{\gammafcn(1-2\beta) \gammafcn( \beta )}{\gammafcn(1-\beta)} \left[ \left( \PT{x} \cdot \PT{z} - \PT{y} \cdot \PT{z} \right)^2 \right]^{\beta-1/2} - \frac{\frac{1}{2} \frac{\Pochhsymb{1-\beta}{L}}{\left( - \eps \right) \gammafcn(L+1)} \left( \frac{\PT{x} \cdot \PT{z}-\PT{y} \cdot \PT{z}}{2 - \PT{x} \cdot \PT{z} - \PT{y} \cdot \PT{z}} \right)^{2L} }{\left( 1 - \frac{1}{2} \PT{x} \cdot \PT{z} - \frac{1}{2} \PT{y} \cdot \PT{z} \right)^{1-2\beta}} \right\} \dd \sigma_d( \PT{z} ).
\end{split}
\end{equation}

The integrand in curly braces above can be rewritten as follows
\begin{equation}
\begin{split} \label{eq:integrand.1}
\Bigg\{ \cdots \Bigg\} 
&= - \frac{2^{1-2\beta}}{2} \frac{\Pochhsymb{1-\beta}{L}}{L!} \frac{\left( \PT{x} \cdot \PT{z}-\PT{y} \cdot \PT{z} \right)^{2L}}{\left( 2 - \PT{x} \cdot \PT{z}-\PT{y} \cdot \PT{z} \right)^{2L + 1 - 2\beta}} \\
&\phantom{=\pm}\times \frac{1}{\eps} \Bigg[ 2^{2\beta} \frac{L! \left( - \eps \right) \gammafcn( 1 - 2 \beta ) \gammafcn( \beta )}{\gammafcn( 1 - \beta) \Pochhsymb{1-\beta}{L}} \left( \frac{\left( \PT{x} \cdot \PT{z}-\PT{y} \cdot \PT{z} \right)^2}{\left( 2 - \PT{x} \cdot \PT{z}-\PT{y} \cdot \PT{z} \right)^2} \right)^{\beta-1/2-L} - 1 \Bigg].
\end{split}
\end{equation}
Since 
\begin{align*}
\Pochhsymb{1-\beta}{L} &= \frac{(-1)^L}{\Pochhsymb{\beta}{-L}} = (-1)^L \frac{\gammafcn( \beta )}{\gammafcn( \beta - L )}, \\
\left( - \eps \right) \gammafcn( - L - \eps ) &= \left( - \eps \right) \gammafcn( - \eps ) \Pochhsymb{-\eps}{-L} = \gammafcn(1-\eps) \frac{(-1)^L}{\Pochhsymb{1+\eps}{L}} = (-1)^L \frac{\gammafcn(1-\eps) \gammafcn(1+\eps)}{\gammafcn( L + 1 + \eps )},
\end{align*}
one has
\begin{align*}
2^{2\beta} \frac{L! \left( - \eps \right) \gammafcn( 1 - 2 \beta ) \gammafcn( \beta )}{\gammafcn( 1 - \beta) \Pochhsymb{1-\beta}{L}} 
&= \frac{L! \left( - \eps \right) \gammafcn( 1/2 - \beta ) \gammafcn( 1 - \beta ) \gammafcn( \beta )}{\gammafcn(1/2) \gammafcn( 1 - \beta) \Pochhsymb{1-\beta}{L}} \\
&= (-1)^L \frac{L! \left( - \eps \right) \gammafcn( 1/2 - \beta ) \gammafcn( 1 - \beta ) \gammafcn( \beta ) \gammafcn( \beta - L )}{\gammafcn(1/2) \gammafcn( 1 - \beta) \gammafcn( \beta )} \\
&= (-1)^L \frac{\gammafcn( L + 1 ) \left( - \eps \right) \gammafcn( 1/2 - \beta ) \gammafcn( \beta - L )}{\gammafcn(1/2)} \\
&= \frac{\gammafcn( L + 1 ) \gammafcn(1-\eps) \gammafcn(1+\eps) \gammafcn( 1/2 + \eps )}{\gammafcn(1/2) \gammafcn( L + 1 + \eps )} \\
&= \frac{\gammafcn( L + 1 )}{\gammafcn( L + 1 + \eps )} \frac{\gammafcn( 1/2 + \eps )}{\gammafcn(1/2)} \frac{\gammafcn(1+\eps)}{\gammafcn(1)} \frac{\gammafcn(1-\eps)}{\gammafcn(1)}.
\end{align*}

Let 
\begin{equation*}
\widetilde{G}( x; \eps ) \DEF \frac{\gammafcn( L + 1 )}{\gammafcn( L + 1 + \eps )} \frac{\gammafcn( 1/2 + \eps )}{\gammafcn(1/2)} \frac{\gammafcn(1+\eps)}{\gammafcn(1)} \frac{\gammafcn(1-\eps)}{\gammafcn(1)} \, x^{\eps}, \qquad -1/2 < \eps < 1/2, \, x > 0.
\end{equation*}
For $x > 0$ this function assumes the value $1$ at $\eps = 0$ and is differentiable at $\eps = 0$ with
\begin{equation*}
\frac{\partial \widetilde{G}( x; \eps )}{\partial \eps} \Big|_{\eps=0} = - \digammafcn( L + 1 ) + \digammafcn( 1 / 2 ) + \digammafcn( 1 ) - \digammafcn( 1 ) + \ln x = \ln x + \digammafcn( 1 / 2 ) - \digammafcn( L + 1 ),
\end{equation*}
where $\digammafcn( z ) \DEF \gammafcn^\prime( z ) / \gammafcn( z )$ denotes the Digamma function.

Hence, taking the limit as $\eps \to 0$ in \eqref{eq:integrand.1}, we obtain
\begin{align*}
\lim_{\eps\to0} \Bigg\{ \cdots \Bigg\} 
&= - \frac{2^{-2L}}{2} \frac{\Pochhsymb{1/2-L}{L}}{L!} \left( \PT{x} \cdot \PT{z}-\PT{y} \cdot \PT{z} \right)^{2L} \lim_{\eps\to0} \frac{\widetilde{G}( \left( \frac{\PT{x} \cdot \PT{z}-\PT{y} \cdot \PT{z}}{2 - \PT{x} \cdot \PT{z}-\PT{y} \cdot \PT{z}} \right)^2 ; \eps) - 1}{\eps} \\
&= \frac{(-1)^{L+1}}{2} \frac{\Pochhsymb{1/2}{L}}{2^{2L} \, L!} \left( \PT{x} \cdot \PT{z}-\PT{y} \cdot \PT{z} \right)^{2L} \\
&\phantom{=\pm}\times \Bigg[ 2 \ln \| \PT{x} - \PT{y} \| + \ln \Big( \frac{\frac{\PT{x}-\PT{y}}{\| \PT{x} - \PT{y} \|} \cdot \PT{z}}{2 - \PT{x} \cdot \PT{z}-\PT{y} \cdot \PT{z}} \Big)^2 + \digammafcn( 1 / 2 ) - \digammafcn( L + 1 ) \Bigg],
\end{align*}
provided $\PT{x} \cdot \PT{z}-\PT{y} \cdot \PT{z} \neq 0$.

It follows that after interchanging limit and integration in \eqref{eq:int.representation} (cf. Appendix~\ref{appendix:limits}), we get
\begin{align*}
\mathcal{K}_{L+1/2}( \PT{x}, \PT{y} ) 
&= \lim_{\eps \to 0} \mathcal{K}_\beta( \PT{x}, \PT{y} ) \\
&= \int_{\mathbb{S}^d} \frac{\frac{1}{2} \sum_{n=1}^{L} \frac{1}{n} \frac{\Pochhsymb{1/2-L}{L-n}}{(L-n)!} \left( \frac{\PT{x} \cdot \PT{z}-\PT{y} \cdot \PT{z}}{2 - \PT{x} \cdot \PT{z} - \PT{y} \cdot \PT{z}} \right)^{2L-2n}}{\left( 1 - \frac{1}{2} \PT{x} \cdot \PT{z} - \frac{1}{2} \PT{y} \cdot \PT{z} \right)^{-2L}} \dd \sigma_{d}(\PT{z}) \\
&\phantom{=}- \int_{\mathbb{S}^d} \frac{\frac{1}{2} \frac{\Pochhsymb{1/2-L}{L}}{L!} \sum_{n=1}^\infty \frac{1}{n} \frac{\Pochhsymb{1/2}{n}}{\Pochhsymb{L+1}{n}} \left( \frac{\PT{x} \cdot \PT{z}-\PT{y} \cdot \PT{z}}{2 - \PT{x} \cdot \PT{z} - \PT{y} \cdot \PT{z}} \right)^{2n+2L}}{\left( 1 - \frac{1}{2} \PT{x} \cdot \PT{z} - \frac{1}{2} \PT{y} \cdot \PT{z} \right)^{-2L}} \dd \sigma_{d}(\PT{z}) \\
&\phantom{=}+ \frac{(-1)^{L+1}}{2} \frac{\Pochhsymb{1/2}{L}}{2^{2L} \, L!} \int_{\mathbb{S}^d} \left( \PT{x} \cdot \PT{z}-\PT{y} \cdot \PT{z} \right)^{2L} \\
&\phantom{=\pm}\times \Bigg[ \ln \Big( \frac{\frac{\PT{x}-\PT{y}}{\| \PT{x} - \PT{y} \|} \cdot \PT{z}}{2 - \PT{x} \cdot \PT{z}-\PT{y} \cdot \PT{z}} \Big)^2 + \digammafcn( 1 / 2 ) - \digammafcn( L + 1 ) \Bigg] \dd \sigma_d( \PT{z} ) \\
&\phantom{=}+ (-1)^{L+1} \frac{\Pochhsymb{1/2}{L}}{2^{2L} \, L!} \ln \| \PT{x} - \PT{y} \| \int_{\mathbb{S}^d} \left( \PT{x} \cdot \PT{z}-\PT{y} \cdot \PT{z} \right)^{2L} \dd \sigma_d( \PT{z} ).
\end{align*}

The right-most integral evaluates as (cf. \eqref{eq:cal.K.beta.1})
\begin{equation*}
\int_{\mathbb{S}^d} \left( \PT{x} \cdot \PT{z}-\PT{y} \cdot \PT{z} \right)^{2L} \dd \sigma_d( \PT{z} ) = \frac{\Pochhsymb{1/2}{L}}{\Pochhsymb{(d+1)/2}{L}} \left\| \PT{x} - \PT{y} \right\|^{2L}.
\end{equation*}

This completes the proof of the integral representation.

{\bf Series representation.} From \eqref{eq:K.beta.2.2ndform} and the following relations we obtain the following representation 
\begin{equation*}
\mathcal{Q}_{\beta-1}( \PT{x}, \PT{y} ) = \frac{1}{2\beta-1} \sum_{n=0}^\infty \frac{\Pochhsymb{1/2 - \beta}{n} \Pochhsymb{1-\beta}{n}}{\Pochhsymb{(d+1)/2}{n} \, n!} \frac{\Pochhsymb{1/2}{n}}{\Pochhsymb{3/2-\beta}{n}} \Hypergeom{2}{1}{-n, 1 - \beta}{1/2-n}{1-v^2}
\end{equation*}
which can be simplified to (using $\beta = L + 1 / 2 + \eps$)
\begin{equation*}
\mathcal{Q}_{\beta-1}( \PT{x}, \PT{y} ) = - \frac{1}{2} \sum_{n=0}^\infty \frac{1}{n-L-\eps} \frac{\Pochhsymb{1/2-L-\eps}{n}\Pochhsymb{1/2}{n}}{\Pochhsymb{(d+1)/2}{n} \, n!} \Hypergeom{2}{1}{-n, 1/2 - L - \eps}{1/2-n}{1-v^2}.
\end{equation*}
This sum can be also split into two non-critical and one critical part regarding a limit process $\beta \to L + 1/2$ (that is, $\eps \to 0$) as follows.
\begin{align*}
\mathcal{Q}_{\beta-1}( \PT{x}, \PT{y} ) 
&= \frac{1}{2} \sum_{n=0}^{L-1} \frac{1}{L+\eps-n} \frac{\Pochhsymb{1/2-L-\eps}{n}\Pochhsymb{1/2}{n}}{\Pochhsymb{(d+1)/2}{n} \, n!} \Hypergeom{2}{1}{-n, 1/2 - L - \eps}{1/2-n}{1-v^2} \\
&\phantom{=}+ \frac{1}{2} \frac{1}{\eps} \frac{\Pochhsymb{1/2-L-\eps}{L}\Pochhsymb{1/2}{L}}{\Pochhsymb{(d+1)/2}{L} \, L!} \Hypergeom{2}{1}{-L, 1/2 - L - \eps}{1/2-L}{1-v^2} \\
&\phantom{=}- \frac{1}{2} \sum_{n=L+1}^\infty \frac{1}{n-L-\eps} \frac{\Pochhsymb{1/2-L-\eps}{n}\Pochhsymb{1/2}{n}}{\Pochhsymb{(d+1)/2}{n} \, n!} \Hypergeom{2}{1}{-n, 1/2 - L - \eps}{1/2-n}{1-v^2}
\end{align*}
Clearly (after shifting the index in the infinite series above)
\begin{align*}
&\lim_{\eps\to0} \left\{ \mathcal{Q}_{\beta-1}( \PT{x}, \PT{y} ) - \frac{1}{2} \frac{1}{\eps} \frac{\Pochhsymb{1/2-L-\eps}{L}\Pochhsymb{1/2}{L}}{\Pochhsymb{(d+1)/2}{L} \, L!} \Hypergeom{2}{1}{-L, 1/2 - L - \eps}{1/2-L}{1-v^2} \right\} \\
&\phantom{equals}= \frac{1}{2} \sum_{n=0}^{L-1} \frac{1}{L-n} \frac{\Pochhsymb{1/2-L}{n}\Pochhsymb{1/2}{n}}{\Pochhsymb{(d+1)/2}{n} \, n!} \Hypergeom{2}{1}{-n, 1/2 - L }{1/2-n}{1-v^2} \\
&\phantom{equals=}- \frac{1}{2} \sum_{n=1}^\infty \frac{1}{n} \frac{\Pochhsymb{1/2-L}{n+L}\Pochhsymb{1/2}{n+L}}{\Pochhsymb{(d+1)/2}{n+L} \, (n+L)!} \Hypergeom{2}{1}{-L-n, 1/2 - L}{1/2-L-n}{1-v^2}.
\end{align*}

We remark that the hypergeometric polynomial
\begin{equation*}
f_n(v) \DEF \Hypergeom{2}{1}{-L-n, 1/2 - L}{1/2-L-n}{1-v^2}
\end{equation*}
is strictly monotonically increasing on $(0,1)$ and assumes there its maximum value $1$ at $v = 1$. For example, this can be seen from applying linear transformations of variable \cite[Eq.s~15.8.1 and 15.8.6 15.6.1]{DLMF2011.08.29} and change to integral representation \cite[Eq.~15.6.1]{DLMF2011.08.29} as follows
\begin{align*}
f_n(v) 
&= v^{2L+2n} \Hypergeom{2}{1}{-L-n,-n}{1/2-L-n}{1 - \frac{1}{v^2}} \\
&= \frac{\Pochhsymb{-L-n}{n}}{\Pochhsymb{1/2-L-n}{n}} v^{2L} \Hypergeom{2}{1}{-n,1/2}{L+1}{v^2} \\
&= \frac{\Pochhsymb{-L-n}{n}}{\Pochhsymb{1/2-L-n}{n}} \frac{\gammafcn(L+1)}{\gammafcn(1/2) \gammafcn(L+1/2)} v^{2L} \int_0^1 \frac{t^{1/2-1} \left( 1 - t \right)^{L+1/2-1}}{\left( 1 - v^2 \, t \right)^n} \dd t
\end{align*}
and the fact that the leading coefficients are positive. Hence, the series in the last limit is uniformly convergent with respect to $v$ in $[0,1]$, since $L \geq 1$. 

Putting everything together, we arrive at
\begin{align}
\mathcal{K}_{L+1/2}( \PT{x}, \PT{y} ) 
&= \lim_{\eps \to 0} \left\{ \mathcal{Q}_{\beta-1}( \PT{x}, \PT{y} ) + (-1)^{L+1} c_d(\beta) \left\| \PT{x} - \PT{y} \right\|^{2\beta-1} \right\} \notag \\
&= \lim_{\eps\to0} \left\{ \mathcal{Q}_{\beta-1}( \PT{x}, \PT{y} ) - \frac{1}{2} \frac{1}{\eps} \frac{\Pochhsymb{1/2-L-\eps}{L}\Pochhsymb{1/2}{L}}{\Pochhsymb{(d+1)/2}{L} \, L!} \Hypergeom{2}{1}{-L, 1/2 - L - \eps}{1/2-L}{1-v^2} \right\} \notag \\
\begin{split}
&\phantom{=}+ \lim_{\eps \to 0} \Bigg\{ \frac{1}{2} \frac{1}{\eps} \frac{\Pochhsymb{1/2-L-\eps}{L}\Pochhsymb{1/2}{L}}{\Pochhsymb{(d+1)/2}{L} \, L!} \Hypergeom{2}{1}{-L, 1/2 - L - \eps}{1/2-L}{1-v^2} \\
&\phantom{=\pm}+ (-1)^{L+1} c_d(\beta) \left\| \PT{x} - \PT{y} \right\|^{2\beta-1} \Bigg\}. \notag
\end{split} \\
&= \lim_{\eps\to0} \left\{ \mathcal{Q}_{\beta-1}( \PT{x}, \PT{y} ) - \frac{1}{2} \frac{1}{\eps} \frac{\Pochhsymb{1/2-L-\eps}{L}\Pochhsymb{1/2}{L}}{\Pochhsymb{(d+1)/2}{L} \, L!} \Hypergeom{2}{1}{-L, 1/2 - L - \eps}{1/2-L}{1-v^2} \right\} \notag \\
&\phantom{=}+ \frac{1}{2} \frac{\Pochhsymb{1/2-L}{L}\Pochhsymb{1/2}{L}}{\Pochhsymb{(d+1)/2}{L} \, L!} v^{2L} \lim_{\eps \to 0} \frac{1}{\eps} \Bigg\{ \Hypergeom{2}{1}{-L, \eps}{1/2-L}{1-\frac{1}{v^2}} - 1 \Bigg\} \notag \\
\begin{split}
&\phantom{=}+ \lim_{\eps \to 0} \Bigg\{ \frac{1}{2} \frac{1}{\eps} \frac{\Pochhsymb{1/2-L-\eps}{L}\Pochhsymb{1/2}{L}}{\Pochhsymb{(d+1)/2}{L} \, L!} v^{2L} + (-1)^{L+1} c_d(\beta) \left\| \PT{x} - \PT{y} \right\|^{2\beta-1} \Bigg\}. \label{eq:K.L.b} 
\end{split}
\end{align}
In the last step we use a linear transformation of variable for the hypergeometric function.

First, we observe that
\begin{align*}
v^{2L} \lim_{\eps \to 0} \frac{1}{\eps} \Bigg\{ \Hypergeom{2}{1}{-L, \eps}{1/2-L}{1-\frac{1}{v^2}} - 1 \Bigg\} 
&= \lim_{\eps \to 0} \sum_{k=1}^L \frac{(-1)^k \Pochhsymb{-L}{k} \Pochhsymb{1+\eps}{k-1}}{\Pochhsymb{1/2-L}{k} k!} v^{2(L-k)} \left( 1 - v^2 \right)^k \\
&= \sum_{k=1}^L \frac{(-1)^k}{k} \frac{\Pochhsymb{-L}{k}}{\Pochhsymb{1/2-L}{k}} v^{2(L-k)} \left( 1 - v^2 \right)^k.  
\end{align*}

By the definition of a Pochhammer symbol as a rising factorial, we have
\begin{equation*}
\frac{\Pochhsymb{1/2-L-\eps}{L}}{L!} = (-1)^L \binom{L+1/2+\eps}{L} = (-1)^L \frac{\Pochhsymb{1/2+\eps}{L}}{L!} = (-1)^L \frac{\gammafcn(L+1/2+\eps)}{\gammafcn(1/2+\eps) \gammafcn(L+1)}.
\end{equation*}
Furthermore, the duplication formula for the Gamma function gives
\begin{align*}
c_d( L + 1/2 + \eps ) 
&= \frac{1}{\sin( \pi \eps )} \frac{2^{1-2\beta}}{2 \left( L + \eps \right)} \frac{\gammafcn( ( d + 1 ) / 2 ) \gammafcn( L + 1/2 + \eps ) \gammafcn( L + 1/2 + \eps )}{\gammafcn(L + ( d + 1 ) / 2 + \eps ) \gammafcn( L + \eps )} \\
&= \frac{\gammafcn( \eps ) \gammafcn( 1 - \eps )}{\pi} \frac{2^{1-2\beta}}{2 \left( L + \eps \right)} \frac{\gammafcn( ( d + 1 ) / 2 ) \gammafcn( L + 1/2 + \eps ) \gammafcn( L + 1/2 + \eps )}{\gammafcn(L + ( d + 1 ) / 2 + \eps ) \gammafcn( L + \eps )} \\
&= \frac{\gammafcn( 1 + \eps ) \gammafcn( 1 - \eps )}{\eps \gammafcn(1/2) \gammafcn(1/2)} \frac{2^{1-2\beta}}{2} \frac{\gammafcn( ( d + 1 ) / 2 ) \gammafcn( L + 1/2 + \eps ) \gammafcn( L + 1/2 + \eps )}{\gammafcn(L + ( d + 1 ) / 2 + \eps ) \gammafcn( L + 1 + \eps )}.
\end{align*}
Hence, the expression in curly braces in \eqref{eq:K.L.b} can be rewritten as follows
\begin{align*}
\Bigg\{ \cdots \Bigg\} 
&= \frac{1}{2} \frac{1}{\eps} \frac{\Pochhsymb{1/2-L-\eps}{L}\Pochhsymb{1/2}{L}}{\Pochhsymb{(d+1)/2}{L} \, L!} v^{2L} + (-1)^{L+1} c_d(L+1/2+\eps) \left( 2 v \right)^{2L+2\eps} \\
&= \frac{1}{\eps} \frac{(-1)^L}{2} \frac{\gammafcn((d+1)/2) \gammafcn(L+1/2) \gammafcn(L+1/2+\eps)}{\gammafcn(1/2)  \gammafcn(L+(d+1)/2) \gammafcn(1/2+\eps) \gammafcn(L+1)} v^{2L} \\
&\phantom{=}+ \frac{1}{\eps} \frac{(-1)^{L+1}}{2} \frac{\gammafcn( 1 + \eps ) \gammafcn( 1 - \eps )}{\gammafcn(1/2) \gammafcn(1/2)} \frac{\gammafcn( ( d + 1 ) / 2 ) \gammafcn( L + 1/2 + \eps ) \gammafcn( L + 1/2 + \eps )}{\gammafcn(L + ( d + 1 ) / 2 + \eps ) \gammafcn( L + 1 + \eps )} v^{2L+2\eps} \\
&= \frac{(-1)^{L+1}}{2} \frac{\gammafcn((d+1)/2) \gammafcn(L+1/2) \gammafcn(L+1/2+\eps)}{\gammafcn(1/2)  \gammafcn(L+(d+1)/2) \gammafcn(1/2+\eps) \gammafcn(L+1)} v^{2L} \, \frac{G(\eps) - 1}{\eps},
\end{align*}
where the function
\begin{equation*}
G( \eps ) \DEF \frac{\gammafcn( 1 / 2 + \eps )}{\gammafcn( 1 / 2 )} \frac{\gammafcn( L + 1 )}{\gammafcn( L + 1 + \eps )} \frac{\gammafcn(L + ( d + 1 ) / 2 )}{\gammafcn(L + ( d + 1 ) / 2 + \eps )} \frac{\gammafcn( L + 1 / 2 + \eps )}{\gammafcn( L + 1 / 2 )} \frac{\gammafcn( 1 + \eps )}{\gammafcn(1)} \frac{\gammafcn( 1 - \eps )}{\gammafcn(1)} \, v^{2\eps}
\end{equation*}
assumes the value $1$ at $\eps = 0$ and is differentiable at $0$ with (using product rule)
\begin{align*}
G^\prime(0) 
&= \lim_{\eps\to0} \frac{G(\eps) - 1}{\eps} \\
&= \frac{\gammafcn^\prime( 1/2 + \eps )}{\gammafcn( 1 / 2 )} \Bigg|_{\eps\to0} - \frac{\gammafcn( L + 1 ) \gammafcn^\prime( L + 1 + \eps )}{[ \gammafcn( L + 1 + \eps ) ]^2} \Bigg|_{\eps\to0} \\
&\phantom{=}- \frac{\gammafcn( L + ( d + 1 ) / 2  ) \gammafcn^\prime( L + ( d + 1 ) / 2  + \eps )}{[ \gammafcn( L + ( d + 1 ) / 2  + \eps ) ]^2} \Bigg|_{\eps\to0} \\
&\phantom{=}+ \frac{\gammafcn^\prime( L + 1/2 + \eps )}{\gammafcn( L + 1 / 2 )} \Bigg|_{\eps\to0} + \frac{\gammafcn^\prime( 1 + \eps )}{\gammafcn( 1 )} \Bigg|_{\eps\to0} - \frac{\gammafcn^\prime( 1 - \eps )}{\gammafcn( 1 )} \Bigg|_{\eps\to0} \\
&\phantom{=}+ \ln v^2.
\end{align*}
Using the Digamma function $\digammafcn( z ) = \gammafcn^\prime( z ) / \gammafcn( z )$, we can write
\begin{align*}
G^\prime(0) 
&= \digammafcn( 1 / 2 ) - \digammafcn( L + 1 ) - \digammafcn( L + ( d + 1 ) / 2 ) + \digammafcn( L + 1 / 2 ) + \digammafcn( 1 ) - \digammafcn( 1 ) + \ln v^2 \\
&= 2\ln( 2 v ) - 2 \ln 2 + \digammafcn( 1 / 2 ) + \digammafcn( L + 1 / 2 ) - \digammafcn( L + 1 ) - \digammafcn( L + ( d + 1 ) / 2 ).  
\end{align*}

Thus, we obtain
\begin{align}
&\lim_{\eps \to 0} \Bigg\{ \frac{1}{2} \frac{1}{\eps} \frac{\Pochhsymb{1/2-L-\eps}{L}\Pochhsymb{1/2}{L}}{\Pochhsymb{(d+1)/2}{L} \, L!} v^{2L} + (-1)^{L+1} c_d(\beta) \left\| \PT{x} - \PT{y} \right\|^{2\beta-1} \Bigg\} \notag \\
&\phantom{equals}= \frac{(-1)^{L+1}}{2} \frac{\gammafcn((d+1)/2) \gammafcn(L+1/2) \gammafcn(L+1/2)}{\gammafcn(1/2) \gammafcn(L+(d+1)/2) \gammafcn(1/2) \gammafcn(L+1)} v^{2L} \, G^\prime(0) \notag \\
\begin{split} \label{eq:curly.braces.limit}
&\phantom{equals}= \frac{(-1)^{L+1}}{2} \frac{\Pochhsymb{1/2}{L} \Pochhsymb{1/2}{L}}{\Pochhsymb{(d+1)/2}{L} L!} \, v^{2L} \Big( 2\ln( 2 v ) - 2 \ln 2 + \digammafcn( 1 / 2 ) + \digammafcn( L + 1 / 2 ) \\
&\phantom{equals=\pm}- \digammafcn( L + 1 ) - \digammafcn( L + ( d + 1 ) / 2 ) \Big). 
\end{split}
\end{align}

Putting everything together, we have
\begin{align}
\mathcal{K}_{L+1/2}( \PT{x}, \PT{y} ) 
&= \lim_{\eps\to0} \left\{ \mathcal{Q}_{\beta-1}( \PT{x}, \PT{y} ) - \frac{1}{2} \frac{1}{\eps} \frac{\Pochhsymb{1/2-L-\eps}{L}\Pochhsymb{1/2}{L}}{\Pochhsymb{(d+1)/2}{L} \, L!} \Hypergeom{2}{1}{-L, 1/2 - L - \eps}{1/2-L}{1-v^2} \right\} \notag \\
&\phantom{=}+ \frac{1}{2} \frac{\Pochhsymb{1/2-L}{L}\Pochhsymb{1/2}{L}}{\Pochhsymb{(d+1)/2}{L} \, L!} v^{2L} \lim_{\eps \to 0} \frac{1}{\eps} \Bigg\{ \Hypergeom{2}{1}{-L, \eps}{1/2-L}{1-\frac{1}{v^2}} - 1 \Bigg\} \notag \\
&\phantom{=}+ \lim_{\eps \to 0} \Bigg\{ \frac{1}{2} \frac{1}{\eps} \frac{\Pochhsymb{1/2-L-\eps}{L}\Pochhsymb{1/2}{L}}{\Pochhsymb{(d+1)/2}{L} \, L!} v^{2L} + (-1)^{L+1} c_d(\beta) \left\| \PT{x} - \PT{y} \right\|^{2\beta-1} \Bigg\} \notag \\
&= \frac{1}{2} \sum_{n=0}^{L-1} \frac{1}{L-n} \frac{\Pochhsymb{1/2-L}{n}\Pochhsymb{1/2}{n}}{\Pochhsymb{(d+1)/2}{n} \, n!} \Hypergeom{2}{1}{-n, 1/2 - L }{1/2-n}{1-v^2} \notag \\
&\phantom{=}- \frac{1}{2} \sum_{n=1}^\infty \frac{1}{n} \frac{\Pochhsymb{1/2-L}{n+L}\Pochhsymb{1/2}{n+L}}{\Pochhsymb{(d+1)/2}{n+L} \, (n+L)!} \Hypergeom{2}{1}{-L-n, 1/2 - L}{1/2-L-n}{1-v^2} \label{eq:tilde.Q.infinite.series}\\
&\phantom{=}+ \frac{1}{2} \frac{\Pochhsymb{1/2-L}{L}\Pochhsymb{1/2}{L}}{\Pochhsymb{(d+1)/2}{L} \, L!} \sum_{k=1}^L \frac{(-1)^k}{k} \frac{\Pochhsymb{-L}{k}}{\Pochhsymb{1/2-L}{k}} v^{2(L-k)} \left( 1 - v^2 \right)^k \notag \\
&\phantom{=}+ \frac{(-1)^{L+1}}{2} \frac{\Pochhsymb{1/2}{L} \Pochhsymb{1/2}{L}}{\Pochhsymb{(d+1)/2}{L} L!} \, v^{2L} \Big( 2\ln( 2 v ) - 2 \ln 2 + \digammafcn( 1 / 2 ) + \digammafcn( L + 1 / 2 ) \notag \\
&\phantom{=\pm}- \digammafcn( L + 1 ) - \digammafcn( L + ( d + 1 ) / 2 ) \Big). \notag
\end{align}

The series expansion of $\widetilde{\mathcal{Q}}_{L-1/2}$ follows.

{\bf Asymptotic analysis.} Inspection of the derivation of the asymptotic relation in Theorem~\ref{thm:integral.general.case} shows that one is permitted to take the limit as $s \to L + d/2$ (which corresponds to $2\beta-1\to 2L$. In particular, the function $\Omega(k)$ remains uniformly bounded for sufficiently small $\eps$ ($s = L + d/2 + \eps$) and therefore, by \eqref{eq:A.k.s.d}
\begin{equation*} 
\widetilde{A}_{k}^{(2)}(L+d/2,d) \sim - 2^{d-2} \left[ \frac{\gammafcn( (d+1)/2 ) \gammafcn(L+(d+1)/2) \gammafcn(2L+d)}{\gammafcn( 1/2 - L) \gammafcn(L+1+d/2)} \right]^2 \frac{\gammafcn(d/2)}{\gammafcn(d/2+2L+d)} \, k^{-4s}.
\end{equation*}

This completes the proof.
\end{proof}

\begin{proof}[Proof of Theorem~\ref{thm:identification.thm}]
From the arguments following the definition of the kernel \eqref{eq:kernel} we already know that $\mathcal{K}_\beta( \PT{x}, \PT{y})$ is a symmetric positive definite kernel; that is the expansion
\begin{equation} \label{eq:kernel.aux}
\mathcal{K}_\beta( \PT{x}, \PT{y} ) = \sum_{n=0}^\infty \lambda_n \, Z(d, n) \, P_n^{(d)}(\PT{x} \cdot \PT{y}).
\end{equation}
has positive coefficients $\lambda_0, \lambda_1, \dots$ (cf. I. J. Schoenberg~\cite{Sch1942}). By Theorems~\ref{thm:integral.general.case} and \ref{thm:integral.exceptional.cases} and the remarks following these theorems, the dominant term of $\lambda_n$ as $n \to \infty$ comes from the coefficient in the ultraspherical expansion of $\mathcal{K}_\beta( \PT{x}, \PT{y}) - \mathcal{Q}_{\beta-1}( \PT{x}, \PT{y})$; that is, $\lambda_n \asymp n^{-2s}$, where $s = \beta + (d-1)/2$. 

The Sobolev space $\mathbb{H}^s(\mathbb{S}^d)$ defined by the coefficients $\lambda_0, \lambda_1, \dots,$ in \eqref{eq:kernel.aux} is a reproducing kernel Hilbert space for $s > d/2$ (that is, $\beta > 1/2$) with reproducing kernel $\mathcal{K}_\beta( \PT{x}, \PT{y} )$. By uniqueness of the reproducing kernel Hilbert space $\mathcal{H}_\beta( \mathcal{K}_\beta, \mathbb{S}^d)$ (uniquely defined by $\mathcal{K}_\beta( \PT{x}, \PT{y} )$), $\mathcal{H}_\beta( \mathcal{K}_\beta, \mathbb{S}^d)$ and $\mathbb{H}^s(\mathbb{S}^d)$ coincide.
\end{proof}

\begin{proof}[Proof of Proposition~\ref{prop:inner-most.int}]
Let $d \geq 2$ and $\beta > 0$. By the Funk-Hecke formula
\begin{equation*}
\int_{\mathbb{S}^d}  \left( \PT{y} \cdot \PT{z} - t \right)_+^{\beta - 1} \, \dd \sigma_d( \PT{y} ) = \frac{\omega_{d-1}}{\omega_d} \int_{t}^1 \left( u - t \right)^{\beta-1} \left( 1 - u^2 \right)^{d/2-1} \dd u.
\end{equation*}
The standard substitution $(1-t) v = u - t$ gives
\begin{align*}
\int_{\mathbb{S}^d}  \left( \PT{y} \cdot \PT{z} - t \right)_+^{\beta - 1} \, \dd \sigma_d( \PT{y} ) 
&= \frac{\omega_{d-1}}{\omega_d} \int_0^1 \frac{\left[ \left( 1 - t \right) v \right]^{\beta-1} \left[ \left( 1 - t \right) \left( 1 - v \right) \right]^{d/2-1}}{\left[ \left( 1 + t \right) \left( 1 + \frac{1-t}{1+t} v \right) \right]^{1-d/2}} \left( 1 - t \right) \dd v \\
&= \left( 1 - t \right)^{\beta + d/2 - 1} \left( 1 + t \right)^{d/2-1} \frac{\omega_{d-1}}{\omega_d} \int_0^1 \frac{v^{\beta-1} \left( 1 - v \right)^{d/2-1}}{\left( 1 + \frac{1-t}{1+t} v \right)^{1-d/2}} \dd v.
\end{align*}
By \cite[Eq.~15.6.1]{DLMF2011.08.29} the last integral represents a Gauss hypergeometric function; that is,
\begin{equation*}
\int_{\mathbb{S}^d}  \left( \PT{y} \cdot \PT{z} - t \right)_+^{\beta - 1} \, \dd \sigma_d( \PT{y} ) = \left( 1 - t \right)^{\beta + d/2 - 1} \left( 1 + t \right)^{d/2-1} \frac{\omega_{d-1}}{\omega_d} \frac{\gammafcn(\beta) \gammafcn(d/2)}{\gammafcn(\beta + d/2)} \Hypergeom{2}{1}{1-d/2,\beta}{\beta+d/2}{-\frac{1-t}{1+t}}.
\end{equation*}
A linear transformation of variable (cf. \cite[Eq.~15.8.1]{DLMF2011.08.29}) yields
\begin{equation*}
\int_{\mathbb{S}^d}  \left( \PT{y} \cdot \PT{z} - t \right)_+^{\beta - 1} \, \dd \sigma_d( \PT{y} ) = \left( 1 - t \right)^{\beta + d/2 - 1} \frac{\omega_{d-1}}{\omega_d} \frac{\gammafcn(\beta) \gammafcn(d/2)}{\gammafcn(\beta + d/2)} 2^{d/2-1} \Hypergeom{2}{1}{1-d/2,d/2}{\beta+d/2}{\frac{1-t}{2}}.
\end{equation*}
Using \eqref{eq:C.d} and simplification gives the result.
\end{proof}

\begin{proof}[Proof of Proposition~\ref{prop:double.int.kernel}]
Let $d \geq 2$ and $\beta > 1/2$. Substituting the integral representation \eqref{eq:kernel} into the double integral and interchanging integrals, we get 
\begin{align*}
&\int_{\mathbb{S}^d} \int_{\mathbb{S}^d} \mathcal{K}_\beta( \PT{x}, \PT{y} ) \, \dd \sigma_d( \PT{x} ) \dd \sigma_d( \PT{y} ) \\
&\phantom{equals}= \int_{\mathbb{S}^d} \int_{\mathbb{S}^d} \int_{-1}^1 \int_{\mathbb{S}^d} \left( \PT{x} \cdot \PT{z} - t \right)_+^{\beta - 1} \left( \PT{y} \cdot \PT{z} - t \right)_+^{\beta - 1} \dd \sigma_d( \PT{z} ) \dd t \, \dd \sigma_d( \PT{x} ) \dd \sigma_d( \PT{y} ) \\
&\phantom{equals}= \int_{-1}^1 \int_{\mathbb{S}^d} \left[ \int_{\mathbb{S}^d} \left( \PT{x} \cdot \PT{z} - t \right)_+^{\beta - 1} \dd \sigma_d( \PT{x} ) \right] \left[ \int_{\mathbb{S}^d} \left( \PT{y} \cdot \PT{z} - t \right)_+^{\beta - 1} \dd \sigma_d( \PT{y} ) \right] \dd \sigma_d( \PT{z} ) \dd t.
\end{align*}
By Proposition~\ref{prop:inner-most.int} the integrals in square-brackets do not depend on $\PT{z}$. Hence
\begin{equation*}
\begin{split}
&\int_{\mathbb{S}^d} \int_{\mathbb{S}^d} \mathcal{K}_\beta( \PT{x}, \PT{y} ) \, \dd \sigma_d( \PT{x} ) \dd \sigma_d( \PT{y} ) \\
&\phantom{equals}= 2^{d-2} \left[ \frac{\gammafcn((d+1)/2) \gammafcn(\beta)}{\sqrt{\pi} \, \gammafcn(\beta + d/2)} \right]^2 \int_{-1}^1 \left( 1 - t \right)^{2\beta + d - 2} \left[ \Hypergeom{2}{1}{1-d/2,d/2}{\beta+d/2}{\frac{1-t}{2}} \right]^2 \dd t.
\end{split}
\end{equation*}
The standard change of variable $2x = 1 - t$ gives
\begin{equation*}
\begin{split}
&\int_{\mathbb{S}^d} \int_{\mathbb{S}^d} \mathcal{K}_\beta( \PT{x}, \PT{y} ) \, \dd \sigma_d( \PT{x} ) \dd \sigma_d( \PT{y} ) \\
&\phantom{equals}= 2^{d-2} \left[ \frac{\gammafcn((d+1)/2) \gammafcn(\beta)}{\sqrt{\pi} \, \gammafcn(\beta + d/2)} \right]^2 2^{2\beta+d-1} \int_{0}^1 x^{2\beta + d - 2} \left[ \Hypergeom{2}{1}{1-d/2,d/2}{\beta+d/2}{x} \right]^2 \dd x.
\end{split}
\end{equation*}

For $d = 2$ the hypergeometric function above reduces to $1$ and therefore
\begin{equation*}
\int_{\mathbb{S}^2} \int_{\mathbb{S}^2} \mathcal{K}_\beta( \PT{x}, \PT{y} ) \, \dd \sigma_2( \PT{x} ) \dd \sigma_2( \PT{y} ) = \left[ \frac{\gammafcn(1+1/2) \gammafcn(\beta)}{\sqrt{\pi} \, \gammafcn(\beta + 1)} \right]^2 2^{2\beta+1} \int_{0}^1 x^{2\beta} \dd x = \left[ \frac{1}{2 \beta} \right]^2 \frac{2^{2\beta+1}}{2\beta+1}.
\end{equation*}

For $d \geq 2$ we apply Proposition~\ref{prop:integral.identity} to obtain 
\begin{equation*}
\begin{split}
\int_{\mathbb{S}^d} \int_{\mathbb{S}^d} \mathcal{K}_\beta( \PT{x}, \PT{y} ) \, \dd \sigma_d( \PT{x} ) \dd \sigma_d( \PT{y} ) 
&= 2^{d-2} \left[ \frac{\gammafcn((d+1)/2) \gammafcn(\beta)}{\sqrt{\pi} \, \gammafcn(\beta + d/2)} \right]^2 2^{2\beta+d-1} \frac{\gammafcn(2\beta+d-1) \gammafcn(1)}{\gammafcn(\gammafcn(2\beta+d)} \\
&\phantom{=\pm}\times \KampedeFerietA{1:1;1}{1:2;2}{\displaystyle 2\beta + d - 1 \\ \displaystyle 2\beta + d}{\displaystyle 1 - d/2, d/2 \\ \displaystyle \beta + d/2}{\displaystyle 1 - d/2, d/2 \\ \displaystyle \beta + d/2}{1,1}.
\end{split} 
\end{equation*}
For $d$ an even positive integer, the series expansion of the given Kamp{\'e} de F{\'e}riet function terminates after finitely many terms (cf. \eqref{eq:Kampe.de.Feriet.function}). The result follows.
\end{proof}

\begin{proof}[Proof of Proposition~\ref{prop:case.beta.EQ.integer}]
From \eqref{eq:K.beta.2..beta.EQ.integer} we have that
\begin{equation*}
\mathcal{Q}_{M-1}( \PT{x}, \PT{y} ) = \frac{1}{2M-1} \sum_{n=0}^{M-1} \frac{\Pochhsymb{1-M}{n} \Pochhsymb{1/2 - M}{n}}{\Pochhsymb{(d+1)/2}{n} \, n!} \Hypergeom{2}{1}{-n,1-M}{3/2-M}{\frac{1 - \PT{x} \cdot \PT{y}}{2}}.
\end{equation*}
Expanding the hypergeometric polynomial and using that $\Pochhsymb{-n}{m} / n! = (-1)^m / (n - m)!$, we obtain after rearranging terms
\begin{equation*}
\mathcal{Q}_{M-1}( \PT{x}, \PT{y} ) = \frac{1}{2M-1} \sum_{m=0}^{M-1} \frac{(-1)^m \Pochhsymb{1-M}{m}}{\Pochhsymb{3/2-M}{m} \, m!} \left[ \sum_{n = m}^{M-1} \frac{\Pochhsymb{1-M}{n} \Pochhsymb{1/2-M}{n}}{\Pochhsymb{(d+1)/2}{n} (n-m)!} \right] \left( \frac{1 - \PT{x} \cdot \PT{y}}{2} \right)^m.
\end{equation*}
The square-bracketed expression is a hypergeometric polynomial evaluated at unity and can be reduced as follows (cf. \cite[Eq.~15.4.20]{DLMF2011.08.29}):
\begin{align*}
&\frac{\Pochhsymb{1-M}{m} \Pochhsymb{1/2-M}{m}}{\Pochhsymb{(d+1)/2}{m}} \Hypergeom{2}{1}{m+1-M,m+1/2-M}{m+(d+1)/2}{1} \\
&\phantom{equals}= \frac{\Pochhsymb{1-M}{m} \Pochhsymb{1/2-M}{m}}{\Pochhsymb{(d+1)/2}{m}} \frac{\gammafcn( m + (d+1)/2 ) \gammafcn( 2M - m + d/2 - 1)}{\gammafcn( M - 1 + (d+1)/2 ) \gammafcn( M + d/2 )} \\
&\phantom{equals}= (-1)^m 2^{2M-1} \frac{\Pochhsymb{d/2}{2M-1}}{\Pochhsymb{d}{2M-1}} \frac{\Pochhsymb{1-M}{m} \Pochhsymb{1/2-M}{m}}{\Pochhsymb{2-d/2-2M}{m}}.
\end{align*}
Thus
\begin{align} 
\mathcal{Q}_{M-1}( \PT{x}, \PT{y} ) 
&= \frac{2^{2M-1}}{2M-1} \frac{\Pochhsymb{d/2}{2M-1}}{\Pochhsymb{d}{2M-1}} \sum_{m=0}^{M-1} \frac{\Pochhsymb{1-M}{m} \Pochhsymb{1-M}{m} \Pochhsymb{1/2-M}{m}}{\Pochhsymb{3/2-M}{m} \Pochhsymb{2-d/2-2M}{m} \, m!} \left( \frac{1 - \PT{x} \cdot \PT{y}}{2} \right)^m \\
&= \frac{2^{2M-1}}{2M-1} \frac{\Pochhsymb{d/2}{2M-1}}{\Pochhsymb{d}{2M-1}} \Hypergeom{3}{2}{1-M,1/2-M,1-M}{3/2-M,2-d/2-2M}{\frac{1 - \PT{x} \cdot \PT{y}}{2}}. \label{eq:Q.M.m.1.aux}
\end{align}

The constant $c_d(M)$ can be obtained from Theorem~\ref{thm:integral.general.case}.

Since $\mathcal{Q}_{M-1}( \PT{x}, \PT{y} )$ is a zonal function (depending on $\PT{x} \cdot \PT{y}$), using \eqref{eq:Q.M.m.1.aux}, by the Funk-Hecke formula and subsequent change of variable $2v = 1 - u$ we get
\begin{align*}
&\int_{\mathbb{S}^d} \mathcal{Q}_{M-1}( \PT{x}, \PT{y} ) \dd \sigma_d( \PT{x} ) \\
&\phantom{equals}= \frac{2^{2M-1}}{2M-1} \frac{\Pochhsymb{d/2}{2M-1}}{\Pochhsymb{d}{2M-1}} \frac{\omega_{d-1}}{\omega_d} \int_{-1}^1 \left( 1 - u^2 \right)^{d/2-1} \Hypergeom{3}{2}{1-M,1/2-M,1-M}{3/2-M,2-d/2-2M}{\frac{1 - u}{2}} \dd u \\
&\phantom{equals}= \frac{2^{2M-1}}{2M-1} \frac{\Pochhsymb{d/2}{2M-1}}{\Pochhsymb{d}{2M-1}} 2^{d-1} \frac{\omega_{d-1}}{\omega_d} \int_{0}^1 v^{d/2-1} \left( 1 - v \right)^{d/2-1} \Hypergeom{3}{2}{1-M,1/2-M,1-M}{3/2-M,2-d/2-2M}{v} \dd v \\
&\phantom{equals}= \frac{2^{2M-1}}{2M-1} \frac{\Pochhsymb{d/2}{2M-1}}{\Pochhsymb{d}{2M-1}} 2^{d-1} \frac{\omega_{d-1}}{\omega_d} \frac{\gammafcn(d/2) \gammafcn(d/2)}{\gammafcn(d)} \Hypergeom{4}{3}{1-M,1/2-M,1-M,d/2}{3/2-M,2-d/2-2M,d}{1}.
\end{align*}
The last step follows from the terminating series expansion of the hypergeometric polynomial or \cite[Eq.~7.2.3.9]{PrBrMa1990III}. Using \eqref{eq:C.d} and simplification gives
\begin{equation*}
\int_{\mathbb{S}^d} \mathcal{Q}_{M-1}( \PT{x}, \PT{y} ) \dd \sigma_d( \PT{x} ) = \frac{2^{2M-1}}{2M-1} \frac{\Pochhsymb{d/2}{2M-1}}{\Pochhsymb{d}{2M-1}} \Hypergeom{4}{3}{1-M,1/2-M,1-M,d/2}{3/2-M,2-d/2-2M,d}{1}.
\end{equation*}
The $_4\HyperF_3$-hypergeometric is Saalschutzian (that is, $1$ + the sum of the upper parameters equals he sum of the lower parameters). Application of Whipple's transformation \cite[Eq.~16.4.14]{DLMF2011.08.29} yields
\begin{equation*}
\begin{split}
&\Hypergeom{4}{3}{1-M,1/2-M,1-M,d/2}{d,3/2-M,2-d/2-2M}{1} = \frac{\Pochhsymb{1}{M-1} \Pochhsymb{3/2-d/2-M}{M-1}}{\Pochhsymb{2-d/2-2M}{M-1} \Pochhsymb{3/2-M}{M-1}} \\
&\phantom{equals\pm}\times \Hypergeom{4}{3}{1-M,1/2-M,M+d-1,d/2}{d,1-M,(d+1)/2}{1}.
\end{split}
\end{equation*}
(Note that one must not simplify the formula by ``cancelling'' $1-M$ in the hypergeometric polynomial (evaluated at $1$) at the right-hand side.) From
\begin{equation*}
\frac{2^{2M-1}}{2M-1} \frac{\Pochhsymb{d/2}{2M-1}}{\Pochhsymb{d}{2M-1}} \frac{\Pochhsymb{1}{M-1} \Pochhsymb{3/2-d/2-M}{M-1}}{\Pochhsymb{2-d/2-2M}{M-1} \Pochhsymb{3/2-M}{M-1}} = (-1)^{M-1} \frac{(M-1)!}{\Pochhsymb{3/2}{M-1}}
\end{equation*}
we obtain the result.

For $d = 2$, we get with the help of Mathematica
\begin{align*}
\int_{\mathbb{S}^2} \int_{\mathbb{S}^2} \mathcal{Q}_{M-1}( \PT{x}, \PT{y} ) \, \dd \sigma_2( \PT{x} ) \dd \sigma_2( \PT{y} ) 
&= (-1)^{M-1} \frac{(M-1)!}{\Pochhsymb{3/2}{M-1}} \Hypergeom{4}{3}{1-M,1/2-M,M+1,1}{1-M,3/2,2}{1} \\
&= (-1)^{M-1} \frac{(M-1)!}{\Pochhsymb{3/2}{M-1}} \sum_{n=0}^{M-1} \frac{\Pochhsymb{1/2-M}{n} \Pochhsymb{M+1}{n}}{\Pochhsymb{3/2}{n} \Pochhsymb{2}{n}} \\
&= \frac{2^{2M}}{M^2 \left( 2 M + 1 \right)^2} - \frac{\gammafcn(1/2-M) \gammafcn(M)}{M \left( 2 M + 1 \right)^2 \sqrt{\pi}} \\
&= \frac{2^{2M}}{M^2 \left( 2 M + 1 \right)^2} + \frac{(-1)^{M-1}}{M \left( 2 M + 1 \right)^2} \frac{(M-1)!}{\Pochhsymb{1/2}{M}}.
\end{align*}
\end{proof}

\appendix

\section{The reproducing kernel Hilbert space $\mathcal{H}_\beta(\mathcal{K}_{\beta}, \mathbb{S}^d)$}\label{sec_appendix_a}

In this section we show that all functions in $\mathcal{H}_\beta(\mathcal{K}_\beta, \mathbb{S}^d)$ have an integral representation of the form \eqref{eq:integral.representation.f.i}. Note that the set of functions $\sum_{i=1}^n a_i \mathcal{K}_\beta(\cdot, \PT{y}_i)$ is dense in $\mathcal{H}_\beta$. Let $f \in \mathcal{H}_\beta$. Then there exists a sequence of functions $f_1, f_2, \ldots \in \mathcal{H}_\beta$ with $$f_n(\PT{x}) = \sum_{i=1}^n a_{i,n} \mathcal{K}_\beta(\PT{x}, \PT{y}_{i,n})$$ such that $\|f-f_n\|_{\mathcal{K}_\beta} \rightarrow 0$ as $n \rightarrow \infty$. Further let
\begin{equation*}
g_n(\PT{z}, t) =  \sum_{i=1}^n a_{i,n} (\PT{z} \cdot \PT{y}_{i,n} - t)_+^{\beta-1}.
\end{equation*}
Thus
\begin{equation*}
f_n(\PT{x}) = \int_{-1}^1 \int_{\mathbb{S}^d} g_n(\PT{z},t) (\PT{z}  \cdot \PT{x} - t)_+^{\beta-1} \,\mathrm{d} \sigma_d(\PT{z}) \,\mathrm{d} t.
\end{equation*}
Then elementary algebra shows that
\begin{align*}
\|f_n-f_m\|_{\mathcal{K}_\beta}^2 = \langle f_n-f_m, f_n-f_m \rangle_{\mathcal{K}_\beta} & =\int_{-1}^1 \int_{\mathbb{S}^d} \left|g_n(\PT{z},t) - g_m(\PT{z},t) \right|^2 \,\mathrm{d} \sigma_d(\PT{z}) \,\mathrm{d} t.
\end{align*}
Since the sequence $(f_n)$ converges in the Hilbert space $\mathcal{H}_\beta$, it is a Cauchy sequence. Thus $\|f_n-f_m\|_{\mathcal{K}_\beta} \to 0$ as $n,m \rightarrow \infty$ and therefore the sequence $(g_n)$ is a Cauchy sequence in the Hilbert space $L_2(\mathbb{S}^d \times [-1,1])$. Thus the sequence $(g_n)$ converges to a function $g$ in $L_2(\mathbb{S}^d \times [-1,1])$. Let
\begin{equation*}
\widetilde{f}(\PT{x}) = \int_{-1}^1 \int_{\mathbb{S}^d} g(\PT{z},t)  (\PT{z} \cdot \PT{x} - t)_+^{\beta-1} \, \mathrm{d} \sigma_d(\PT{z}) \,\mathrm{d} t.
\end{equation*}
Then we have
\begin{equation*}
0 \le \|f-\widetilde{f}\|_{\mathcal{K}_\beta} \le \|f-f_n\|_{\mathcal{K}_\beta} + \|f_n - \widetilde{f}\|_{\mathcal{K}_\beta}.
\end{equation*}
We have $\|f-f_n\|_{\mathcal{K}_\beta} \to 0$ as $n \to \infty$ by the definition of the $f_n$ and
\begin{equation*}
\|f_n - \widetilde{f}\|_{\mathcal{K}_\beta}^2 = \int_{-1}^1 \int_{\mathbb{S}^d} \left|g_n(\PT{z},t) - g(\PT{z},t) \right|^2 \,\mathrm{d} \sigma_d(\PT{z}) \,\mathrm{d} t \to 0
\end{equation*}
as $n \to \infty$, since $g$ is defined as the limit of the sequence $(g_n)$ in $L_2(\mathbb{S}^d \times [-1,1])$. Thus we have $\|f-\widetilde{f}\|_{\mathcal{K}_\beta} = 0$, which shows the statement.

\section{Auxiliary results}

\begin{prop} \label{prop:Appell.F.identity}
Let $| a / A | + | b / A | < 1$. Then
\begin{equation*}
\AppellFtwo{\alpha, \beta, \beta^\prime}{2\beta, 2 \beta^\prime}{\frac{a}{A}, \frac{b}{A}} = \left( \frac{2A}{2A-a-b} \right)^\alpha \AppellFfour{\alpha/2, ( \alpha + 1 ) / 2}{\beta+1/2,\beta^\prime+1/2}{\left( \frac{a}{2A - a - b} \right)^2, \left( \frac{b}{2A - a - b} \right)^2}.
\end{equation*}
\end{prop}

\begin{proof}
Since $| a / A | + | b / A | < 1$, one can use the series expansion of the Appell function. Reordering gives
\begin{align*}
f( a/A, b/A ) 
&= \sum_{m=0}^\infty \sum_{n=0}^\infty \frac{\Pochhsymb{\alpha}{m+n} \Pochhsymb{\beta}{m} \Pochhsymb{\beta^\prime}{m}}{\Pochhsymb{2\beta}{m} \Pochhsymb{2\beta^\prime}{n} \, m! n!} \left( \frac{a}{A} \right)^m \left( \frac{b}{A} \right)^n \\
&= \sum_{m=0}^\infty \frac{\Pochhsymb{\alpha}{m} \Pochhsymb{\beta}{m}}{\Pochhsymb{2\beta}{m} \, m!} \left[ \sum_{n=0}^\infty \frac{\Pochhsymb{m+\alpha}{n} \Pochhsymb{\beta^\prime}{m}}{\Pochhsymb{2\beta^\prime}{n} \, n!} \left( \frac{b}{A} \right)^n \right] \left( \frac{a}{A} \right)^m \\
&= \sum_{m=0}^\infty \frac{\Pochhsymb{\alpha}{m} \Pochhsymb{\beta}{m}}{\Pochhsymb{2\beta}{m} \, m!} \Hypergeom{2}{1}{m+\alpha, \beta^\prime}{2\beta^\prime}{\frac{b}{A}} \left( \frac{a}{A} \right)^m.
\end{align*}
Since $|\phase( 1 - b / A ) | < \pi$, the quadratic transformation of variable \cite[Eq.~15.8.13]{DLMF2011.08.29} can be applied; that is
\begin{equation*}
f( a/A, b/A ) = \left( \frac{2A}{2A-b} \right)^\alpha \sum_{m=0}^\infty \frac{\Pochhsymb{\alpha}{m} \Pochhsymb{\beta}{m}}{\Pochhsymb{2\beta}{m} \, m!} \Hypergeom{2}{1}{( m+\alpha ) / 2, ( m+\alpha+1 ) / 2}{\beta^\prime + 1/2}{\left( \frac{b}{2A-b} \right)^2} \left( \frac{a}{A} \right)^m.
\end{equation*}
Invoking the series expansion of the hypergeometric function and identities for the Pochhammer symbol, one gets
\begin{align*}
\frac{f( a/A, b/A )}{\left( \frac{2A}{2A-b} \right)^\alpha} 
&= \sum_{m=0}^\infty \sum_{n=0}^\infty \frac{\Pochhsymb{\alpha}{m} \Pochhsymb{\beta}{m} \Pochhsymb{( m + \alpha ) / 2}{n} \Pochhsymb{( m + \alpha + 1 ) / 2}{n}}{\Pochhsymb{2\beta}{m} \Pochhsymb{\beta^\prime + 1/2}{n} \, m! n!} \left( \frac{2a}{2A-b} \right)^m \left( \frac{b}{2A-b}\right)^{2n} \\
&= \sum_{m=0}^\infty \sum_{n=0}^\infty \frac{\Pochhsymb{\alpha}{m} \Pochhsymb{\beta}{m} 2^{-2n} \Pochhsymb{m + \alpha}{2n} }{\Pochhsymb{2\beta}{m} \Pochhsymb{\beta^\prime + 1/2}{n} \, m! n!} \left( \frac{2a}{2A-b} \right)^m \left( \frac{b}{2A-b}\right)^{2n} \\
&= \sum_{m=0}^\infty \sum_{n=0}^\infty \frac{\Pochhsymb{\alpha}{m+2n} \Pochhsymb{\beta}{m}}{\Pochhsymb{2\beta}{m} \Pochhsymb{\beta^\prime + 1/2}{n} \, m! n!} \left( \frac{2a}{2A-b} \right)^m \left( \frac{b/2}{2A-b}\right)^{2n} \\
&= \sum_{n=0}^\infty \frac{\Pochhsymb{\alpha}{2n}}{\Pochhsymb{\beta^\prime+1/2}{n} \, n!} \Hypergeom{2}{1}{2n+\alpha,\beta}{2\beta}{\frac{2a}{2A-b}} \left( \frac{b/2}{2A-b}\right)^{2n}.
\end{align*}
Application of \cite[Eq.~15.8.13]{DLMF2011.08.29} gives
\begin{equation*}
\begin{split}
f( a/A, b/A ) 
&= \left( \frac{2A}{2A-a-b} \right)^\alpha \sum_{n=0}^\infty \frac{\Pochhsymb{\alpha}{2n}}{\Pochhsymb{\beta^\prime+1/2}{n} \, n!} \\
&\phantom{=\pm}\times \Hypergeom{2}{1}{( 2n + \alpha ) / 2, ( 2n + \alpha + 1 ) / 2}{\beta+1/2}{\left( \frac{2a}{4A - 2a - 2a} \right)^2} \left( \frac{b/2}{2A-a-b} \right)^{2n}.
\end{split}
\end{equation*}
Reversing to series expansions and invoking properties of the Pochhammer symbol, we arrive at
\begin{align*}
\frac{f( a/A, b/A )}{\left( \frac{2A}{2A-a-b} \right)^\alpha} 
&= \sum_{m=0}^\infty \sum_{n=0}^\infty \frac{\Pochhsymb{\alpha}{2m+2n}}{\Pochhsymb{\beta+1/2}{m} \Pochhsymb{\beta^\prime+1/2}{m} \, m! n!} \left( \frac{a/2}{2A-a-b} \right)^{2m} \left( \frac{b/2}{2A-a-b} \right)^{2n} \\
&= \sum_{m=0}^\infty \sum_{n=0}^\infty \frac{\Pochhsymb{\alpha/2}{m+n} \Pochhsymb{(\alpha+1)/2}{m+n}}{\Pochhsymb{\beta+1/2}{m} \Pochhsymb{\beta^\prime+1/2}{m} \, m! n!} \left( \frac{a}{2A-a-b} \right)^{2m} \left( \frac{b}{2A-a-b} \right)^{2n}.
\end{align*}
By \cite[Eq.~16.13.4]{DLMF2011.08.29} the double sum is an Appel $\HyperF_4$-function. The result follows.
\end{proof}

Next, we expand a hypergeometric polynomial in terms of normalized Gegenbauer polynomials $P_k^{(d)}(t) = \GegenbauerC_k^{(d-1)/2}(t) / \GegenbauerC_k^{(d-1)/2}(1)$.

\begin{prop} \label{prop:hypergeometric.polynomial.expansion}
Let $c$ be not an integer $\leq 0$. Then
\begin{equation*}
Q_n(t) \DEF \Hypergeom{2}{1}{-n,b}{c}{\frac{1-t}{2}} = \sum_{k=0}^n \frac{a_{k}}{Z(d,k)} \, Z(d,k) P_k^{(d)}(t),
\end{equation*}
where
\begin{equation*}
\frac{a_{k}}{Z(d,k)} = \frac{(-1)^k\Pochhsymb{-n}{k} \Pochhsymb{b}{k}\Pochhsymb{d/2}{k}}{\Pochhsymb{c}{k}\Pochhsymb{d}{2k}} \Hypergeom{3}{1}{k-n,k+b,k+d/2}{k+c,2k+d}{1}.
\end{equation*}
\end{prop}

\begin{proof}
The polynomial $Q_n(t)$ has degree $n$ and can therefore be expressed in terms of the family of normalized Gegenbauer polynomials (ultraspherical polynomials) $\{ P_0^{(d)}, \dots, P_n^{(d)}\}$
\begin{equation*}
Q_n(t) = \sum_{k=0}^n a_{k} \, P_k^{(d)}(t) = \sum_{k=0}^n \frac{a_{k}}{Z(d,k)} \, Z(d,k) P_k^{(d)}(t),
\end{equation*}
where the coefficients $a_k$ depend on $Q_n$ and
\begin{equation*}
\frac{a_{k}}{Z(d,k)} = \frac{\omega_{d-1}}{\omega_d} \int_{-1}^1 Q_n(t) \, P_k^{(d)}(t) \left( 1 - t^2 \right)^{d/2-1} \dd t.
\end{equation*}
(Note that $\frac{\omega_{d-1}}{\omega_d} \int_{-1}^1 P_k^{(d)}(t) P_k^{(d)}(t) ( 1 - t^2 )^{d/2-1} \dd t = 1 / Z(d,k)$; cf., for example, \cite{DLMF2011.08.29}.)

Using Rodrigues' formulas (cf. \cite[Eq.~18.5.5]{DLMF2011.08.29})
\begin{equation*}
P_k^{(d)}(t) \left( 1 - t^2 \right)^{d/2-1}= \frac{(-1)^k}{2^k \Pochhsymb{d/2}{k}} \left\{ \left( 1 - t^2 \right)^{k + d/2 - 1} \right\}^{(k)}
\end{equation*}
and $k$-times successive integration by parts, we arrive at
\begin{align*}
\frac{a_{k}}{Z(d,k)} 
&= \frac{\omega_{d-1}}{\omega_d} \frac{(-1)^k}{2^k \Pochhsymb{d/2}{k}} \int_{-1}^1 Q_n(t) \left\{ \left( 1 - t^2 \right)^{k + d/2 - 1} \right\}^{(k)} \dd t \\
&= \frac{\omega_{d-1}}{\omega_d} \frac{1}{2^k \Pochhsymb{d/2}{k}} \int_{-1}^1 Q_n^{(k)}(t) \left( 1 - t^2 \right)^{k + d/2 - 1} \dd t.
\end{align*}

The differentiation rule for hypergeometric functions \cite[Eq.~15.5.1]{DLMF2011.08.29} yields
\begin{equation*}
Q_n^{(k)}(t) = \frac{(-1)^k}{2^k} \frac{\Pochhsymb{-n}{k} \Pochhsymb{b}{k}}{\Pochhsymb{c}{k}} \Hypergeom{2}{1}{k-n,k+b}{k+c}{\frac{1-t}{2}}.
\end{equation*}

Using the (terminating) series expansion of $Q_n^{(k)}$, we get 
\begin{equation*}
\frac{a_{k}}{Z(d,k)} = \frac{\omega_{d-1}}{\omega_d} \frac{1}{2^k \Pochhsymb{d/2}{k}} \frac{(-1)^k}{2^k} \frac{\Pochhsymb{-n}{k} \Pochhsymb{b}{k}}{\Pochhsymb{c}{k}} \sum_{\nu=0}^{n-k} \frac{\Pochhsymb{k-n}{\nu} \Pochhsymb{k+b}{\nu}}{\Pochhsymb{k+c}{\nu} \nu!} \, b_\nu,
\end{equation*}
where the standard substitution $2x = 1 - t$ leads to a beta integral
\begin{align*}
b_\nu 
&\DEF \int_{-1}^1 \left( \frac{1 - t}{2} \right)^\nu \left( 1 - t^2 \right)^{k + d/2 - 1} \dd t \\
&= 2^{2k+d-1} \int_0^1 x^{\nu + k + d/2 - 1} \left( 1 - x \right)^{k + d/2 - 1} \dd x \\
&= 2^{2k+d-1} \frac{\gammafcn(\nu + k + d/2) \gammafcn( k + d / 2 )}{\gammafcn( \nu + 2 k + d)} \\
&= 2^{2k+d-1} \frac{\gammafcn(k + d/2) \gammafcn( d / 2 )}{\gammafcn( 2 k + d)} \frac{\Pochhsymb{k + d/2}{\nu} \Pochhsymb{d / 2}{k}}{\Pochhsymb{2 k + d}{\nu}} \\
&= 2^{d-1} \frac{\gammafcn(d/2) \gammafcn( d / 2 )}{\gammafcn( d )} 2^{2k} \frac{\Pochhsymb{d/2}{k}}{\Pochhsymb{d}{2k}} \frac{\Pochhsymb{k + d/2}{\nu} \Pochhsymb{d / 2}{k}}{\Pochhsymb{2 k + d}{\nu}}.
\end{align*}
Hence
\begin{align*}
\frac{a_{k}}{Z(d,k)} 
&= 2^{d-1} \frac{\omega_{d-1}}{\omega_d} \frac{\gammafcn(d/2) \gammafcn( d / 2 )}{\gammafcn( d )} \frac{(-1)^k\Pochhsymb{-n}{k} \Pochhsymb{b}{k}\Pochhsymb{d/2}{k}}{\Pochhsymb{c}{k}\Pochhsymb{d}{2k}}  \sum_{\nu=0}^{n-k} \frac{\Pochhsymb{k-n}{\nu} \Pochhsymb{k+b}{\nu}\Pochhsymb{k + d/2}{\nu}}{\Pochhsymb{k+c}{\nu} \Pochhsymb{2 k + d}{\nu} \nu!} \\
&= \frac{(-1)^k\Pochhsymb{-n}{k} \Pochhsymb{b}{k}\Pochhsymb{d/2}{k}}{\Pochhsymb{c}{k}\Pochhsymb{d}{2k}} \Hypergeom{3}{1}{k-n,k+b,k+d/2}{k+c,2k+d}{1}.
\end{align*}
\end{proof}

As a corollary to the last result we have the following
\begin{prop} \label{prop:hypergeometric.polynomial.expansion.specialized}
Let $3/2-\beta$ be not an integer $\leq 0$. Then
\begin{equation*}
\Hypergeom{2}{1}{-n,1-\beta}{3/2-\beta}{\frac{1-t}{2}} = \sum_{k=0}^n \frac{a_{k}}{Z(d,k)} \, Z(d,k) P_k^{(d)}(t),
\end{equation*}
where
\begin{equation*}
\frac{a_{k,n}}{Z(d,k)} = \frac{n! \Pochhsymb{1-\beta}{k}\Pochhsymb{d/2}{k}}{\Pochhsymb{3/2-\beta}{k}\Pochhsymb{d}{2k} (n-k)!} \Hypergeom{3}{2}{k-n,k+1-\beta,k+d/2}{k+3/2-\beta,2k+d}{1}.
\end{equation*}
Let $k + 1 - \beta > 0$. Then
\begin{equation*}
\begin{split}
\frac{a_{k,n}}{Z(d,k)} 
&= \frac{\gammafcn(3/2 - \beta) \gammafcn(d)}{\gammafcn(1 - \beta) \gammafcn(1/2) \gammafcn(d/2) \gammafcn(d/2)} \frac{n!}{\Pochhsymb{d/2}{k} (n-k)!} \\
&\phantom{=\pm}\times \int_0^1 \int_0^1 t^{k + 1 - \beta-1} x^{k + d/2-1} \left( 1 - t \right)^{1/2-1} \left( 1 - x \right)^{k + d/2-1} \left( 1 - x t \right)^{n-k} \dd x \dd t.
\end{split}
\end{equation*}
\end{prop}

\begin{proof} 
On observing that $(-1)^k \Pochhsymb{-n}{k} = n! / (n - k)!$, the first form of the coefficients $a_{k,n} / Z(d,k)$ follows from Proposition~\ref{prop:hypergeometric.polynomial.expansion}. 

Let $k + 1 - \beta > 0$. Then a hypergeometric function as in the coefficient $a_{k,n} / Z(d,k)$ can be expressed as follows (cf. \cite[Eq.~16.5.2]{DLMF2011.08.29} and \cite[Eq.~15.6.1]{DLMF2011.08.29}). For $b > 0$ and $c > 0$ 
\begin{align*}
\Hypergeom{3}{2}{-n,b,c}{b+1/2,2c}{1} 
&= \frac{\gammafcn(b+1/2)}{\gammafcn(b) \gammafcn(1/2)} \int_0^1 t^{b-1} \left( 1 - t \right)^{1/2-1} \Hypergeom{2}{1}{-n, c}{2c}{t} \dd t \\
&= \frac{\gammafcn(b+1/2)}{\gammafcn(b) \gammafcn(1/2)} \frac{\gammafcn( 2c )}{\gammafcn(c) \gammafcn(c)} \int_0^1 \int_0^1 \frac{t^{b-1} x^{c-1} \left( 1 - t \right)^{1/2-1} \left( 1 - x \right)^{c-1}}{\left( 1 - x t \right)^{-n}} \dd x \dd t
\end{align*}
That is (with $b = k + 1 - \beta$, $c = k + d/2$, and $n = n - k$)
\begin{align*}
&\Hypergeom{3}{1}{k-n,k+1-\beta,k+d/2}{k+3/2-\beta,2k+d}{1} = \frac{\gammafcn(k + 3/2 - \beta)}{\gammafcn(k + 1 - \beta) \gammafcn(1/2)} \frac{\gammafcn( 2k + d )}{\gammafcn(k + d/2) \gammafcn(k + d/2)} \\
&\phantom{equals\pm}\times \int_0^1 \int_0^1 t^{k + 1 - \beta-1} x^{k + d/2-1} \left( 1 - t \right)^{1/2-1} \left( 1 - x \right)^{k + d/2-1} \left( 1 - x t \right)^{n-k} \dd x \dd t.
\end{align*}
The coefficient in front of the double integral equals
\begin{equation*}
\frac{\gammafcn(3/2 - \beta) \gammafcn(d)}{\gammafcn(1 - \beta) \gammafcn(1/2) \gammafcn(d/2) \gammafcn(d/2)} \frac{\Pochhsymb{3/2 - \beta}{k}}{\Pochhsymb{1 - \beta}{k}} \frac{\Pochhsymb{d}{2k}}{\Pochhsymb{d/2}{k}\Pochhsymb{d/2}{k}}.
\end{equation*}
The second form of the coefficient $a_{k,n} / Z(d,k)$ follows.
\end{proof}

\begin{prop} \label{prop:integral.identity}
Let $\re \gamma > \re \beta > 0$ and $\re( c - a - b ) > 0$ and $\re( c^\prime - a^\prime - b^\prime ) > 0$. Then
\begin{equation*}
\int_0^1 t^{\beta-1} \left( 1 - t \right)^{\gamma - \beta - 1} \Hypergeom{2}{1}{a,b}{c}{t} \Hypergeom{2}{1}{a^\prime,b^\prime}{c^\prime}{t} \dd t = \frac{\gammafcn(\beta) \gammafcn(\gamma-\beta)}{\gammafcn(\gamma)} \KampedeFerietA{1:1;1}{1:2;2}{\displaystyle \beta \\ \displaystyle \gamma}{\displaystyle a, b \\ \displaystyle c}{ \displaystyle a^\prime, b^\prime \\ \displaystyle c^\prime}{1,1}.
\end{equation*}
\end{prop}

\begin{proof}
Under the given assumptions the hypergeometric functions have series expansions that are uniformly and absolutely convergent for all $x$ with $|x| \leq 1$. Let $A$ denote the integral we want to compute. Substituting these series expansions and interchanging summation with integration, we get 
\begin{equation*}
A = \sum_{m=0}^\infty \sum_{n=0}^\infty \frac{\Pochhsymb{a}{m}\Pochhsymb{b}{m}\Pochhsymb{a^\prime}{n}\Pochhsymb{b^\prime}{n}}{\Pochhsymb{c}{m}\Pochhsymb{c^\prime}{n} m! n!} \int_0^1 t^{m+n+\beta-1} \left( 1 - t \right)^{\gamma-\beta-1} \dd t.
\end{equation*}

Evaluating the beta integral and using Pochhammer symbols, we obtain
\begin{equation*}
A = \frac{\gammafcn(\beta) \gammafcn(\gamma-\beta)}{\gammafcn( \gamma )} \sum_{m=0}^\infty \sum_{n=0}^\infty \frac{\Pochhsymb{\beta}{m+n}}{\Pochhsymb{\gamma}{m+n}}  \frac{\Pochhsymb{a}{m}\Pochhsymb{b}{m}}{\Pochhsymb{c}{m}} \frac{\Pochhsymb{a^\prime}{n}\Pochhsymb{b^\prime}{n}}{\Pochhsymb{c^\prime}{n}} \frac{1^m 1^n}{m! n!}.
\end{equation*}
The double series above represents a Kamp{\'e} de F{\'e}riet function evaluated at unity arguments. The result follows using \eqref{eq:Kampe.de.Feriet.function}.
\end{proof}

\section{Limits}
\label{appendix:limits}

Because of \eqref{eq:repr.hypergeometric.1} limit and integration can be interchanged in the first integral in \eqref{eq:int.representation}.

For the second integral in \eqref{eq:int.representation} we observe that the function $\widetilde{G}( x; \eps )$ has the following Taylor expansion with integral remainder
\begin{equation*}
\widetilde{G}( x; \eps ) = 1 + \frac{\partial \widetilde{G}( x; \eps )}{\partial \eps} \Big|_{\eps = 0} \, \eps + \eps^2 \int_0^1 \frac{\partial^2 \widetilde{G}( x; \eps )}{\partial \eps^2} \Big|_{\eps \mapsto \eps v} \left( 1 - v \right) \dd v,
\end{equation*}
where
\begin{equation*}
\frac{\partial^2 \widetilde{G}( x; \eps )}{\partial \eps^2} = \left( \log x \right)^2 + \left( \log x \right) \widetilde{g}_1( \eps ) + \widetilde{g}_2( \eps )
\end{equation*}
and the functions $\widetilde{g}_1( \eps )$ and $\widetilde{g}_2( \eps )$ are continuous on $[0,1]$ and depend only on $L$. This shows that in the second integral in \eqref{eq:int.representation} logarithmic singularities appear; that is, apart from the term with $\frac{\partial \widetilde{G}( x; \eps )}{\partial \eps} \big|_{\eps = 0}$ (that contributes to the desired limit) the remainder is
\begin{equation*}
\eps \left[ \frac{1}{2} \int_{\mathbb{S}^d} \frac{- 2^{1-2\beta}}{2} \frac{\Pochhsymb{1-\beta}{L}}{L!} \frac{\left( \PT{x} \cdot \PT{z}-\PT{y} \cdot \PT{z} \right)^{2L}}{\left( 2 - \PT{x} \cdot \PT{z}-\PT{y} \cdot \PT{z} \right)^{2L + 1 - 2\beta}} \left( \log\left( \frac{\PT{x} \cdot \PT{z}-\PT{y} \cdot \PT{z}}{2 - \PT{x} \cdot \PT{z} - \PT{y} \cdot \PT{z}} \right)^2 \right)^2 \dd \sigma_d( \PT{z} ) + \cdots \right]
\end{equation*}
Since $1 - \PT{x} \cdot \PT{z} - \PT{y} \cdot \PT{z} \geq c > 0$ for $\PT{y} \neq \PT{x}$ and $\PT{y} \neq - \PT{x}$ and $2\beta > 2L + 1$, the integrals above can be bounded independently of $\eps$. Thus, the remainder above vanishes as $\eps \to 0$. Hence, we can interchange the limit in the second integral in \eqref{eq:int.representation}.

\section*{Acknowledgements}

The authors are grateful to Michael Gnewuch for a fruitful discussion which lead to the proof in Appendix A.

\bibliographystyle{abbrv}
\bibliography{bibliography}
\end{document}